\documentclass[12pt,reqno,a4paper]{article}
\usepackage{mathrsfs}
\usepackage{amsmath}
\usepackage{amsthm}
\usepackage{amsfonts}
\usepackage{amssymb}
\usepackage[dvips]{graphicx}
\usepackage[dvips]{color}
\usepackage{latexsym}
\usepackage{enumerate}
\usepackage{xspace}
\usepackage{bm}
\newtheorem{thm}{Theorem}
\newtheorem{theorem}[thm]{Theorem}

\newtheorem{lemma}[thm]{Lemma}
\newtheorem{proposition}[thm]{Proposition}
\newtheorem{definition}[thm]{Definition}

\begin{document}
\title{\bf Sparse Representer Theorems for Learning in Reproducing Kernel Banach Spaces}
\author{Rui Wang\thanks{School of Mathematics, Jilin University, Changchun 130012, P. R. China. E-mail address: {\it rwang11@jlu.edu.cn}. }, \quad  \ Yuesheng Xu\thanks{
Department of Mathematics and Statistics, Old Dominion University, Norfolk, VA 23529, USA. This author is also a Professor Emeritus of Mathematics, Syracuse University, Syracuse, NY 13244, USA. E-mail address: {\it y1xu@odu.edu.} All correspondence should be sent to this author.}\quad and \ Mingsong Yan \thanks{ Department of Mathematics and Statistics, Old Dominion University, Norfolk, VA 23529, USA.  E-mail address: {\it myan007@odu.edu}.} }

\date{}

\maketitle{}

\begin{abstract}
Sparsity of a learning solution is a desirable feature in machine learning. Certain reproducing kernel Banach spaces (RKBSs) are appropriate hypothesis spaces for sparse learning methods. The goal of this paper is to understand what kind of RKBSs can promote sparsity for learning solutions. 
We consider two typical learning models in an RKBS: the minimum norm interpolation (MNI) problem and the regularization problem. 
We first establish an explicit representer theorem for solutions of these problems, which represents the extreme points of the solution set by a linear combination of the extreme points of the subdifferential set, of the norm function, which is \textit{data-dependent}. 
We then propose sufficient conditions on the RKBS that can transform the explicit representation of the solutions to a sparse kernel representation having fewer terms than the number of the observed data. Under the proposed sufficient conditions, we investigate the role of the regularization parameter on sparsity of the regularized solutions. We further show that two specific RKBSs, the sequence space $\ell_1(\mathbb{N})$ and the measure space, can have sparse representer theorems for both MNI and regularization models.
\end{abstract}

    
\section{Introduction} 
The goal of this paper is to study a class of RKBSs that can promote sparsity for learning solutions in the spaces. In order to alleviate the computational burden brought by big data, developing sparse learning methods is the future of machine learning (\cite{hoefler2021future}). Reproducing kernel Banach spaces (RKBSs), spaces of functions on which the point-evaluation functionals are continuous, were introduced in \cite{zhang2009reproducing} as potential appropriate hypothesis spaces for sparse learning methods. Some of these spaces have sparsity promoting norms which lead to sparse representations for learning solutions under suitable bases. Following \cite{xu2022sparse}, we are interested in a class of RKBSs, each of which has an adjoint RKBS. Such an RKBS provides a reproducing kernel for representing not only the point evaluation functionals but also learning solutions in the spaces.

We first clarify the notion of sparsity. Intuitively, a vector or a sequence is said to be sparse if 
most of its components are zero. We say that a function in an RKBS has a sparse representation under the kernel sessions (a kernel with one of its two variables evaluated at given points) if the coefficient vector of the representation is sparse. It is well-known from the celebrated representer theorem 
(\cite{de1966splines,Argyriou2009,Cox1990,Kimeldorf1970,scholkopf2001generalized}) 
that when we learn a target function in a reproducing kernel Hilbert space (RKHS) from $n$ function values, a solution of the regularization problem is a linear combination of the $n$ kernel sessions. Often large amount of data are used to learn a target function and the learning solution is used in prediction or other decision-making procedures repeatedly. As a result, a dense learning solution will lead to large computational costs. As we have known, RKHSs do not lead to sparse learning solutions, see for example, \cite{xu2022sparse}. We then appeal to RKBSs as hypothesis spaces for learning methods and hope that some of them can offer sparse representations for their learning solutions under the kernel sessions, which have terms significantly fewer than the number of the given data points. 

We consider two typical learning models in an RKBS: the
minimum norm interpolation (MNI) problem and the regularization problem. Representer theorems for the solutions of these two models in RKBSs have received considerable attention in the literature (\cite{huang2021generalized, unser2021unifying, unser2016representer, 
wang2021representer, 
xu2019generalized, zhang2009reproducing, zhang2012regularized}). 
In particular, a systematic study of the representer theorems for a solution of
the MNI problem and the regularization problem in a Banach space was conducted by \cite{wang2021representer}. The resulting representer theorem stated that the
solutions lie in a subdifferential set of the norm function evaluated at a finite linear combination of given functionals. %
On the other hand, an explicit representer theorem for a variational problem in a Banach space was proved in \cite{boyer2019representer, bredies2020sparsity} in which the extreme points of the solution set of the variational problem were expressed as a finite linear combination of extreme points of the unit ball in the Banach space. The representer theorem of this kind is {\it data-independent}, as the searching area for extreme points claimed to express the solution is always the unit ball no matter what the given data are.

The road-map for establishing the sparse representer theorem for the solutions of the MNI problem and the regularization problem in an RKBS may be described as follows. By combining the advantages of the two representer theorems in \cite{wang2021representer} and \cite{boyer2019representer}, we first put forward an explicit solution representation for the MNI problem in a general Banach space, which is assumed to have a pre-dual space. The new explicit representer theorem allows us to represent the extreme points of the solution set as a linear combination of the extreme points of a subdifferential set of a pre-dual norm, evaluated at a linear combination of the functionals determined by given data. As a result, unlike the representer theorems presented in \cite{boyer2019representer,bredies2020sparsity}, which are data-independent, the new representer theorem is \textit{data-dependent}. Moreover, we prove that the extreme point set of the subdifferential set is a subset of the extreme point set of the unit ball and its cardinality is finite for certain specific Banach spaces, while the extreme point set of the unit ball includes infinitely many elements.
By making use of the explicit representer theorem, we propose a sufficient condition on an RKBS so that the solution of the MNI problem in the RKBS has a sparse kernel representation. The establishment of a sparse kernel representation for the learning solution requires addressing two issues. The first one is to represent the learning solution by the kernel sessions. Observing from the explicit representer theorem, it is intuitive to require that the elements in the subdifferential set, the building blocks of the solution set, coincide with the kernel sessions. This is the first assumption imposed on the RKBS. The second issue is to ensure that the kernel representation has fewer terms than the number of the given data points. To this end, we impose an additional assumption on the RKBS by requiring its norm to be equivalent to the $\ell_1$ norm which is well-known to promote sparsity. Under these two assumptions, we install kernel representations for the extreme points of the solution set of the MNI problem. The number of the kernel sessions emerging in the representation is bounded above by the rank of a matrix determined by the observed data. Under a mild condition, the rank of the matrix may be less than the number of the observed data. This leads to a sparse representer theorem for the solutions of the MNI problem in the RKBS. We then convert the resulting sparse representer theorem for the MNI problem to that for the regularization problem through a relation between the solutions of the two problems pointed out in \cite{micchelli2004function, wang2021representer}. Unlike the MNI problem, the regularization problem involves a regularization parameter. Under the assumptions on the RKBS, we reveal that the regularization parameter can further promote the sparsity level of the solution. Moreover, we show that the sequence space $\ell_1(\mathbb{N})$ and the space of functions constructed by the measure space satisfy the imposed assumptions. In this way, the sparse representer theorems for the MNI problem and the regularization problem in these two RKBSs are established,  showing that they indeed can promote sparsity in kernel representations for learning functions in these two spaces.  

Banach spaces were recently employed to understand neural networks and to promote sparsity.
A sparse technique was successfully applied in \cite{rosset2007l1} to learning with regularization for a feature space of infinite (possibly non-countable) dimension. Neural networks of a single hidden layer with infinitely many neurons as functions by integral
representation with a variational norm were studied in \cite{bach2017breaking}.
Based on the variational framework of $\mathrm{L}$-splines developed in \cite{unser2017splines} and the representer theorem established in \cite{bredies2020sparsity}, \cite{parhi2021banach} obtained a representer theorem expressing neural networks of a single hidden layer as solutions of a variation problem with the TV regularization in the Radon domain. The representer theorem for single-output neural networks of one hidden layer was extended in \cite{shenouda2023vector} to multi-output networks by considering vector-valued variation spaces.  RKBSs were employed in \cite{bartolucci2023understanding} to study neural networks with one hidden layer. %
It was shown in \cite{spek2023duality} that the Barron spaces were a class of integral RKBSs and their dual spaces as RKBSs were studied.  Results developed in \cite{rosset2007l1} were found useful in understanding neural networks in RKBSs. 

We organize this paper in six sections and two appendices. In Section 2, we review the framework of RKBSs and describe the MNI and regularization problems to be considered in this paper. Also, we define precisely the notion of sparsity of a learning solution in an RKBS. Moreover, we present several new observations for RKBSs, including the closure and weak* closure of the space of point evaluation functionals and the unique reproducing kernel of a RKBS whose $\delta$-dual space is isometrically
isomorphic to a Banach space of functions.
In Section 3, we first establish a representer theorem for the solutions of the MNI problem in a general Banach space that is assumed to have a pre-dual space. We then impose two assumptions on the RKBSs, which ensure that the spaces have the sparsity promoting property. Under the assumptions, we convert a representation of the solutions to a sparse kernel representation. In Section 4, we translate the sparse representer theorem established in Section 3 for the MNI problem to the regularization problem via a connection between the solutions of these two problems. We also study the effect of the regularization parameter on the sparsity of the regularized solutions and obtain choices of the regularization parameter for sparse solutions. In Section 5, we specialize the sparse representer theorem to the sequence space $\ell_1(\mathbb{N})$. By comparing this space with the sequence spaces $\ell_p(\mathbb{N})$ for $1<p<+\infty$, we show that $\ell_1(\mathbb{N})$ can promote sparsity of learning solutions but spaces $\ell_p(\mathbb{N})$  have no such a feature. Section 6 concerns a specific RKBS constructed by the measure space. We establish the sparse representer theorems for the solutions of the MNI problem and the regularization problem in the space by verifying that the RKBS satisfies the imposed assumptions. We include in Appendix A a complete proof of the explicit representer theorem of the MNI problem in a general Banach space established in Section 3. In Appendix B, we characterize the dual problem of the MNI problem in order to acquire the dual element emerging in the representer theorem established in this paper.

\section{Learning in RKBSs}
Motivated by developing sparse learning algorithms, RKBSs have been proposed as appropriate hypothesis spaces for learning an objective function from its values. Since the introduction of the notion of RKBSs, theory and applications of these function spaces have attracted much research interest (\cite{bartolucci2023understanding, fasshauer2015solving, lin2021multi, lin2022on, song2013reproducing, spek2023duality, xu2019generalized, zhang2009reproducing,  zhang2012regularized}). In this section, we recall the notion of RKBSs. We reveal the important role of the family of the point evaluation functionals in the dual space of an RKBS and identify a unique reproducing kernel for the RKBS as closed-form function representations for the point evaluation functionals. We also describe the MNI problem and the regularization problem in an RKBS and introduce the notion of sparse
kernel representations for solutions of such learning problems. 

We start with recalling the notion of RKBSs. A Banach space $\mathcal{B}$ is called a space of functions on a prescribed set $X$ if $\mathcal{B}$ is composed of functions defined on $X$ and for each $f\in\mathcal{B}$, $\|f\|_{\mathcal{B}}=0$ implies that $f(x) = 0$ for all $x\in X$. The notion of RKBSs was originally introduced in  \cite{zhang2009reproducing}, to ensure the stability of sampling function values from functions in the hypothesis space. 

\begin{definition}\label{RKBS}
A Banach space $\mathcal{B}$ of functions on a prescribed set $X$ is called an RKBS if all the point evaluation functionals $\delta_x$, $x\in X$, are continuous on $\mathcal{B}$, that is, for each $x\in X$, there exists a constant $C_x>0$ such that 
\[
|\delta_x(f)|\leq C_x\|f\|_\mathcal{B}, \quad\text{ for all }f\in\mathcal{B}.
\]
\end{definition}

This definition appeared originally in \cite{zhang2009reproducing} in a somewhat restricted version and its current form is adopted from \cite{xu2022sparse}. We let
\begin{equation}\label{Delta}
\Delta:=\mathrm{span}\{\delta_x:x\in X\}.
\end{equation}
It is essential to understand the closures of $\Delta$ under various types of topology. 

We first review necessary notions in Banach spaces. For a Banach space $\mathcal{B}$ with a norm $\|\cdot\|_{\mathcal{B}}$, we denote by $\mathcal{B}^*$ the dual space of $\mathcal{B}$, which is composed of all bounded linear functionals on $\mathcal{B}$ endowed with the norm 
$$
\|\nu\|_{\mathcal{B}^*}:=\sup_{\|f\|_{\mathcal{B}}\leq1}|\nu(f)|,\ \mbox{for all}\ \nu\in\mathcal{B}^*.
$$
The dual bilinear form $\langle\cdot,\cdot\rangle_{\mathcal{B}}$ on $\mathcal{B}^*\times\mathcal{B}$ is defined by $\langle\nu,f\rangle_{\mathcal{B}}:=\nu(f)$ for all $\nu\in\mathcal{B}^*$ and all $f\in\mathcal{B}$. The weak${}^*$ topology of the dual space $\mathcal{B}^{*}$ is the smallest topology for $\mathcal{B}^{*}$ such that, for each $f \in \mathcal{B}$, the linear functional $\nu \rightarrow\langle\nu, f\rangle_{\mathcal{B}}$ on $\mathcal{B}^{*}$ is continuous with respect to the topology. A topological property that holds with respect to the weak${}^*$ topology is said to hold weakly${}^*$. For example, a sequence $\nu_{n}, n \in \mathbb{N}$, in $\mathcal{B}^{*}$ is said to converge weakly${}^*$ to $\nu \in \mathcal{B}^{*}$ if 
$\lim _{n \rightarrow+\infty}\left\langle\nu_{n}, f\right\rangle_{\mathcal{B}}=\langle\nu, f\rangle_{\mathcal{B}},\ \mbox{for all}\ f\in\mathcal{B}.$
Let $\mathcal{M}$ and
$\mathcal{M}'$ be subsets of $\mathcal{B}$ and $\mathcal{B}^*$, respectively. The annihilator, in $\mathcal{B}^*$, of $\mathcal{M}$ is defined by 
$$
\mathcal{M}^\perp:=\{\nu\in\mathcal{B}^*:\langle\nu,f\rangle_{\mathcal{B}}=0, \ \text{for all }f\in\mathcal{M}\},
$$
and the annihilator, in $\mathcal{B}$, of $\mathcal{M}'$ is defined by
$$
^{\perp}\mathcal{M}':=\{f\in\mathcal{B}: \langle\nu,f\rangle_{\mathcal{B}}=0,
\ \mbox{for all}\  \nu\in\mathcal{M}'\}.
$$ 
We denote by $\overline{\mathcal{M}'}$ and $\overline{\mathcal{M}'}^{w^*}$ the closure of $\mathcal{M}'$ in the norm topology and the weak$^*$ topology of $\mathcal{B}^*$, respectively.

We now turn to characterizing the weak$^*$ density of the linear span $\Delta$ %
in $\mathcal{B}^*$. We observe from Definition \ref{RKBS} that if $\mathcal{B}$ is an RKBS on $X$, then there holds $\delta_x\in\mathcal{B}^*$ for all $x\in X$. That is, $\Delta\subseteq\mathcal{B}^*$ . It is known that if $\mathcal{B}$ is an RKHS, then there holds  $\overline{\Delta}=\mathcal{B}^*$ due to the Riesz representation theorem of the Hilbert space, where the closure is in the sense of the norm $\|\cdot\|_{\mathcal{B}}$ of the Hilbert space. This result cannot be extended to a general RKBS $\mathcal{B}$, because of the lack of the representation theorem in a general Banach space. However, we can show that the linear span $\Delta$ is weakly${}^*$ dense in $\mathcal{B}^*$. We present this result in the next proposition.

\begin{proposition}\label{prop: In RKBS, delta x in weak* dense}
    If $\mathcal{B}$ is an RKBS on $X$, then 
    $\overline{\Delta}^{w^*}=\mathcal{B}^*$.
\end{proposition}
\begin{proof}
By Definition \ref{RKBS}, we clearly see that $\Delta\subseteq\mathcal{B}^*$. Note that $f\in$ $ ^{\perp}\Delta$ if and only if $\delta_x(f)=0,$ for all $x\in X$. That is, $f(x)=0,$ for all $x\in X$, or $f=0$. Therefore, we obtain that $^{\perp}\Delta=\{0\}$, which further leads to $(^{\perp}\Delta)^{\perp}=\mathcal{B}^*$. According to Proposition 2.6.6 of \cite{megginson2012introduction}, there holds that $(^{\perp}\Delta)^{\perp}=\overline{\Delta}^{w^*}$. Combining the above two equations, we conclude that $\overline{\Delta}^{w^*}=\mathcal{B}^*$.
\end{proof}

When an RKBS $\mathcal{B}$ is assumed to have a pre-dual space $\mathcal{B}_*$ satisfying $\Delta\subseteq\mathcal{B}_*$, we can identify $\overline{\Delta}$ with the pre-dual $\mathcal{B}_*$, where the closure $\overline{\Delta}$ is taken in the norm of  $\mathcal{B}^*$.  We say that the Banach space $\mathcal{B}$ has a pre-dual space if there exists a Banach space $\mathcal{B}_*$ such that $(\mathcal{B}_*)^*=\mathcal{B}$ and we call the space  $\mathcal{B}_*$ a pre-dual space of $\mathcal{B}$. The existence of a pre-dual space $\mathcal{B}_*$ makes it valid for $\mathcal{B}$ to be equipped with weak${}^*$ topology. Since the pre-dual space $\mathcal{B}_*$ can be isometrically embedded into $\mathcal{B}^*$, any element in $\mathcal{B}_*$ can be viewed as a bounded linear functional on $\mathcal{B}$. Namely,   $\mathcal{B}_*\subset\mathcal{B}^*$ and there holds 
\begin{equation}\label{natural-map-predual}
\langle\nu,f\rangle_{\mathcal{B}}=\langle f,\nu\rangle_{\mathcal{B}_*},
\ \mbox{for all} \ f\in\mathcal{B} \ \mbox{and all} \ \nu\in\mathcal{B}_*.
\end{equation}

\begin{proposition}\label{prop: In RKBS, delta x in dense in predual}
Suppose that an RKBS $\mathcal{B}$ on $X$ has a pre-dual space $\mathcal{B}_*$. If $\Delta\subseteq\mathcal{B}_*$, then $\overline{\Delta}=\mathcal{B}_*$.
\end{proposition}

\begin{proof}
Suppose that $f\in\mathcal{B}$ satisfies $\langle f,\delta_{x}\rangle_{\mathcal{B}_*}=0$ for all $x\in X$. By equation \eqref{natural-map-predual}, we have that $\langle \delta_{x}, f\rangle_{\mathcal{B}}=0$ for all $x\in X$. This leads to $f(x)=0$ for all $x\in X$ and thus $f=0$. Due to the arbitrariness of $f\in\mathcal{B}$, we get the desired density result.  
\end{proof}

A reflexive RKBS $\mathcal{B}$ always takes the dual space $\mathcal{B}^*$
as a pre-dual space $\mathcal{B}_*$. As a direct consequence of Proposition \ref{prop: In RKBS, delta x in dense in predual}, we obtain that  $\overline{\Delta}=\mathcal{B}^*$ for a reflexive RKBS $\mathcal{B}$. 

It is known that each RKHS enjoys a unique reproducing kernel, which provides closed-form function representations for all point evaluation functionals on the RKHS. The existence of the reproducing kernel lies in the well-known Riesz representation theorem, which states that the dual space of a Hilbert space is isometrically isomorphic to itself. However, in general, the dual space of a Banach space is not isometrically isomorphic to itself. To ensure the existence of the reproducing kernel for an RKBS, we 
need to impose additional assumptions. Various hypotheses have been imposed on RKBSs (\cite{lin2022on,xu2022sparse,xu2019generalized,zhang2009reproducing}) in the literature. A hypothesis, which better captures the essence of reproducing kernels, was described in \cite{xu2022sparse}. Following \cite{xu2022sparse}, we call $\overline{\Delta}$ the $\delta$-dual space of $\mathcal{B}$ and denote it by $\mathcal{B}'$. Note that $\mathcal{B}'$ is the smallest Banach space that
contains all the point evaluation functionals on $\mathcal{B}$. We assume that the $\delta$-dual space $\mathcal{B}'$ is isometrically isomorphic to a Banach space $\mathcal{F}$ of functions on a set $X'$. In the rest of this paper, we will not distinguish $\mathcal{B}'$ from $\mathcal{F}$. We now identify a unique reproducing kernel with each RKBS satisfying the hypothesis.

\begin{proposition}\label{reproducing kernels for RKBS}
    Suppose that $\mathcal{B}$ is an RKBS on $X$ and its $\delta$-dual space $\mathcal{B}'$ is isometrically isomorphic to a Banach space of functions on $X'$. Then there exists a unique function $K:X\times X'\to\mathbb{R}$ such that the following statements hold. 
    
    (1) For each $x\in X$, $K(x,\cdot)\in \mathcal{B}'$ and
    \begin{equation}\label{def: reproducing property}
        f(x)=\langle K(x,\cdot),f\rangle_{\mathcal{B}},\ \mbox{for all}\ f\in\mathcal{B}.
    \end{equation}

    (2) The linear span 
    $K(X):=\mathrm{span}\{K(x,\cdot):x\in X\}$ is dense in $\mathcal{B}'$.
\end{proposition}
\begin{proof}
We first prove statement (1). For each $x\in X$, since $\delta_x$ is a continuous linear functional on $\mathcal{B}$, there exists $k_x\in\mathcal{B}'$ such
that $f(x)=\langle k_x,f\rangle_{\mathcal{B}}, \ \mbox{for all}\ f\in\mathcal{B}.$
By defining a function $K:X\times X'\to\mathbb{R}$ as
$K(x,x'):=k_x(x'),\ x\in X, \ x'\in X'$,
we obtain that $K(x,\cdot)\in \mathcal{B}'$ for all $x\in X$. Thus, we observe that equation \eqref{def: reproducing property} holds. 

It suffices to verify that the function $K$ on $X\times X'$ satisfying the above properties is unique. Assume that there exists another $\widetilde{K}:X\times X'\rightarrow \mathbb{R}$ such that $\widetilde{K}(x,\cdot)\in \mathcal{B}'$ for all $x\in X$, and $f(x)=\langle \widetilde{K}(x,\cdot),f\rangle_{\mathcal{B}}$, for all $f\in\mathcal{B}$ and all $x\in X$. It follows from the above equation and equation \eqref{def: reproducing property} that 
$$
\langle K(x,\cdot)-\widetilde{K}(x,\cdot),f\rangle_{\mathcal{B}}=0, \ \ \mbox{for all}\ \ f\in\mathcal{B}\ \ \mbox{and all}\ \  x\in X.
$$
That is, for all $x\in X$, $K(x,\cdot)-\widetilde{K}(x,\cdot)=0$. Noting that $\mathcal{B}'$ is isometrically isomorphic to a Banach space of functions on $X'$, we conclude that $K(x,x')-\widetilde{K}(x,x')=0$ for all $x\in X$, $x'\in X'$, which is equivalent to $K=\widetilde{K}$.

We next show the density stated in (2). It follows from statement (1) that the linear span $K(X)$ is isometrically
isomorphic to $\Delta$ defined by \eqref{Delta}. Hence, the closure $\overline{K(X)}$ in the norm topology is isometrically
isomorphic to the closure $\overline{\Delta}$ in the norm topology, which together with $\overline{\Delta}=\mathcal{B}'$ leads to  $\overline{K(X)}=\mathcal{B}'$.  
\end{proof}

We call the function $K:X\times X'\rightarrow\mathbb{R}$, satisfying $K(x,\cdot)\in \mathcal{B}'$ for all $x\in X$, and equation \eqref{def: reproducing property}, the reproducing kernel for the RKBS $\mathcal{B}$. Moreover, equation \eqref{def: reproducing property} is called the reproducing property. 

Motivated by representing the solutions of learning problems in an RKBS via its reproducing kernel, we introduce the notion of the adjoint RKBS.  %
Specifically, if in addition $\mathcal{B}'$ is an RKBS on $X'$, $K(\cdot,x')\in\mathcal{B}$ for all $x'\in X'$, and 
\begin{equation}\label{reproducing property in B'}
g(x')=\langle g,K(\cdot, x')\rangle_{\mathcal{B}}, \text{ for all }g\in\mathcal{B}'\text{ and all }x'\in X',
\end{equation}
we call $\mathcal{B}'$ an adjoint RKBS of $\mathcal{B}$ and call $\mathcal{B}$, $\mathcal{B}'$ a pair of RKBSs. We introduce the linear span of the point evaluation functions on $\mathcal{B}'$ by 
\begin{equation}\label{Delta'}
\Delta':=\mathrm{span}\{\delta_{x'}:x'\in X'\}.
\end{equation}
Observing from equation \eqref{reproducing property in B'}, the linear span  $K(X'):=\mathrm{span}\{K(\cdot,x'):x'\in X'\}$, as a subset of $\mathcal{B}$, is isometrically
isomorphic to $\Delta'$ defined by \eqref{Delta'}. Moreover, the next result shows that if $\mathcal{B}'$ is a pre-dual space of $\mathcal{B}$, then the function space $K(X')$ is large enough to fill in $\mathcal{B}$ in the sense that any function $f$ in $\mathcal{B}$ could be approximated arbitrarily well by elements in $K(X')$ with respect to weak$^*$ topology.

\begin{proposition}\label{prop: nabla is weak dense in B}
Suppose that $\mathcal{B}$ is an RKBS on $X$, the $\delta$-dual space $\mathcal{B}'$ is an adjoint RKBS on $X'$ of $\mathcal{B}$ and $K$ is the reproducing kernel. If $\mathcal{B}'$ is a pre-dual space of $\mathcal{B}$, then there holds $\overline{K(X')}^{w^*}=\mathcal{B}$.
\end{proposition}
\begin{proof}
Since $\mathcal{B}'$ is an RKBS on $X'$, Proposition \ref{prop: In RKBS, delta x in weak* dense} with $\mathcal{B}$ being replaced by $\mathcal{B}'$ ensures that $\overline{\Delta'}^{w^*}=(\mathcal{B}')^*$. This together with the assumption that $(\mathcal{B}')^*=\mathcal{B}$ leads to $\overline{\Delta'}^{w^*}=\mathcal{B}$.
Note that the linear span  $K(X')$ is isometrically
isomorphic to $\Delta'$. Hence, we get the desired  weak$^*$ density of $K(X')$ in $\mathcal{B}$.

\end{proof}

We also characterize the RKBS in terms of the feature representation.%

\begin{proposition}\label{prop: feature map}
A Banach space $\mathcal{B}$ of functions on $X$ is an RKBS if and only if there exist a Banach space $W$ and a map $\Phi:X\to W^*$ satisfying \begin{equation}\label{WeakStarClosure}
    \overline{\mathrm{span}\{\Phi(x):x\in X\}}^{w^*}=W^*,
\end{equation}
such that
\begin{equation}\label{eq: feature characterization of RKBS}
    \mathcal{B}=\left\{\langle\Phi(\cdot),u\rangle_W: u\in W\right\},
\end{equation}   
equipped with  
\begin{equation}\label{eq: feature characterization of norm}
    \|\langle\Phi(\cdot),u\rangle_W\|_{\mathcal{B}}=\|u\|_W, \ \ \mbox{for all}\ \ u\in W.
\end{equation}
\end{proposition}
\begin{proof}
Suppose that $\mathcal{B}$ is an RKBS on $X$. We choose the Banach space $W:=\mathcal{B}$ and the map $\Phi:X\rightarrow\mathcal{B}^*$ defined by $\Phi(x):=\delta_x$, $x\in X$. Hence, we have that
$\mathrm{span}\{\Phi(x):x\in X\}=\Delta$. It follows from Proposition \ref{prop: In RKBS, delta x in weak* dense} that the density condition \eqref{WeakStarClosure} holds true. In addition, we can trivially represent any $f\in\mathcal{B}$ as $f=\langle\Phi(\cdot),u\rangle_W$ with $u:=f$ and hence $\|f\|_{\mathcal{B}}=\|u\|_W$. 
    
Conversely, suppose that there exist $W$ and $\Phi:X\to W^*$ such that equation \eqref{WeakStarClosure} holds true. We then define a vector space $\mathcal{B}$ by \eqref{eq: feature characterization of RKBS} and a map $\|\cdot\|_{\mathcal{B}}:\mathcal{B}\rightarrow\mathbb{R}$ by \eqref{eq: feature characterization of norm}. We point out that under the assumption that $\mathrm{span}\{\Phi(x):x\in X\}$ is weakly${}^*$ dense in $W^*$, any function in $\mathcal{B}$ has a unique representation. Indeed, if $\langle\Phi(x),u\rangle_W=\langle\Phi(x),v\rangle_W$, for all $x\in X$, then 
$\langle\Phi(x),u-v\rangle_W=0$, for all $x\in X$. This together with equation \eqref{WeakStarClosure} leads to $u=v$. As a result, the map defined by \eqref{eq: feature characterization of norm} is well-defined. It is easy to verify that $\|\cdot\|_{\mathcal{B}}$ is a norm on $\mathcal{B}$. Moreover, $\mathcal{B}$ is isometrically isomorphic to $W$ and then is a Banach space. It suffices to show that point evaluation functionals are continuous on $\mathcal{B}$. For all $x\in X$ and all $f=\langle\Phi(\cdot),u\rangle_W\in\mathcal{B}$ with $u\in W$, it holds that 
$$
|f(x)|\leq\|\Phi(x)\|_{W^*}\|u\|_{W}=\|\Phi(x)\|_{W^*}\|f\|_{\mathcal{B}}.
$$
This ensures that the point evaluation functionals are all continuous on $\mathcal{B}$. According to Definition \ref{RKBS}, $\mathcal{B}$ is an RKBS.  
\end{proof}

To close this section, we describe two learning problems in RKBSs to be considered in this paper. Learning a function from a finite number of sampled data is often formulated as
a MNI problem or a regularization problem. For each $n\in\mathbb{N}$, let $\mathbb{N}_n:=\{1,2,\ldots,n\}$. Suppose that $\mathcal{B}$ is an RKBS having a pre-dual space $\mathcal{B}_*$ and $\nu_j\in\mathcal{B}_*$, $j\in\mathbb{N}_n$, are linearly independent. Associated with these functionals, we set
\begin{equation}\label{V_span}
\mathcal{V}:=\mathrm{span}\{\nu_j:j\in\mathbb{N}_n\},
\end{equation}
and define an operator $\mathcal{L}: \mathcal{B} \rightarrow \mathbb{R}^{n}$ by  
\begin{equation}\label{operator L}
    \mathcal{L}(f):=\left[\left\langle\nu_{j}, f\right\rangle_{\mathcal{B}}: j \in \mathbb{N}_{n}\right], \text { for all } f \in \mathcal{B}.
\end{equation}
Let $\mathbf{y}:=\left[y_{j}: j \in \mathbb{N}_{n}\right] \in \mathbb{R}^{n}$ be a given vector. Learning a
target function in $\mathcal{B}$ from the given sampled data $\{(\nu_j,y_j):j\in\mathbb{N}_n\}$ consists of solving the
first kind operator equation \begin{equation}\label{operator_equation}
\mathcal{L}(f)=\mathbf{y},
\end{equation}
for $f\in\mathcal{B}$. The MNI aims at finding an element in $\mathcal{B}$, having the smallest
norm and satisfying equation \eqref{operator_equation}.
By introducing a subset of $\mathcal{B}$ as
\begin{equation}\label{hyperplane My}
    \mathcal{M}_{\mathbf{y}}:=\{f \in \mathcal{B}: \mathcal{L}(f)=\mathbf{y}\},
\end{equation}
we formulate the MNI problem with the given data $\{(\nu_j,y_j):j\in\mathbb{N}_n\}$ as  
\begin{equation}\label{MNI}
    \inf \left\{\|f\|_{\mathcal{B}}: f \in \mathcal{M}_{\mathbf{y}}\right\}.
\end{equation}
To address the ill-posedness of equation \eqref{operator_equation}, the regularization approach adds a regularization term to a data fidelity term constructed
from equation \eqref{operator_equation} such that the resulting optimization problem is much less
sensitive to disturbances. 
Specifically, the regularization problem has the form 
\begin{equation}\label{regularization problem special}
\inf \left\{\mathcal{Q}_{\mathbf{y}}(\mathcal{L}(f))+\lambda \varphi\left(\|f\|_{\mathcal{B}}\right): f \in \mathcal{B}\right\},
\end{equation}
where $\mathcal{Q}_{\mathbf{y}}: \mathbb{R}^{n} \rightarrow \mathbb{R}_{+}:=[0,+\infty)$ is a loss function, $\varphi: \mathbb{R}_{+} \rightarrow \mathbb{R}_{+}$ is a regularizer and $\lambda$ is a positive regularization parameter. We always assume $\mathcal{Q}_{\mathbf{y}}$ and $\varphi$ to be both lower semi-continuous. A function $\mathcal{T}$ mapping from a topological space $\mathcal{X}$ to $\mathbb{R}$ is said to be lower semi-continuous if
$\mathcal{T}(f)\leq\liminf_{\alpha}\mathcal{T}(f_{\alpha})$
whenever $f_{\alpha}$, $\alpha\in I,$ for some index set $I$ is a net in $\mathcal{X}$ converging to some element $f\in \mathcal{X}$. We also assume that $\varphi$ is increasing and coercive, that is, $\lim_{t\rightarrow+\infty}\varphi(t)=+\infty$. Throughout this paper, we denote by $\rm{S}(\mathbf{y})$ and $\rm{R}(\mathbf{y})$ the solution sets of the MNI problem \eqref{MNI} and the regularization problem \eqref{regularization problem special} with $\mathbf{y}\in\mathbb{R}^n$, respectively.

Proposition \ref{prop: nabla is weak dense in B} motivates us to consider representing the solutions of the MNI problem and the regularization problems in $\mathcal{B}$ by the kernel sessions $K(\cdot,x')$, $x'\in X'$. 
In this paper, a solution of a learning problem in an RKBS from $n$ data points is said to have a sparse kernel representation if there exist a nonnegative integer $m$ less than $n$ and $x_j'\in X'$, $\alpha_j\in\mathbb{R}\backslash\{0\}$, $j\in\mathbb{N}_m$, such that 
$$
f(\cdot)=\sum_{j\in\mathbb{N}_m} \alpha_j K(\cdot, x_j').
$$
We call the smallest nonnegative integer $m$ such that the above equation holds the sparsity level of $f$ under the kernel representation.

\section{Sparse Representer Theorem for MNI}\label{sec: sparse representer theorem of MNI}
The goal of this section is to establish a sparse representer theorem for solutions of the MNI problem in an RKBS. For this purpose, we first provide a representer theorem for solutions of the MNI problem in a general Banach space. We then formulate two assumptions on the RKBS and the functionals used to produce the sampled data, which together with the representer theorem transfer the MNI problem in the RKBS of infinite dimension to a finite dimensional one. We finally establish the sparse representer theorem for solutions of the MNI problem in the RKBS using the sparse representation for the solutions of the finite dimensional MNI problem.

We begin with the MNI problem \eqref{MNI} in an RKBS $\mathcal{B}$. To obtain sparse kernel representations for the solutions, we first need an explicit representer theorem. A representer theorem for the solutions of the MNI problem (\ref{MNI}) in a general Banach space having a pre-dual space was established in \cite{wang2021representer}. To describe this result, we recall the notion of the subdifferential of a convex function on a Banach space $\mathcal{B}$. A convex function $\phi:\mathcal{B}\rightarrow\mathbb{R}\cup\{+\infty\}$ is said to be subdifferentiable at $f\in\mathcal{B}$ if there exists $\nu\in\mathcal{B}^*$ such that
$$
\phi(g)-\phi(f)\geq\langle\nu,g-f\rangle_{\mathcal{B}},\ \mbox{for all}\ g\in\mathcal{B}.
$$ 
The set of all functionals in $\mathcal{B}^*$ satisfying the above inequalities is called the \textit{subdifferential} of $\phi$ at $f$ and denoted by $\partial \phi(f)$. A subset $A$ of a vector space $\mathcal{X}$ is called a convex set if $t x+(1-t)y\in A$ for all $x,y\in A$ and all $t\in[0,1]$. It can be directly verified by definition that for any $f\in\mathcal{B}$,  $\partial \phi(f)$ is a convex and weakly${}^*$ closed subset of $\mathcal{B}^*$. In particular, the subdifferential of the norm function $\|\cdot\|_\mathcal{B}$ at each $f\in\mathcal{B}\backslash\{0\}$ has the following essential property (\cite{cioranescu2012geometry})
\begin{equation}\label{subdifferential = norming functional}
  \partial \|\cdot\|_\mathcal{B}(f)=\left\{\nu\in\mathcal{B}^*:\|\nu\|_{\mathcal{B}^*}=1,\langle \nu,f\rangle_{\mathcal{B}}=\|f\|_{\mathcal{B}}\right\}.
\end{equation}
The elements in $\partial \|\cdot\|_\mathcal{B}(f)$ are also called norming functionals of $f$ as applying such functional on $f$ turns out exactly the norm of $f$. Suppose that $\mathcal{B}$ is a Banach space having a pre-dual space $\mathcal{B}_*$ and $\nu_j\in \mathcal{B}_*$,
$j\in\mathbb{N}_n$, are linearly independent. Theorem 12 in \cite{wang2021representer} states that $\hat{f} \in \mathcal{B}$ is a solution of 
the MNI problem \eqref{MNI} with $\mathbf{y}\in\mathbb{R}^n$ if and only if $\hat{f} \in \mathcal{M}_{\mathbf{y}}$ and there exists $\hat\nu\in\mathcal{V}$, such that
\begin{equation}\label{f hat in dualality mapping}
   \hat{f} \in \|\hat\nu\|_{\mathcal{B}_*} \partial\|\cdot\|_{\mathcal{B}_{*}}\left(\hat\nu\right).
\end{equation}
 This representer theorem provides a characterization for the solutions of the MNI problem \eqref{MNI} by using an inclusion relation.
We will develop an explicit representer theorem for the solution $\hat f$ of the MNI problem \eqref{MNI} based on the aforementioned result.

We recall the notion of extreme points of a closed convex subset. Let $A$ be a nonempty closed convex subset of a Hausdorff topological vector space $\mathcal{X}$. An element $z\in A$ is said to be an extreme point of $A$ if $x,y\in A$ and $tx+(1-t)y=z$ for some $t\in(0,1)$ implies that $x=y=z$. By $\mathrm{ext}(A)$ we denote the set of extreme points of $A$. It is known (\cite{megginson2012introduction}) that $z\in \mathrm{ext}(A)$ if and only if whenever $x,y\in A$ and $z=(x+y)/2$, it follows that $x=y=z$. We also need the notion of the convex hull of a subset. 
The convex hull of a subset $A$ of a vector space $\mathcal{X}$, denoted by $\mathrm{co}(A)$, is the smallest convex set that contains $A$. It is easy to show that 
$$
\operatorname{co}(A)=\left\{\sum_{j \in \mathbb{N}_{n}} t_{j} x_{j}: x_{j} \in A, t_{j} \in[0,+\infty), \sum_{j \in \mathbb{N}_{n}} t_{j}=1, j \in \mathbb{N}_{n}, n \in \mathbb{N}\right\}.
$$ 
If $\mathcal{X}$ has a topology, then the closed convex hull of $A$, denoted by $\overline{\mathrm{co}}(A)$, is the smallest closed convex set that contains $A$.
The celebrated Krein-Milman theorem (\cite{megginson2012introduction}) states that if $A$ is a nonempty compact convex subset of a Hausdorff locally convex topological vector space $\mathcal{X}$, then $A$ is the closed convex hull of its set of extreme points, that is, 
$A=\overline{\mathrm{co}}\left(\mathrm{ext}(A)\right)$.
As a direct consequence of this result, the set $\mathrm{ext}(A)$ must be nonempty. Notice that if the Banach space $\mathcal{B}$ has a pre-dual space, then it  equipped with the weak${}^*$ topology is a Hausdorff locally convex topological vector space. The solution set $\rm{S}(\mathbf{y})$, guaranteed by Lemma \ref{lemma: solution set of MNI} in Appendix A, is a nonempty, convex and weakly${}^*$ compact subset of $\mathcal{B}$. Then the Krein-Milman Theorem enables us to express $\rm{S}(\mathbf{y})$ by its extreme points as
\begin{equation}\label{Krein Milamn for S_MNI}
    \rm{S}(\mathbf{y})=\overline{\mathrm{co}}\left(\mathrm{ext}\left(\rm{S}(\mathbf{y})\right)\right),
\end{equation}
where the closed convex hull is taken under the weak${}^*$ topology. Observing from equation \eqref{Krein Milamn for S_MNI}, we will only provide closed-form representations for the extreme points of $\rm{S}(\mathbf{y})$.

With the help of the dual element $\hat\nu$ appearing in inclusion \eqref{f hat in dualality mapping},  we establish in the following proposition an explicit representer theorem for the extreme points of the solution set $\rm{S}({\mathbf{y}})$ of problem (\ref{MNI}). In this result, we consider a general Banach space which has a pre-dual space. Its complete proof is included in Appendix A. 

\begin{proposition}\label{theorem: representer for MNI}
Suppose that $\mathcal{B}$ is a Banach space having a pre-dual space $\mathcal{B}_{*}$. Let $\nu_{j} \in \mathcal{B}_{*}$, $j \in \mathbb{N}_{n}$, be linearly independent and $\mathbf{y} \in \mathbb{R}^{n}\backslash\{\mathbf{0}\}$. Suppose that  $\mathcal{V}$ and   $\mathcal{M}_{\mathbf{y}}$ are defined by \eqref{V_span} and \eqref{hyperplane My}, respectively, and $\hat\nu\in\mathcal{V}$ satisfies 
\begin{equation}\label{Non-empty-set}
(\|\hat\nu\|_{\mathcal{B}_*}\partial\|\cdot\|_{\mathcal{B}_*}(\hat\nu))\cap\mathcal{M}_{\mathbf{y}}\neq\emptyset.
\end{equation}
Then for any $\hat f\in \mathrm{ext}(\rm{S}({\mathbf{y}}))$, there exist $\gamma_j\in\mathbb{R}$, $j\in\mathbb{N}_{n}$, with  $\sum_{j\in\mathbb{N}_{n}}\gamma_j=\|\hat\nu\|_{\mathcal{B}_*}$ and $u_j\in\mathrm{ext}\left(\partial\|\cdot\|_{\mathcal{B}_*}(\hat\nu)\right)$, $j\in\mathbb{N}_{n}$, such that
    \begin{equation}\label{eq: expansion of p}
        \hat f=\sum\limits_{j\in\mathbb{N}_{n}} \gamma_j u_j.
    \end{equation}
\end{proposition}

Two remarks about this explicit representer theorem are in order. First, a representer theorem was proposed in \cite{boyer2019representer,bredies2020sparsity} based on which additional results were obtained in \cite{bartolucci2023understanding, lin2020sparse, unser2022convex}. While specializing the result in \cite{boyer2019representer} to the MNI problem \eqref{MNI} in a Banach space which has a pre-dual space, any $\hat f\in\mathrm{ext}(\rm{S}({\mathbf{y}}))$ can be represented as in \eqref{eq: expansion of p} with $\gamma_j\in\mathbb{R}$, $j\in\mathbb{N}_n$  and $u_j\in\mathrm{ext}(B_0)$, $j\in\mathbb{N}_n$, where $B_0$ denotes the closed unit ball of $\mathcal{B}$ with center in the origin. This representation does not depend on the given data that define the MNI problem \eqref{MNI}. Proposition \ref{theorem: representer for MNI} strengthens the representer theorem of \cite{boyer2019representer} by specifying the elements $u_j$, $j\in\mathbb{N}_{n}$, used to represent the extreme points of  $\rm{S}({\mathbf{y}})$, to belong to the data-dependent set $\mathrm{ext}(\partial\|\cdot\|_{\mathcal{B}_*}(\hat\nu))$, where $\hat \nu$ is some linear combination of the given functionals $\nu_j$, $j\in\mathbb{N}_n$. Moreover, as shown in Proposition \ref{prop: ext is smaller} of Appendix A, the set $\mathrm{ext}(\partial\|\cdot\|_{\mathcal{B}_*}(\hat\nu))$ is indeed smaller than the set $\mathrm{ext}(B_0)$. In Section 5, we will show that in the special case that $\mathcal{B}:=\ell_1(\mathbb{N})$, the set $\mathrm{ext}(B_0)$ includes
infinitely many elements while $\mathrm{ext}(\partial\|\cdot\|_{\mathcal{B}_*}(\hat\nu))$ has only finite elements. In a word, 
Proposition \ref{theorem: representer for MNI} provides a more precise characterization for $u_j$, $j\in\mathbb{N}_{n}$ by using the given data. Second, observing from Proposition \ref{theorem: representer for MNI}, the element $\hat\nu\in\mathcal{V}$ satisfying \eqref{Non-empty-set} plays an important role in the representation \eqref{eq: expansion of p} of the extreme points of $\rm{S}({\mathbf{y}})$. To characterize such an element, we establish in Appendix B a dual problem for the MNI problem \eqref{MNI}.  Proposition \ref{prop: solution of dual problem gives solution of original problem} in Appendix B demonstrates that the element $\hat\nu$ can be obtained by solving the associated dual problem.

The explicit representer theorem enables us to establish sparse kernel representations for solutions of the MNI problem \eqref{MNI} in certain RKBSs. In the remaining part of this section, we always assume that $\mathcal{B}$ and $\mathcal{B}'$ are a pair of RKBSs and $K:X\times X'\rightarrow \mathbb{R}$ is the reproducing kernel for $\mathcal{B}$. In addition, we suppose that $\mathcal{B}_*$ is a pre-dual space of $\mathcal{B}$ and $\nu_{j}$, $j\in \mathbb{N}_{n}$, are linearly independent elements in $\mathcal{B}_*$.  Proposition \ref{theorem: representer for MNI} ensures that any extreme point of the solution set $\rm{S}(\mathbf{y})$ of problem \eqref{MNI} with $\mathbf{y}\in \mathbb{R}^{n}\backslash\{\mathbf{0}\}$ can be expressed by the linear combination of elements in the set $\mathrm{ext}(\partial\|\cdot\|_{\mathcal{B}_*}(\hat\nu))$ for some $\hat\nu\in\mathcal{V}$. In order to promote sparse kernel representations for the solutions of the MNI problem, it is intuitive to require that the elements in $\mathrm{ext}(\partial\|\cdot\|_{\mathcal{B}_*}(\nu))$, the building blocks of the solution set, are as sparse as possible. For this purpose, we require that the RKBS $\mathcal{B}$ and the functionals $\nu_j\in\mathcal{B}_*,$ $j\in\mathbb{N}_n$, satisfy the following assumption.

(A1) For any nonzero $\nu\in\mathcal{V}$, there exists a finite subset $X_{{\nu}}'$ of $ X'$ such that
\begin{equation*}
    \mathrm{ext}(\partial\|\cdot\|_{\mathcal{B}_*}({\nu}))\subset\left\{-K(\cdot,x'), K(\cdot,x'):x'\in X_{{\nu}}'\right\}. 
\end{equation*} 

Under Assumption (A1), we can express any extreme point of the solution set $\rm{S}({\mathbf{y}})$ as a linear combination of the kernel sessions. We impose an additional assumption on the RKBS $\mathcal{B}$ by relating its norm with the well-known sparsity-promoting norm $\|\cdot\|_1$. 

(A2) There exists a positive constant $C$ such that for any $m\in\mathbb{N}$, distinct points $x_j'\in X'$, $j\in\mathbb{N}_m$, and $\bm{\alpha}=[\alpha_j:j\in\mathbb{N}_m]\in\mathbb{R}^m$, there holds 
\begin{equation*}
\left\|\sum\limits_{j\in\mathbb{N}_m}\alpha_j K(\cdot, x_j')\right\|_{\mathcal{B}}=C\|\bm{\alpha}\|_1.
\end{equation*}

We now turn to establishing sparse kernel representations for the solutions of the MNI problem \eqref{MNI} under the above assumptions. To this end, we introduce a finite dimensional MNI problem. Suppose that Assumption (A1) holds and $\hat{\nu}\in\mathcal{V}$ satisfies \eqref{Non-empty-set} with $\mathbf{y}\in\mathbb{R}^n$. We denote by $n(\hat\nu)$ the cardinality of the set $X_{\hat\nu}'$ and suppose that \begin{equation}\label{extreme_point_set}
X_{\hat\nu}':=\left\{x_j':j\in\mathbb{N}_{n(\hat{\nu})}\right\}.
\end{equation} 
By defining a matrix 
\begin{equation}\label{finite matrix V hat nu}
    \mathbf{L}_{\hat\nu}:=[\langle \nu_i, K(\cdot,x_j')\rangle_{\mathcal{B}}:\ i\in\mathbb{N}_n,\ j\in\mathbb{N}_{n(\hat{\nu})}]\in\mathbb{R}^{n\times n(\hat\nu)},
\end{equation} 
we introduce the finite dimensional MNI problem by
\begin{equation}\label{RKBS finite d MNI section 2}
    \inf \left\{\|\bm{\alpha}\|_{1}:\mathbf{L}_{\hat{\nu}}\bm{\alpha}=\mathbf{y},\ \bm{\alpha}\in\mathbb{R}^{n(\hat{\nu})}\right\},
\end{equation}
and denote by $\rm{S}_{\hat\nu}(\mathbf{y})$ the solution set of problem \eqref{RKBS finite d MNI section 2} with $\mathbf{y}\in\mathbb{R}^n$. The next result concerns the
sparsity of the elements in $\rm{S}_{\hat\nu}(\mathbf{y})$.

\begin{proposition}\label{theorem: finite dimensional l1 MNI section 2}

Suppose that $\mathbf{y}\in \mathbb{R}^{n}\backslash\{\mathbf{0}\}$, $\mathcal{V}$ and $\mathcal{M}_{\mathbf{y}}$ be defined by \eqref{V_span} and \eqref{hyperplane My}, respectively, and $\hat{\nu}\in\mathcal{V}$ satisfy \eqref{Non-empty-set}.
If Assumption (A1) holds and $\mathbf{L}_{\hat{\nu}}$ is defined by \eqref{finite matrix V hat nu}, then $\hat{\bm{\alpha}}\in\mathrm{ext}\left(\rm{S}_{\hat{\nu}}(\mathbf{y})\right)$ has at most $\mathrm{rank}(\mathbf{L}_{\hat{\nu}})$ nonzero components.
\end{proposition}
\begin{proof} 
We prove this result by employing Proposition \ref{theorem: representer for MNI}. In this case, the Banach space $\widetilde{\mathcal{B}}:=\mathbb{R}^{n(\hat{\nu})}$ endowed with the $\ell_{1}$ norm has the pre-dual space $\widetilde{\mathcal{B}}_{*}:=\mathbb{R}^{n(\hat{\nu})}$ endowed with the $\ell_{\infty}$ norm. In addition, for each $j\in\mathbb{N}_n$, the element $\widetilde{\nu}_j\in\widetilde{\mathcal{B}}_*$ is chosen as the $j$-th row of $\mathbf{L}_{\hat{\nu}}$. For $\mathbf{y}:=[y_j:j\in\mathbb{N}_n]$, we set  $\widetilde{\mathcal{M}}_{\mathbf{y}}:=\{\bm{\alpha}\in\widetilde{\mathcal{B}}: \left\langle\widetilde{\nu}_{j}, \bm{\alpha}\right\rangle_{\widetilde{\mathcal{B}}}=y_j, j\in\mathbb{N}_n\}$. If $\widetilde{\nu}_j\in\widetilde{\mathcal{B}}_*$, $j\in\mathbb{N}_n$, are linearly dependent, we may select a maximal linearly independent subset $\{\widetilde{\nu}_{n_j}:j\in\mathbb{N}_{\mathrm{dim}(\widetilde{\mathcal{V}})}\}$ with $\widetilde{\mathcal{V}}:=\mathrm{span}\{\widetilde{\nu}_j:j\in\mathbb{N}_n\}$. For $\mathbf{y}':=[y_{n_j}:j\in\mathbb{N}_{\mathrm{dim}(\widetilde{\mathcal{V}})}]\in\mathbb{R}^{\mathrm{dim}(\widetilde{\mathcal{V}})}$, we define $$
\widetilde{\mathcal{M}}_{\mathbf{y}'}:=\left\{\bm{\alpha}\in\widetilde{\mathcal{B}}: \left\langle\widetilde{\nu}_{n_j}, \bm{\alpha}\right\rangle_{\widetilde{\mathcal{B}}}=y_{n_j}, j\in\mathbb{N}_{\mathrm{dim}(\widetilde{\mathcal{V}})}\right\}.
$$
It is obvious that if $\widetilde{\mathcal{M}}_{\mathbf{y}}$ is nonempty, then $\widetilde{\mathcal{M}}_{\mathbf{y}}$ and $\widetilde{\mathcal{M}}_{\mathbf{y}'}$ are exactly the same. This allows us to consider an equivalent MNI problem $\inf \{\|\bm{\alpha}\|_{\widetilde{\mathcal{B}}}: \bm{\alpha} \in \widetilde{\mathcal{M}}_{\mathbf{y}'}\}$, in which the given functionals are linearly independent. 
 
 By choosing $\widetilde{\nu}\in\widetilde{\mathcal{V}}$ satisfying $\|\widetilde{\nu}\|_{\widetilde{\mathcal{B}}_*}\partial\|\cdot\|_{\widetilde{\mathcal{B}}_*}(\widetilde{\nu})\cap\widetilde{\mathcal{M}}_{\mathbf{y}}\neq\emptyset$, Proposition \ref{theorem: representer for MNI} ensures that any extreme point $\hat{\bm{\alpha}}$ of the solution set $\rm{S}_{\hat{\nu}}(\mathbf{y})$ can be represented as a linear combination of at most  $\mathrm{dim}(\widetilde{\mathcal{V}})$ extreme points of $\partial\|\cdot\|_{\widetilde{\mathcal{B}}_*}(\widetilde{\nu})$. Then the desired result follows from the fact that $\mathrm{dim}(\widetilde{\mathcal{V}})=\mathrm{rank}(\mathbf{L}_{\hat{\nu}})$ and each extreme points of $\partial\|\cdot\|_{\widetilde{\mathcal{B}}_*}(\widetilde{\nu})$ has just one nonzero component.     
\end{proof}

Empirical results show that the MNI problem 
with the $\ell_1$ norm can promote sparsity of a solution. Proposition \ref{theorem: finite dimensional l1 MNI section 2} provides a theoretical characterization of the sparsity of the solutions of problem \eqref{RKBS finite d MNI section 2}. In fact, the sparsity can be further characterized by the  positional relationship between the hyperplanes constructed by the given data and the unit ball of $\mathbb{R}^{n(\hat{\nu})}$ under the $\ell_1$ norm.  Such relationship can be multifarious making the comprehensive analysis rather complicated. 

 Below, we reveal the relation between the solutions of problems \eqref{MNI} and \eqref{RKBS finite d MNI section 2}.

\begin{lemma}\label{lemma: relation between MNI and finite MNI section 2}

    Suppose that $\mathbf{y}\in \mathbb{R}^{n}\backslash\{\mathbf{0}\}$, $\mathcal{V}$ and $\mathcal{M}_{\mathbf{y}}$ be defined by \eqref{V_span} and \eqref{hyperplane My}, respectively, and $\hat{\nu}\in\mathcal{V}$ satisfy \eqref{Non-empty-set}. If Assumptions (A1) and (A2) hold,   
    $X_{\hat\nu}'$ and  $\mathbf{L}_{\hat{\nu}}$ are defined by \eqref{extreme_point_set} and \eqref{finite matrix V hat nu}, respectively, then the following statements hold.
    \begin{enumerate}
        \item $\hat{f}:=\sum_{j\in\mathbb{N}_{n(\hat{\nu})}}\hat\alpha_j K(\cdot,x_j')\in\rm{S}(\mathbf{y})$ if and only if $\hat{\bm{\alpha}}:=[\hat\alpha_j:j\in\mathbb{N}_{n(\hat{\nu})}]\in\rm{S}_{\hat\nu}(\mathbf{y})$.
        
        \item If $\hat{f}:=\sum_{j\in\mathbb{N}_{n(\hat{\nu})}}\hat\alpha_j K(\cdot,x_j')\in\mathrm{ext}(\rm{S}(\mathbf{y}))$, then $\hat{\bm{\alpha}}:=[\hat\alpha_j:j\in\mathbb{N}_{n(\hat{\nu})}]\in\mathrm{ext}(\rm{S}_{\hat\nu}(\mathbf{y}))$.
    \end{enumerate}
\end{lemma}

\begin{proof}
     For any $\bm{\alpha}:=[\alpha_j:j\in\mathbb{N}_{n(\hat{\nu})}]\in\mathbb{R}^{n(\hat{\nu})}$, we set $f:=\sum_{j\in\mathbb{N}_{n(\hat{\nu})}}\alpha_j K(\cdot,x_j')$. It follows from Assumption (A2) that \begin{equation}\label{f-and-alpha-1}    \|f\|_{\mathcal{B}}=C\|\bm{\alpha}\|_1.
     \end{equation}
     By definition \eqref{finite matrix V hat nu} of the matrix $\mathbf{L}_{\hat{\nu}}$, we have that 
     $$     \mathbf{L}_{\hat\nu}\bm{\alpha}=\left[\sum_{j\in\mathbb{N}_{n(\hat{\nu})}}\alpha_j\langle \nu_i, K(\cdot,x_j')\rangle_{\mathcal{B}}:i\in\mathbb{N}_n\right].
     $$
     This together with definition \eqref{operator L} of the operator $\mathcal{L}$ leads to \begin{equation}\label{f-and-alpha-2} \mathbf{L}_{\hat\nu}\bm{\alpha}=\mathcal{L}(f).
    \end{equation} 
     
     We first prove statement 1. Suppose that $\hat{f}:=\sum_{j\in\mathbb{N}_{n(\hat{\nu})}}\hat\alpha_j K(\cdot,x_j')\in\rm{S}(\mathbf{y})$ and set $\hat{\bm{\alpha}}:=[\hat\alpha_j:j\in\mathbb{N}_{n(\hat{\nu})}]$. It follows that $\mathcal{L}(\hat{f})=\mathbf{y}$. This together with equation \eqref{f-and-alpha-2} with $\bm{\alpha}:=\hat{\bm{\alpha}}$ and $f:=\hat f$ leads to $\mathbf{L}_{\hat{\nu}}\hat{\bm{\alpha}}=\mathbf{y}$. To verify $\hat{\bm{\alpha}}\in\rm{S}_{\hat\nu}(\mathbf{y})$, it suffices to show that $\left\|\hat{\bm{\alpha}}\right\|_1\leq\|\bm\alpha\|_1$ for any $\bm{\alpha}\in\mathbb{R}^{n(\hat{\nu})}$ satisfying $\mathbf{L}_{\hat\nu}\bm{\alpha}=\mathbf{y}$. Suppose that $\bm{\alpha}:=[\alpha_j:j\in\mathbb{N}_{n(\hat{\nu})}]\in\mathbb{R}^{n(\hat{\nu})}$ satisfying $\mathbf{L}_{\hat\nu}\bm{\alpha}=\mathbf{y}$. By equation \eqref{f-and-alpha-2}, the function $f:=\sum_{j\in\mathbb{N}_{n(\hat{\nu})}}\alpha_j K(\cdot,x_j')$ satisfies $\mathcal{L}(f)=\mathbf{y}.$ That is, $f\in\mathcal{M}_{\mathbf{y}}$. This combined with $\hat{f}\in\rm{S}(\mathbf{y})$ leads to $\|\hat{f}\|_{\mathcal{B}}\leq\|f\|_{\mathcal{B}}$. Substituting equation \eqref{f-and-alpha-1} with the pair $f,\bm\alpha$ and the same equation with the pair $\hat{f}, \hat{\bm{\alpha}}$ into the above inequality, we conclude that $\left\|\hat{\bm{\alpha}}\right\|_1\leq\|\bm\alpha\|_1$. Conversely, suppose that $\hat{\bm{\alpha}}:=[\hat\alpha_j:j\in\mathbb{N}_{n(\hat{\nu})}]\in\rm{S}_{\hat\nu}(\mathbf{y})$ and set $\hat{f}:=\sum_{j\in\mathbb{N}_{n(\hat{\nu})}}\hat\alpha_j K(\cdot,x_j')$. It follows from $\mathbf{L}_{\hat{\nu}}\hat{\bm{\alpha}}=\mathbf{y}$ and equation \eqref{f-and-alpha-2} with $\bm{\alpha}:=\hat{\bm{\alpha}}$, $f:=\hat f$ that $\mathcal{L}(\hat f)=\mathbf{y}$. That is $\hat f\in\mathcal{M}_{\mathbf{y}}$.  We choose $f\in\mathrm{ext}(\rm{S}(\mathbf{y}))$ and proceed to show that $\|\hat{f}\|_{\mathcal{B}}\leq\|f\|_{\mathcal{B}}$. Combining Proposition \ref{theorem: representer for MNI} with Assumption (A1), we represent $f$ as $f=\sum_{j\in\mathbb{N}_{n(\hat{\nu})}}\alpha_j K(\cdot,x_j')$ for some 
   $\bm{\alpha}:=[\alpha_j:j\in\mathbb{N}_{n(\hat{\nu})}]\in\mathbb{R}^{n(\hat{\nu})}$.
   Equation \eqref{f-and-alpha-2} ensures that  $\mathbf{L}_{\hat{\nu}}\bm{\alpha}=\mathbf{y}$, which together with  $\hat{\bm{\alpha}}\in\rm{S}_{\hat\nu}(\mathbf{y})$ leads to $\left\|\hat{\bm{\alpha}}\right\|_1\leq\|\bm\alpha\|_1$. Again substituting  equation \eqref{f-and-alpha-1} with the pair $f,\bm\alpha$ and the same equation with the pair $\hat{f}, \hat{\bm{\alpha}}$ into the above inequality, we obtain that  $\|\hat{f}\|_{\mathcal{B}}\leq\|f\|_{\mathcal{B}}$. By noting that $f\in\rm{S}(\mathbf{y})$, we get the conclusion that  $\hat{f}\in\rm{S}(\mathbf{y})$.
  
   We next verify statement 2. Suppose that $\hat f:=\sum_{j\in\mathbb{N}_{n(\hat{\nu})}}\hat\alpha_j K(\cdot, x_j')\in\mathrm{ext}(\rm{S}(\mathbf{y}))$. Statement 1 of this lemma ensures that $\hat{\bm{\alpha}}:=[\hat\alpha_j:j\in\mathbb{N}_{n(\hat{\nu})}]\in\rm{S}_{\hat\nu}(\mathbf{y})$. According to the definition of extreme points, it suffices to prove that whenever $\hat{\bm{\beta}}:=[\hat\beta_j:j\in\mathbb{N}_{n(\hat{\nu})}]\in\rm{S}_{\hat{\nu}}(\mathbf{y})$ and $\hat{\bm{\gamma}}:=[\hat\gamma_j:j\in\mathbb{N}_{n(\hat{\nu})}]\in\rm{S}_{\hat{\nu}}(\mathbf{y})$ satisfying $\hat{\bm{\alpha}}=(\hat{\bm{\beta}}+\hat{\bm{\gamma}})/2$, we have $\hat{\bm{\beta}}=\hat{\bm{\gamma}}$. Set $\hat g:=\sum_{j\in\mathbb{N}_{n(\hat{\nu})}}\hat\beta_j K(\cdot, x_j')$ and $\hat h:=\sum_{j\in\mathbb{N}_{n(\hat{\nu})}}\hat\gamma_j K(\cdot, x_j')$. Clearly,  $\hat{f}=(\hat{g}+\hat{h})/2$. Again using statement 1 of this lemma, we obtain that $\hat{g},\hat{h}\in\rm{S}(\mathbf{y})$. Combining $f\in\mathrm{ext}(\rm{S}(\mathbf{y}))$ with the definition of extreme points, we get that $\hat{g}=\hat{h}$. 
   It follows from Assumption (A2) that  $\|\hat{\bm{\beta}}-\hat{\bm{\gamma}}\|_1=\|\hat g-\hat h\|_{\mathcal{B}}/C=0$. Thus, $\hat{\bm{\beta}}=\hat{\bm{\gamma}}$, which complets the proof.
\end{proof}

Combining the relation between the solutions of problems \eqref{MNI} and \eqref{RKBS finite d MNI section 2} and the sparsity characterization of the latter, we are ready to provide sparse kernel representations for the solutions of the MNI problem \eqref{MNI}.

\begin{theorem}\label{theorem: sparser representer for MNI}

   Suppose that $\mathbf{y}\in \mathbb{R}^{n}\backslash\{\mathbf{0}\}$, $\mathcal{V}$ and $\mathcal{M}_{\mathbf{y}}$ are defined by \eqref{V_span} and \eqref{hyperplane My}, respectively, and $\hat{\nu}\in\mathcal{V}$ satisfy \eqref{Non-empty-set}. If Assumptions (A1) and (A2) hold with a positive constant $C$, $X_{\hat\nu}'$ and $\mathbf{L}_{\hat{\nu}}$ are defined by \eqref{extreme_point_set} and \eqref{finite matrix V hat nu}, respectively, then for any $\hat f\in\mathrm{ext}\left(\rm{S}({\mathbf{y}})\right)$, there exist $\hat{\alpha}_j\neq 0$, $j\in\mathbb{N}_{M}$, with  $\sum_{j\in\mathbb{N}_{M}}|\hat{\alpha}_j|=\|\hat\nu\|_{\mathcal{B}_*}/C$ and $x_j'\in X_{\hat{\nu}}'$, $j\in\mathbb{N}_{M}$, such that \begin{equation}\label{eq: RKBS expansion of p}
        \hat f=\sum\limits_{j\in\mathbb{N}_{M}} \hat{\alpha}_j K(\cdot, x_j'),
    \end{equation}
    for some positive integer ${M}\leq \mathrm{rank}(\mathbf{L}_{\hat{\nu}})$.
\end{theorem}

\begin{proof}
    Suppose that $\hat f\in\mathrm{ext}\left(\rm{S}({\mathbf{y}})\right)$. By Proposition \ref{theorem: representer for MNI} and noting that Assumption (A1) holds, we represent $f$  as 
\begin{equation}\label{representation1}
    \hat f=\sum_{j\in\mathbb{N}_{n(\hat{\nu})}}\hat{\alpha}_j K(\cdot,x_j'),
\end{equation}
    for some  $\hat{\alpha}_j\in\mathbb{R}$, $x_j'\in X_{\hat{\nu}}'$, $j\in\mathbb{N}_{n(\hat{\nu})}$. Since $\hat f\in\mathrm{ext}(\rm{S}(\mathbf{y}))$, we get by statement 2 of Lemma \ref{lemma: relation between MNI and finite MNI section 2} that $\hat{\bm{\alpha}}:=[\hat{\alpha}_j:j\in\mathbb{N}_{n(\hat{\nu})}]\in\mathrm{ext}(\rm{S}_{\hat\nu}(\mathbf{y}))$. Proposition \ref{theorem: finite dimensional l1 MNI section 2} ensures that $\hat{\bm{\alpha}}$ has at most $\mathrm{rank}(\mathbf{L}_{\hat\nu})$ nonzero entries, which allows us to rewrite equation \eqref{representation1} as  
 \eqref{eq: RKBS expansion of p} with $\hat{\alpha}_j\neq 0$, $j\in\mathbb{N}_{M}$, for some positive integer ${M}\leq \mathrm{rank}(\mathbf{L}_{\hat{\nu}})$. It remains to show that   $\sum_{j\in\mathbb{N}_{M}}|\hat{\alpha}_j|=\|\hat\nu\|_{\mathcal{B}_*}/C$. It follows from  Assumption (A2) that 
 \begin{equation}\label{f-and-alpha_3}
 \sum_{j\in\mathbb{N}_{M}}|\hat{\alpha}_j|=\|\hat f\|_{\mathcal{B}}/C.
 \end{equation}
 Theorem 12 in \cite{wang2021representer} guarantees that if $\hat{\nu}\in\mathcal{V}$ satisfies \eqref{Non-empty-set}, then any $f\in(\|\hat\nu\|_{\mathcal{B}_*}\partial\|\cdot\|_{\mathcal{B}_*}(\hat\nu))\cap\mathcal{M}_{\mathbf{y}}$ is a solution of 
the MNI problem \eqref{MNI}. According to property \eqref{subdifferential = norming functional}, any $f\in(\|\hat\nu\|_{\mathcal{B}_*}\partial\|\cdot\|_{\mathcal{B}_*}(\hat\nu))\cap\mathcal{M}_{\mathbf{y}}$ satisfies that $\|f\|_{\mathcal{B}}=\|\hat\nu\|_{\mathcal{B}_*}$. That is, the infimum of the MNI problem \eqref{MNI} is $\|\hat\nu\|_{\mathcal{B}_*}$. By noting that $\hat f\in\mathrm{S}(\mathbf{y})$, we obtain that $\|\hat f\|_{\mathcal{B}}=\|\hat\nu\|_{\mathcal{B}_*}$. Substituting the above equation into equation \eqref{f-and-alpha_3}, we get that $\sum_{j\in\mathbb{N}_{M}}|\hat{\alpha}_j|=\|\hat\nu\|_{\mathcal{B}_*}/C$, which completes the proof.
\end{proof}

Theorem \ref{theorem: sparser representer for MNI} provides kernel representations, for the solutions of the MNI problem \eqref{MNI}, with the number of the kernel
sessions, which appear in the resulting representations,   being no more than the number $n$ of
the observed data. In Section 5, we will show by specific examples that $\mathrm{rank}(\mathbf{L}_{\hat\nu})$ is usually less than the number of
the data. Hence, Theorem \ref{theorem: sparser representer for MNI} may be taken as a sparse representer theorem for the solutions of the MNI problem. Such a sparse kernel representation profits from the fact that Assumptions (A1) and (A2) allow us to transform the original MNI problem to an equivalent finite dimensional MNI problem with the $\ell_1$ norm. As a result, a further characterization of the sparsity of the solutions of problem \eqref{RKBS finite d MNI section 2} may lead to a more precise sparsity of  the solutions of problem \eqref{MNI}.

\section{Sparse Representer Theorem for Regularization Problems}\label{sec: sparse representer theorem of regularization}
In this section, we establish a sparse representer theorem for regularization problems in the RKBS. This is done by translating the sparse representer theorem established in the last section for the MNI problem to regularization problems via the connection between the solutions of these problems. Unlike the MNI problem, the regularization problem involves a regularization parameter which allows us to further promote the sparsity level of the solution. Specifically, we convert the regularization problem in the infinite dimensional RKBS to a finite
dimensional one by using the sparse representer theorem. We then obtain choices of the regularization parameter for sparse solutions of the regularization problem in the RKBS by using the existing results for the finite dimensional regularization problem.

Throughout this section, we suppose that $\mathcal{B}$ and $\mathcal{B}'$ are a pair of RKBSs and $K:X\times X'\rightarrow \mathbb{R}$ is the reproducing kernel for $\mathcal{B}$. Let $\mathcal{B}_*$ be a pre-dual space of $\mathcal{B}$ and $\nu_{j}$, $j\in \mathbb{N}_{n}$, be linearly independent elements in $\mathcal{B}_*$. We also assume that for any given $\mathbf{y}\in\mathbb{R}^n$, both $\mathcal{Q}_{\mathbf{y}}:\mathbb{R}^{n} \rightarrow \mathbb{R}_{+}$ and $\varphi: \mathbb{R}_{+} \rightarrow \mathbb{R}_{+}$ are lower semi-continuous and moreover, $\varphi$ is increasing and coercive. It is known (\cite{unser2021unifying}) that the
solution set $\rm{R}(\mathbf{y})$  of the regularization problem \eqref{regularization problem special} with $\mathbf{y}\in\mathbb{R}^n$ and $\lambda>0$ is nonempty and weakly${}^*$ compact, and moreover, if both $\mathcal{Q}_{\mathbf{y}}$ and  $\varphi$ are convex, then $\rm{R}(\mathbf{y})$ is also convex. 

We begin with recalling the relation between a solution of the MNI problem \eqref{MNI} and that of the regularization problem \eqref{regularization problem special} which was put forward in \cite{wang2021representer}. Recalling $\mathcal{L}:\mathcal{B}\to \mathbb{R}^n$ defined by \eqref{operator L},
we introduce a subset $\mathcal{D}_{\lambda,\mathbf{y}}$ of $\mathbb{R}^n$ by  
\begin{equation}\label{Dy}
\mathcal{D}_{\lambda,\mathbf{y}}:=\mathcal{L}(\rm{R}(\mathbf{y})). 
\end{equation}
In this notation, Proposition 41 in \cite{wang2021representer} shows that  
\begin{equation}\label{Proposition41-in- WangXu1}
\bigcup_{{\mathbf{z}}\in\mathcal{D}_{\lambda,\mathbf{y}}}\rm{S}({\mathbf{z}})\subset \rm {R}(\mathbf{y}),
\end{equation}
and if $\varphi$ is further assumed to be strictly increasing, then
\begin{equation}\label{Proposition41-in- WangXu2}
\bigcup_{\mathbf{z}\in\mathcal{D}_{\lambda,\mathbf{y}}}\rm{S}({\mathbf{z}})=\rm{R}(\mathbf{y}).
\end{equation}

The next lemma concerns a relation between the extreme points of the  solution sets of problems \eqref{MNI} and  \eqref{regularization problem special}. 
\begin{lemma}\label{lemma: extreme point of regularization problem is one of MNI problem}
    Suppose that $\mathbf{y}_0\in \mathbb{R}^{n}$ and $\lambda>0$. Let $\mathcal{D}_{\lambda,\mathbf{y}_0}$ be defined by \eqref{Dy} with $\mathbf{y}:=\mathbf{y}_0$. If  both $\mathcal{Q}_{\mathbf{y}_0}, \varphi$ are convex and moreover, $\varphi$ is strictly increasing, then  
    \begin{equation}\label{relation_extreme_sets}
        \mathrm{ext}\left(\rm{R}(\mathbf{y}_0)\right)\subset\bigcup_{\mathbf{z}\in\mathcal{D}_{\lambda,\mathbf{y}_0}}\mathrm{ext}\left(\rm{S}(\mathbf{z})\right).
    \end{equation}
\end{lemma}

\begin{proof}
We first note that the Krein-Milman theorem ensures that the extreme point sets appearing in inclusion \eqref{relation_extreme_sets} are all nonempty. 

We next prove that inclusion \eqref{relation_extreme_sets} holds true. We assume that 
\begin{equation}\label{hat f in ext sry in the proof}
\hat f\in\mathrm{ext}(\rm{R}(\mathbf{y}_0)),  
\end{equation}
and proceed to show that $\hat f$ belongs to the set on the right-hand-side of \eqref{relation_extreme_sets}.
To this end, we choose $\hat{\mathbf{z}}:=\mathcal{L}(\hat f)$ and clearly, $\hat{\mathbf{z}}\in \mathcal{D}_{\lambda,\mathbf{y}_0}$. It suffices to show that $\hat f\in\mathrm{ext}(\rm{S}(\hat{\mathbf{z}}))$. It follows from \eqref{Proposition41-in- WangXu2} with $\mathbf{y}:=\mathbf{y}_0$ that 
\begin{equation}\label{smz in sry in the proof}
\rm{S}(\hat{\mathbf{z}})\subset\rm{R}(\mathbf{y}_0),
\end{equation}
and $\hat f\in \rm{S}(\hat{\mathbf{z}})$. For any $f_1,f_2\in \rm{S}(\hat{\mathbf{z}})$ satisfying $\hat f=(f_1+f_2)/2$, according to \eqref{smz in sry in the proof}, it follows that $f_1,f_2\in\rm{R}(\mathbf{y}_0)$. This combined with \eqref{hat f in ext sry in the proof} and the definition of extreme points leads to $f_1=f_2=\hat f$. Again using the definition of extreme points, we obtain that $\hat f\in\mathrm{ext}(\rm{S}(\hat{\mathbf{z}}))$. 
\end{proof} 

Through the connection between the solutions of these two problems, we translate the sparse representer theorem \ref{theorem: sparser representer for MNI} for the MNI problem \eqref{MNI} to the regularization problem \eqref{regularization problem special}.

\begin{theorem}\label{theorem: sparser representer theorem regularization}
   Suppose that $\mathbf{y}_0\in \mathbb{R}^{n}$ and $\lambda>0$. Let $\mathcal{V}$ be defined by \eqref{V_span} and $\mathcal{D}_{\lambda,\mathbf{y}_0}$ be defined by \eqref{Dy} with $\mathbf{y}:=\mathbf{y}_0$. If Assumptions (A1) and (A2) hold with a positive constant $C$, then the following statements hold.
   \begin{enumerate}
    \item If $\mathcal{D}_{\lambda,\mathbf{y}_0}\neq\{\mathbf{0}\}$, then there exists $\hat f\in\rm{R}(\mathbf{y}_0)$ such that 
    \begin{equation}\label{eq: solution hat f regularization}
        \hat f=\sum\limits_{j\in\mathbb{N}_M}\hat{\alpha}_j K(\cdot, x_j'),
    \end{equation}
    for some $\hat\nu\in\mathcal{V}$,   positive integer ${M}\leq \mathrm{rank}(\mathbf{L}_{\hat{\nu}})$, $x_j'\in X_{\hat{\nu}}'$, $j\in\mathbb{N}_{M}$, and  $\hat{\alpha}_j\neq 0$, $j\in\mathbb{N}_{M}$, with  $\sum_{j\in\mathbb{N}_{M}}|\hat{\alpha}_j|=\|\hat\nu\|_{\mathcal{B}_*}/C$. 
    
    \item If both $\mathcal{Q}_{\mathbf{y}_0}$ and $\varphi$ are convex and $\varphi$ is strictly increasing, then every nonzero extreme point $\hat f$ of $\rm{R}(\mathbf{y}_0)$ satisfies 
    \eqref{eq: solution hat f regularization} for some $\hat\nu\in\mathcal{V}$,   positive integer ${M}\leq \mathrm{rank}(\mathbf{L}_{\hat{\nu}})$, $x_j'\in X_{\hat{\nu}}'$, $j\in\mathbb{N}_{M}$, and $\hat{\alpha}_j\neq 0$, $j\in\mathbb{N}_{M}$, with  $\sum_{j\in\mathbb{N}_{M}}|\hat{\alpha}_j|=\|\hat\nu\|_{\mathcal{B}_*}/C$.
    \end{enumerate}
\end{theorem}

\begin{proof}
   We first prove Statement 1. Since $\rm{R}(\mathbf{y}_0)$ is nonempty and $\mathcal{D}_{\lambda,\mathbf{y}_0}\neq\{\mathbf{0}\}$, there exists $\hat g \in \rm{R}(\mathbf{y}_0)$ such that $\hat{\mathbf{z}}:=\mathcal{L}(\hat g)\neq \mathbf{0}$. 
   It follows from \eqref{Proposition41-in- WangXu1} that $\rm{S}(\hat{\mathbf{z}})\subset\rm{R}(\mathbf{y}_0)$. As a result, $\mathrm{ext}\left(\rm{S}(\hat{\mathbf{z}})\right)\subset\rm{R}(\mathbf{y}_0)$. Noting that  $\mathrm{ext}\left(\rm{S}(\hat{\mathbf{z}})\right)$ is nonempty, we choose $\hat{f}\in\mathrm{ext}\left(\rm{S}(\hat{\mathbf{z}})\right)$ and clearly, $\hat f\in\rm{R}(\mathbf{y}_0)$. It suffices to represent $\hat f$ as in \eqref{eq: solution hat f regularization}. According to Proposition \ref{prop: solution of dual problem gives solution of original problem}, we select $\hat\nu\in\mathcal{V}$ satisfying \eqref{Non-empty-set} with $\mathbf{y}:=\hat{\mathbf{z}}$. Then Theorem \ref{theorem: sparser representer for MNI} guarantees that $\hat f$, as an element of $\mathrm{ext}\left(\rm{S}(\hat{\mathbf{z}})\right)$, can be represented as in \eqref{eq: solution hat f regularization} for some positive integer ${M}\leq \mathrm{rank}(\mathbf{L}_{\hat{\nu}})$, $x_j'\in X_{\hat{\nu}}'$, $j\in\mathbb{N}_{M}$, and  $\hat{\alpha}_j\neq 0$, $j\in\mathbb{N}_{M}$, with  $\sum_{j\in\mathbb{N}_{M}}|\hat{\alpha}_j|=\|\hat\nu\|_{\mathcal{B}_*}/C$. 
    
    We next verify Statement 2. Suppose that $\hat f$ is a nonzero extreme point of $\rm{R}(\mathbf{y}_0)$ and set $\hat{\mathbf{z}}:=\mathcal{L}(\hat f)$. We will show that $\hat{\mathbf{z}}\neq \mathbf{0}.$ It follows from Lemma \ref{lemma: extreme point of regularization problem is one of MNI problem} that $\hat f\in\mathrm{ext}\left(\rm{S}(\hat{\mathbf{z}})\right)$. If $\hat{\mathbf{z}}= \mathbf{0}$, we obtain that $\rm{S}(\hat{\mathbf{z}})=\{0\}$ and thus $\hat f=0$. This is a contradiction. Again, we choose $\hat\nu\in\mathcal{V}$ satisfying \eqref{Non-empty-set} with $\mathbf{y}:=\hat{\mathbf{z}}$ by Proposition \ref{prop: solution of dual problem gives solution of original problem}. Theorem \ref{theorem: sparser representer for MNI} enables us to represent $\hat f\in\mathrm{ext}\left(\rm{S}(\hat{\mathbf{z}})\right)$ as in \eqref{eq: solution hat f regularization} for some positive integer ${M}\leq \mathrm{rank}(\mathbf{L}_{\hat{\nu}})$, $x_j'\in X_{\hat{\nu}}'$, $j\in\mathbb{N}_{M}$, and  $\hat{\alpha}_j\neq 0$, $j\in\mathbb{N}_{M}$, with  $\sum_{j\in\mathbb{N}_{M}}|\hat{\alpha}_j|=\|\hat\nu\|_{\mathcal{B}_*}/C$.   
\end{proof}

The regularization parameters play an important role in promoting the sparsity of the regularized solutions. Based upon the sparse representer theorem \ref{theorem: sparser representer theorem regularization}, we reveal in the following how the
regularization parameter can further promote the sparsity level of the
solution. For this purpose, we convert the regularization problem \eqref{regularization problem special} to a finite dimensional regularization problem with the $\ell_1$ norm. For given $\mathbf{y}\in\mathbb{R}^n$ and $\lambda>0$, we choose $\hat{\mathbf{z}}\in\mathcal{D}_{\lambda,\mathbf{y}}\backslash\{\mathbf{0}\}$ and $\hat\nu\in\mathcal{V}$ satisfying \eqref{Non-empty-set} with $\mathbf{y}:=\hat{\mathbf{z}}$. We then introduce the regularization problem in $\mathbb{R}^{n(\hat\nu)}$ by
\begin{equation}\label{finite general finite reg problem}
    \inf \left\{\mathcal{Q}_{\mathbf{y}}(\mathbf{L}_{\hat{\nu}}\bm{\alpha})+\lambda \varphi\left(C\|\bm{\alpha}\|_1\right): \bm{\alpha}\in\mathbb{R}^{n(\hat\nu)}\right\}.
\end{equation}
Denote by $\rm{R}_{\hat\nu}(\mathbf{y})$ the solution set of problem \eqref{finite general finite reg problem}. We show in the following lemma the relation between the solutions of the regularization problems \eqref{regularization problem special} and \eqref{finite general finite reg problem}. 
\begin{lemma}\label{lemma: solution equivalent finite regularization}
    Suppose that $\mathbf{y}_0\in \mathbb{R}^{n}$ and  $\lambda>0$. Let $\mathcal{V}$ be defined by \eqref{V_span},  $\mathcal{D}_{\lambda,\mathbf{y}_0}$ be defined by \eqref{Dy} with $\mathbf{y}:=\mathbf{y}_0$, and let $\hat{\mathbf{z}}\in\mathcal{D}_{\lambda,\mathbf{y}_0}\backslash\{\mathbf{0}\}$, $\hat\nu\in\mathcal{V}$ satisfy \eqref{Non-empty-set} with $\mathbf{y}:=\hat{\mathbf{z}}$. If Assumptions (A1) and (A2) hold, $X_{\hat\nu}'$ and  $\mathbf{L}_{\hat{\nu}}$ are defined by \eqref{extreme_point_set} and \eqref{finite matrix V hat nu}, respectively, then  $\hat{f}:=\sum_{j\in\mathbb{N}_{n(\hat{\nu})}}\hat\alpha_j K(\cdot,x_j')\in\rm{R}(\mathbf{y}_0)$ if and only if $\hat{\bm{\alpha}}:=[\hat\alpha_j:j\in\mathbb{N}_{n(\hat{\nu})}]\in\rm{R}_{\hat\nu}(\mathbf{y}_0)$.
\end{lemma}
\begin{proof}
    We suppose that $\hat{f}:=\sum_{j\in\mathbb{N}_{n(\hat{\nu})}}\hat\alpha_j K(\cdot,x_j')\in\rm{R}(\mathbf{y}_0)$ and proceed to prove that 
    $\hat{\bm{\alpha}}:=[\hat\alpha_j:j\in\mathbb{N}_{n(\hat{\nu})}]\in\rm{R}_{\hat\nu}(\mathbf{y}_0)$.
    It suffices to show that \begin{equation}\label{hat alpha is a solution}
        \mathcal{Q}_{\mathbf{y}_0}(\mathbf{L}_{\hat{\nu}}\hat{\bm{\alpha}})+\lambda \varphi(C\|\hat{\bm{\alpha}}\|_1)\leq \mathcal{Q}_{\mathbf{y}_0}(\mathbf{L}_{\hat{\nu}}\bm{\alpha})+\lambda \varphi\left(C\|\bm{\alpha}\|_1\right),
    \end{equation}
    for any $\bm{\alpha}\in\mathbb{R}^{n(\hat{\nu})}$. 
    Let $\bm{\alpha}:=[\alpha_j:j\in\mathbb{N}_{n(\hat{\nu})}]\in\mathbb{R}^{n(\hat{\nu})}$ and set $f:=\sum_{j\in\mathbb{N}_{n(\hat{\nu})}}\alpha_j K(\cdot,x_j')$.  Since $\hat{f}\in\rm{R}(\mathbf{y}_0)$, we get that  
    \begin{equation}\label{hat f is a solution}
        \mathcal{Q}_{\mathbf{y}_0}(\mathcal{L}(\hat f))+\lambda \varphi(\|\hat f\|_{\mathcal{B}})\leq \mathcal{Q}_{\mathbf{y}_0}(\mathcal{L}(f))+\lambda \varphi\left(\|f\|_{\mathcal{B}}\right).
    \end{equation}
   Substituting  equations \eqref{f-and-alpha-1} and \eqref{f-and-alpha-2} with the pair $f,\bm\alpha$ and the same equations with the pair $\hat{f}, \hat{\bm{\alpha}}$ into inequality \eqref{hat f is a solution}, we get the desired inequality \eqref{hat alpha is a solution}. Conversely, suppose that  $\hat{\bm{\alpha}}:=[\hat\alpha_j:j\in\mathbb{N}_{n(\hat{\nu})}]\in\rm{R}_{\hat\nu}(\mathbf{y}_0)$ and set $\hat{f}:=\sum_{j\in\mathbb{N}_{n(\hat{\nu})}}\hat\alpha_j K(\cdot,x_j')$. Noting that $\mathcal{D}_{\lambda,\mathbf{y}_0}\neq\{\mathbf{0}\}$, statement 1 of Theorem \ref{theorem: sparser representer theorem regularization} ensures that there exists $\widetilde{\bm{\alpha}}:=[\widetilde{\alpha}_j:j\in\mathbb{N}_{n(\hat{\nu})}]\in\mathbb{R}^{n(\hat{\nu})}$ such that $\widetilde f:=\sum_{j\in\mathbb{N}_{n(\hat{\nu})}}\widetilde{\alpha}_j K(\cdot,x_j')\in \rm{R}(\mathbf{y}_0)$. Since $\hat{\bm{\alpha}}\in\rm{R}_{\hat\nu}(\mathbf{y}_0)$, inequality \eqref{hat alpha is a solution} holds with $\bm{\alpha}:=\widetilde{\bm{\alpha}}$. Again substituting  equations \eqref{f-and-alpha-1} and \eqref{f-and-alpha-2} with the pair $\widetilde f,\widetilde{\bm{\alpha}}$ and the same equations with the pair $\hat{f}, \hat{\bm{\alpha}}$ into inequality \eqref{hat alpha is a solution} with $\bm{\alpha}:=\widetilde{\bm{\alpha}}$, we obtain inequality \eqref{hat f is a solution} with $f:=\widetilde f$. This together with $\widetilde f\in\rm{R}(\mathbf{y}_0)$ leads to $\hat{f}\in\rm{R}(\mathbf{y}_0)$.
\end{proof}

The role of the regularization parameter
on the sparsity of the solutions of the finite dimensional regularization problem with the $\ell_1$ norm has been studied in \cite{liu2023parameter}. By similar arguments in \cite{liu2023parameter},  we present a sparsity characterization of the solutions of problem \eqref{finite general finite reg problem} 
as follows. For each $j\in\mathbb{N}_{n(\hat\nu)}$, we denote
by $\mathbf{e}_j$ the unit vector with $1$ for the $j$th component and $0$ otherwise. Using these vectors, we define $n(\hat\nu)$ numbers of subsets of $\mathbb{R}^{n(\hat\nu)}$ by
$$
\Omega_l:=\left\{\sum_{j\in\mathbb{N}_l}\alpha_{k_j}\mathbf{e}_{k_j}
:\alpha_{k_j}\in\mathbb{R}\setminus{\{0\}},\ \mathrm{for} \
1\leq k_1<k_2<\cdots< k_l\leq n(\hat\nu)\right\},\ \mbox{for all}\ l\in\mathbb{N}_{n(\hat\nu)}.
$$

\begin{proposition}\label{prop: finite regularization}
    Suppose that $\mathbf{y}_0\in \mathbb{R}^{n}$, $\lambda>0$, both $\mathcal{Q}_{\mathbf{y}_0}:\mathbb{R}^{n} \rightarrow \mathbb{R}_{+}$ and $\varphi: \mathbb{R}_{+} \rightarrow \mathbb{R}_{+}$ are convex, and moreover, $\varphi$ is differentiable and strictly increasing. Let $\mathcal{V}$ be defined by \eqref{V_span}, $\mathcal{D}_{\lambda,\mathbf{y}_0}$ be defined by \eqref{Dy} with $\mathbf{y}:=\mathbf{y}_0$, and let $\hat{\mathbf{z}}\in\mathcal{D}_{\lambda,\mathbf{y}_0}\backslash\{\mathbf{0}\}$, $\hat\nu\in\mathcal{V}$ satisfy \eqref{Non-empty-set} with $\mathbf{y}:=\hat{\mathbf{z}}$. If Assumptions (A1) and (A2) hold with a positive constant $C$ and $\mathbf{L}_{\hat{\nu}}$ be defined by \eqref{finite matrix V hat nu}, then problem \eqref{finite general finite reg problem} with $\mathbf{y}:=\mathbf{y}_0$ has a solution  $\hat{\bm{\alpha}}=\sum_{i\in\mathbb{N}_{l}}\hat{\alpha}_{k_i}\mathbf{e}_{k_i}\in\Omega_l$ for some $l\in\mathbb{N}_{n(\hat\nu)}$ if and only if there exists 
    $\mathbf{a}\in\partial \mathcal{Q}_{\mathbf{y}_0}(\mathbf{L}_{\hat\nu}\hat{\bm{\alpha}})$ such that  
    \begin{align}         
    \lambda&=-(\mathbf{L}_{\hat\nu}^\top\mathbf{a})_{k_i}\mathrm{sign}(\hat\alpha_{k_i})/(C\varphi'(C\|\hat{\bm{\alpha}}\|_1)),  \ \ i\in\mathbb{N}_l,\label{lambda is greater than 1}\\
    \lambda&\geq |(\mathbf{L}_{\hat\nu}^\top\mathbf{a})_{j}|/(C\varphi'(C\|\hat{\bm{\alpha}}\|_1)), \ j\in\mathbb{N}_{n(\hat\nu)}\backslash\{k_i:i\in\mathbb{N}_l\}. \label{lambda is greater than 2}  
    \end{align}
\end{proposition}

\begin{proof}
    Due to the convexity of  $\mathcal{Q}_{\mathbf{y}_0}$ and the linearity of $\mathbf{L}_{\hat\nu}$, the fidelity term $\mathcal{Q}_{\mathbf{y}_0}\circ\mathbf{L}_{\hat\nu}$ is convex. Moreover, since $\varphi$ is increasing, convex and the norm function $\|\cdot\|_1$ is convex, we claim that the regularization term $\varphi(C\|\cdot\|_1)$ is also convex. By using the Fermat rule (\cite{Zalinescu2002convex}) and the continuity of $\varphi(C\|\cdot\|_1)$, we conclude that $\hat{\bm{\alpha}}$ is a solution of problem \eqref{finite general finite reg problem} with $\mathbf{y}:=\mathbf{y}_0$ if and only if \begin{equation}\label{Fermat's rule1}
    \mathbf{0}\in\partial (\mathcal{Q}_{\mathbf{y}_0}\circ\mathbf{L}_{\hat\nu})(\hat{\bm{\alpha}})+\lambda\partial(\varphi(C\|\cdot\|_1))(\hat{\bm{\alpha}}).
    \end{equation}
    According to the chain rule of the subdifferential and the differentiability of $\varphi$, inclusion \eqref{Fermat's rule1} can be rewritten as 
    $$
    \mathbf{0}\in\mathbf{L}_{\hat\nu}^\top\partial \mathcal{Q}_{\mathbf{y}_0}(\mathbf{L}_{\hat\nu}\hat{\bm{\alpha}})+\lambda C\varphi'(C\|\hat{\bm{\alpha}}\|_1)\partial\|\cdot\|_1(\hat{\bm{\alpha}}).
    $$ 
    Equivalently,
    there exists $\mathbf{a}\in\partial \mathcal{Q}_{\mathbf{y}_0}(\mathbf{L}_{\hat\nu}\hat{\bm{\alpha}})$ such that 
    \begin{equation}\label{Fermat's rule2}
    - \mathbf{L}_{\hat\nu}^\top\mathbf{a} \in\lambda C \varphi'(C\|\hat{\bm{\alpha}}\|_1) \partial\|\cdot\|_1(\hat{\bm{\alpha}}).    
    \end{equation}
    Noting that $\hat{\bm{\alpha}}=\sum_{i\in\mathbb{N}_{l}}\hat{\alpha}_{k_i}\mathbf{e}_{k_i}\in\Omega_l$  with $\hat\alpha_{k_i}\in\mathbb{R}\setminus\{0\}$, $i\in\mathbb{N}_{l}$, we obtain that 
    $$ \partial\|\cdot\|_1(\hat{\bm{\alpha}})
    =\left\{\mathbf{z}\in\mathbb{R}^{n(\hat\nu)}: z_{k_i}=\mathrm{sign}(\hat\alpha_{k_i}),i\in\mathbb{N}_l \ \mbox{and}\ |z_j|\leq 1, j\in\mathbb{N}_{n(\hat\nu)}\setminus\{k_i:i\in\mathbb{N}_l\}\right\}.$$
    By using the above equation and noting that $\varphi'(t)>0$ for all $t\in(0,+\infty)$, we rewrite inclusion \eqref{Fermat's rule2} as \eqref{lambda is greater than 1} and \eqref{lambda is greater than 2}. This completes the proof of this proposition.
\end{proof}

Combining Lemma \ref{lemma: solution equivalent finite regularization} with Proposition \ref{prop: finite regularization}, we are ready to obtain 
choices of the regularization parameter for sparse solutions of the regularization problem  \eqref{regularization problem special}.

\begin{theorem}\label{theorem: lambda is greater than}
Suppose that $\mathbf{y}_0\in \mathbb{R}^{n}$, $\lambda>0$, both $\mathcal{Q}_{\mathbf{y}_0}:\mathbb{R}^{n} \rightarrow \mathbb{R}_{+}$ and $\varphi: \mathbb{R}_{+} \rightarrow \mathbb{R}_{+}$ are convex, and moreover, $\varphi$ is differentiable and strictly increasing. Let $\mathcal{V}$ be defined by \eqref{V_span}, $\mathcal{D}_{\lambda,\mathbf{y}_0}$ be defined by \eqref{Dy} with $\mathbf{y}:=\mathbf{y}_0$, and let $\hat{\mathbf{z}}\in\mathcal{D}_{\lambda,\mathbf{y}_0}\backslash\{\mathbf{0}\}$, $\hat\nu\in\mathcal{V}$ satisfy \eqref{Non-empty-set} with $\mathbf{y}:=\hat{\mathbf{z}}$. If Assumptions (A1) and (A2) hold, $X_{\hat\nu}'$ and    $\mathbf{L}_{\hat{\nu}}$ are defined by \eqref{extreme_point_set} and \eqref{finite matrix V hat nu}, respectively, then problem \eqref{regularization problem special} with $\mathbf{y}:=\mathbf{y}_0$ has a solution  $\hat{f}=\sum_{i\in\mathbb{N}_{l}}\hat{\alpha}_{k_i}K(\cdot,x'_{k_i})$ with $\hat\alpha_{k_i}\in\mathbb{R}\setminus\{0\}$, $x'_{k_i}\in X_{\hat{\nu}}'$, $i\in\mathbb{N}_{l}$ for some $l\in\mathbb{N}_{n(\hat\nu)}$ if and only if there exists 
    $\mathbf{a}\in\partial \mathcal{Q}_{\mathbf{y}_0}(\mathbf{L}_{\hat\nu}\hat{\bm{\alpha}})$ with $\hat{\bm{\alpha}}:=\sum_{i\in\mathbb{N}_{l}}\hat{\alpha}_{k_i}\mathbf{e}_{k_i}$ such that \eqref{lambda is greater than 1} and \eqref{lambda is greater than 2} hold.
\end{theorem}

\begin{proof}
    It follows from Lemma \ref{lemma: solution equivalent finite regularization} that $\hat f:=\sum_{i\in\mathbb{N}_{l}}\hat\alpha_{k_i}K(\cdot,x'_{k_i})$ with $\hat\alpha_{k_i}\in\mathbb{R}\setminus\{0\}$, $x'_{k_i}\in X_{\hat{\nu}}'$, $i\in\mathbb{N}_{l}$, is a solution of problem \eqref{regularization problem special} with $\mathbf{y}:=\mathbf{y}_0$ if and only if $\hat{\bm{\alpha}}:=\sum_{i\in\mathbb{N}_{l}}\hat\alpha_{k_i}\mathbf{e}_{k_i}\in\Omega_l$ is a solution of problem \eqref{finite general finite reg problem} with $\mathbf{y}:=\mathbf{y}_0$. Proposition \ref{prop: finite regularization} ensures that the latter is equivalent to that there exists $\mathbf{a}\in\partial \mathcal{Q}_{\mathbf{y}_0}(\mathbf{L}_{\hat\nu}\hat{\bm{\alpha}})$ such that \eqref{lambda is greater than 1} and \eqref{lambda is greater than 2} hold.  This completes the proof of this theorem.
\end{proof}

Observing from Theorem \ref{theorem: lambda is greater than}, the choice
of the regularization parameter can influence the sparsity of the solution. Specifically, the equalities in \eqref{lambda is greater than 1} and
the inequalities in \eqref{lambda is greater than 2} correspond to the nonzero components
and the zero components of the solution, respectively. 
As the number of the inequalities increases,
 the solution becomes more sparse. Such a characterization of the sparsity of the solutions also benefits from Assumptions (A1) and (A2)  satisfied by the RKBS.

\section{Sparse Learning in $\ell_1(\mathbb{N})$}
In this section, we consider the MNI problem and the regularization problem in the sequence
space $\ell_1(\mathbb{N})$ which is a typical RKBS. We first verify that $\ell_1(\mathbb{N})$ satisfies Assumptions (A1) and (A2). We then specialize Theorems \ref{theorem: sparser representer for MNI} and \ref{theorem: sparser representer theorem regularization} to the RKBS $\ell_1(\mathbb{N})$ and establish sparse representer theorems for the solutions of the MNI problem and the regularization problem in this space. For the regularization problem in $\ell_1(\mathbb{N})$, we further study the influence of the regularization parameter
on the sparsity of the solutions. Finally, we show that unlike $\ell_1(\mathbb{N})$, the RKBSs $\ell_p(\mathbb{N})$, for all $1<p<+\infty$, cannot promote sparsity of a learning solution in them.

We first recall the sequence space $\ell_1(\mathbb{N})$. The Banach space $\ell_1(\mathbb{N})$ consists of all real sequences $\mathbf{x}:=[x_j:j\in\mathbb{N}]$ such that 
$\|\mathbf{x}\|_1:=\sum_{j\in\mathbb{N}}|x_j|<+\infty$. It is known that $\ell_1(\mathbb{N})$ has $c_0(\mathbb{N})$ as its pre-dual space, where $c_0(\mathbb{N})$ denotes the space of all real sequences  $\mathbf{v}:=[v_j:j\in\mathbb{N}]$ converging to $0$ as $j\to\infty$, endowed with $\|\mathbf{v}\|_{\infty}:=\sup\{|v_j|: j\in\mathbb{N}\}<+\infty$. The dual bilinear form $\langle\cdot,\cdot\rangle_{\ell_1(\mathbb{N})}$ on $c_0(\mathbb{N})\times\ell_1(\mathbb{N})$ is defined by
$\langle\mathbf{v},\mathbf{x}\rangle_{\ell_1(\mathbb{N})}:=\sum_{j\in\mathbb{N}}v_jx_j,$
for all $\mathbf{v}:=[v_j:j\in\mathbb{N}]\in c_0(\mathbb{N})$ and all $\mathbf{x}:=[x_j:j\in\mathbb{N}]\in\ell_1(\mathbb{N})$. It has been established in \cite{xu2022sparse} that $\ell_1(\mathbb{N})$ is an RKBS composed of functions defined on $\mathbb{N}$ and its $\delta$-dual is isometrically isomorphic to $c_0(\mathbb{N})$ which is also the adjoint RKBS of $\ell_1(\mathbb{N})$. That is, $\ell_1(\mathbb{N})$ and $c_0(\mathbb{N})$ are a pair of RKBSs. Moreover, the function $K:\mathbb{N}\times\mathbb{N}\to\mathbb{R}$ defined by 
\begin{equation}\label{reproducing kernel of lp}
    K(i,j):=
    \begin{cases}
    1,&i=j,\\
    0,&i\neq j,
    \end{cases}
    \ \  \mbox{for any}\ (i,j)\in\mathbb{N}\times\mathbb{N}, 
\end{equation}
is the reproducing kernel for $\ell_1(\mathbb{N})$.  

We next describe the MNI problem and the regularization problem in $\ell_1(\mathbb{N})$. We suppose that $\mathbf{v}_i:=[v_{i,j}:j\in\mathbb{N}]$, $i\in\mathbb{N}_n$, are a finite number of linearly independent elements in $c_0(\mathbb{N})$ and set 
\begin{equation}\label{V_span_l1}
\mathcal{V}:=\mathrm{span}\{\mathbf{v}_j:j\in\mathbb{N}_n\}.
\end{equation}
The operator $\mathcal{L}:\ell_1(\mathbb{N})\to\mathbb{R}^n$ with the form \eqref{operator L} may be taken as a semi-infinite matrix 
\begin{equation}\label{mathbf T}
    \mathbf{V}:=[v_{i,j}:i\in\mathbb{N}_n, j\in\mathbb{N}].
\end{equation}
For a given vector $\mathbf{y}\in\mathbb{R}^n$, the hyperplane $\mathcal{M}_{\mathbf{y}}$, defined by \eqref{hyperplane My}, has the form
\begin{equation}\label{hyperplane My l1}
    \mathcal{M}_{\mathbf{y}}:=\{\mathbf{x} \in \ell_1(\mathbb{N}): \mathbf{V}\mathbf{x}=\mathbf{y}\}.
\end{equation}
With the notation above, the MNI problem in $\ell_1(\mathbb{N})$ is formulated as 
\begin{equation}\label{MNI in l1}
    \inf \left\{\left\|\mathbf{x}\right\|_{1}: \mathbf{x} \in \mathcal{M}_{\mathbf{y}}\right\}.
\end{equation} 
By introducing a loss function $\mathcal{Q}_{\mathbf{y}}: \mathbb{R}^{n}\rightarrow \mathbb{R}_{+}$ and a regularization parameter $\lambda>0$, the regularization problem in $\ell_1(\mathbb{N})$ has the form   
\begin{equation}\label{eq: regularization problem l1}
    \inf \left\{\mathcal{Q}_{\mathbf{y}}(\mathbf{V}\mathbf{x})+\lambda\|\mathbf{x}\|_{1}: \mathbf{x}\in \ell_1(\mathbb{N})\right\}.
\end{equation}
In this section, we still denote by $\rm{S}(\mathbf{y})$ and $\rm{R}(\mathbf{y})$ the solution sets of the MNI problem \eqref{MNI in l1} and the regularization problem \eqref{eq: regularization problem l1}, respectively.

We now turn to establishing the sparse representer theorem for the solutions of problems \eqref{MNI in l1} and \eqref{eq: regularization problem l1}. We begin with verifying that the RKBS $\ell_1(\mathbb{N})$ satisfies Assumptions (A1) and (A2). To this end, we need a result about the subdifferential of the $\ell_{\infty}$ norm at any $\mathbf{v}\in c_0(\mathbb{N})$. For each $\mathbf{v}:=[v_j:j\in\mathbb{N}]\in c_0(\mathbb{N})$, by $\mathbb{N}(\mathbf{v})$ we denote the index set where the sequence $\mathbf{v}$ achieves its supremum norm $\|\mathbf{v}\|_\infty$, namely, 
\begin{equation}\label{section l1 N(v)}
    \mathbb{N}(\mathbf{v}):=\left\{j\in\mathbb{N}:|v_j|=\|\mathbf{v}\|_\infty\right\}.
\end{equation}
Note that the sequence $\mathbf{v}\in c_0(\mathbb{N})$ goes to $0$ while $j$ approaches to infinity and hence the index set $\mathbb{N}(\mathbf{v})$ is of finite cardinality. Let $n(\mathbf{v})$ denote the cardinality of $\mathbb{N}(\mathbf{v})$. %
We also introduce for each $\mathbf{v}:=[v_j:j\in\mathbb{N}]\in c_0(\mathbb{N})$ a subset of $\ell_1(\mathbb{N})$ as 
\begin{equation}\label{def: set Vu}
    \Omega(\mathbf{v}):=\left\{\mathrm{sign}(v_j)K(\cdot,j):j\in\mathbb{N}(\mathbf{v})\right\}.
\end{equation}
As has been shown in \cite{cheng2021minimum}, it holds for any $\mathbf{v}\in c_0(\mathbb{N})\backslash\{\mathbf{0}\}$ that 
\begin{equation}\label{subdifferential of infinity in c0}
    \partial\|\cdot\|_\infty(\mathbf{v})=\mathrm{co}(\Omega(\mathbf{v})).
\end{equation} 
This together with noting that $\Omega(\mathbf{v})$ is a finite set further leads to \begin{equation}\label{extreme points of partial infinity}
    \mathrm{ext}\left(\partial\|\cdot\|_\infty(\mathbf{v})\right)=\Omega(\mathbf{v}).
\end{equation} 
The next lemma shows that the RKBS $\ell_1(\mathbb{N})$ satisfies Assumptions (A1) and (A2). 

\begin{lemma}\label{lemma: A1-A2-l1}
If $\mathbf{v}_{j}, j \in \mathbb{N}_{n}$, are linearly independent elements in $c_0(\mathbb{N})$ and  $\mathcal{V}$ is defined by \eqref{V_span_l1}, then the space $\ell_1(\mathbb{N})$ satisfies Assumption (A1) with $X_{\mathbf{v}}':=\mathbb{N}(\mathbf{v})$ for any $\mathbf{v}\in\mathcal{V}$ and Assumption (A2) with $C:=1$.
\end{lemma}
\begin{proof}
    We first show that Assumption (A1) holds. It follows from definition  \eqref{def: set Vu} and equation \eqref{extreme points of partial infinity} that for any nonzero $\mathbf{v}\in\mathcal{V}$, there holds $$\mathrm{ext}\left(\partial\|\cdot\|_\infty(\mathbf{v})\right)=\left\{\mathrm{sign}(v_j)K(\cdot,j):j\in\mathbb{N}(\mathbf{v})\right\}.$$
    That is, for any nonzero $\mathbf{v}\in\mathcal{V}$, there exists a finite subset $X_{\mathbf{v}}':=\mathbb{N}(\mathbf{v})$ of $ X':=\mathbb{N}$ such that
\begin{equation*}
    \mathrm{ext}\left(\partial\|\cdot\|_\infty(\mathbf{v})\right)\subset\{-K(\cdot,j), K(\cdot,j):j\in X_{\mathbf{v}}'\}. 
\end{equation*}  
Thus, the RKBS $\ell_1(\mathbb{N})$ satisfies Assumption (A1) with  $X_{\mathbf{v}}':=\mathbb{N}(\mathbf{v})$ for any $\mathbf{v}\in\mathcal{V}$. 
   
   We next verify that Assumption (A2) holds. Note that for each $j\in\mathbb{N},$ the kernel session $K(\cdot,j)$ coincides with the vector $\mathbf{e}_j$. 
   As a result, for any $m\in\mathbb{N}$, distinct points $l_j\in \mathbb{N}$, $j\in\mathbb{N}_m$, and $\bm{\alpha}=[\alpha_j:j\in\mathbb{N}_m]\in\mathbb{R}^m$, there holds 
   $$
   \left\|\sum\limits_{j\in\mathbb{N}_m}\alpha_j K(\cdot, l_j)\right\|_1=\|\bm{\alpha}\|_1.
   $$
   Clearly, the RKBS $\ell_1(\mathbb{N})$ satisfies Assumption (A2) with $C:=1$.
\end{proof}

We are ready to specialize the sparse representer theorem \ref{theorem: sparser representer for MNI} to the MNI problem \eqref{MNI in l1}. For each $\mathbf{v}:=[v_j:j\in\mathbb{N}]\in c_0(\mathbb{N})$, we define a truncation matrix of $\mathbf{V}$ as follows.  Suppose that $\mathbb{N}(\mathbf{v})=\{k_j\in\mathbb{N}:j\in\mathbb{N}_{n(\mathbf{v})}\}$. We truncate the semi-infinite matrix $\mathbf{V}$ with the form \eqref{mathbf T} by throwing away the columns with index not appearing in $\mathbb{N}(\mathbf{v})$. 
Specifically, we define the truncation matrix 
\begin{equation}\label{def: interpolation matrix}
    \mathbf{V}_{\mathbf{v}}:=[v_{i,k_j}:i\in\mathbb{N}_n, j\in\mathbb{N}_{n(\mathbf{v})}]\in\mathbb{R}^{n\times n(\mathbf{v})}.
\end{equation}

\begin{theorem}\label{theorem: representer theorem for l1 with rank}
Suppose that $\mathbf{v}_{j}$, $j \in \mathbb{N}_{n}$, are linearly independent elements in $c_0(\mathbb{N})$ and $\mathbf{y}\in \mathbb{R}^{n}\backslash\{\mathbf{0}\}$. Let $\mathcal{V}$ and $\mathcal{M}_{\mathbf{y}}$ be defined by \eqref{V_span_l1} and \eqref{hyperplane My l1}, respectively. If  $\hat{\mathbf{v}}\in\mathcal{V}$ satisfies 
\begin{equation}\label{non empty set: l1 case}
(\|\hat{\mathbf{v}}\|_{\infty}\mathrm{co}(\Omega(\hat{\mathbf{v}})))\cap\mathcal{M}_{\mathbf{y}}\neq\emptyset,
\end{equation}
$\mathbb{N}(\hat{\mathbf{v}})$ and $\mathbf{V}_{\hat{\mathbf{v}}}$ are defined by \eqref{section l1 N(v)} and \eqref{def: interpolation matrix} with $\mathbf{v}:=\hat{\mathbf{v}}$, respectively, then for any $\hat{\mathbf{x}}\in\mathrm{ext}\left(\rm{S}({\mathbf{y}})\right)$,  
there exist $\hat{\alpha}_j\neq 0$, $j\in\mathbb{N}_{M}$, with $\sum_{j\in\mathbb{N}_{M}}|\hat{\alpha}_j|=\|\hat{\mathbf{v}}\|_{\infty}$ and $k_j\in\mathbb{N}(\hat{\mathbf{v}})$, $j\in\mathbb{N}_{M}$, such that 
\begin{equation}\label{eq: l1 finite support RKBS}
    \hat{\mathbf{x}}=\sum_{j\in\mathbb{N}_M}\hat{\alpha}_j K(\cdot, k_j).
\end{equation}
for some positive integer $M\leq\mathrm{rank}(\mathbf{V}_{\hat{\mathbf{v}}})$.
\end{theorem}
\begin{proof}
    We prove this result by specializing Theorem \ref{theorem: sparser representer for MNI} to the MNI problem in $\ell_1(\mathbb{N})$. We first note that $\ell_1(\mathbb{N})$ and $c_0(\mathbb{N})$ are a pair of RKBSs and $K$ defined by \eqref{reproducing kernel of lp} is the reproducing kernel for 
    $\ell_1(\mathbb{N})$.  Moreover, $\ell_1(\mathbb{N})$ has $c_0(\mathbb{N})$ as its pre-dual space.
    We next show that $\hat{\mathbf{v}}$ satisfies \eqref{Non-empty-set}. Substituting equation \eqref{subdifferential of infinity in c0} with $\mathbf{v}$ being replaced by $\hat{\mathbf{v}}$ into assumption \eqref{non empty set: l1 case} leads directly to   $(\|\hat{\mathbf{v}}\|_{\infty}\partial\|\cdot\|_\infty(\hat{\mathbf{v}}))\cap\mathcal{M}_{\mathbf{y}}\neq\emptyset$. That is, $\hat{\mathbf{v}}$ satisfies \eqref{Non-empty-set}.
    Finally, Lemma \ref{lemma: A1-A2-l1} guarantees that both Assumption (A1) holds with $X_{\mathbf{v}}':=\mathbb{N}(\mathbf{v})$ for any $\mathbf{v}\in\mathcal{V}$  and Assumption (A2) holds with $C:=1$. Consequently, the hypotheses of Theorem \ref{theorem: sparser representer for MNI} are satisfied. By Theorem \ref{theorem: sparser representer for MNI} with noting that $X_{\hat{\mathbf{v}}}':=\mathbb{N}(\hat{\mathbf{v}})$, $C:=1$ and the matrix $\mathbf{L}_{\hat{\mathbf{v}}}$, defined by \eqref{finite matrix V hat nu}, coincides exactly with the truncation matrix $\mathbf{V}_{\hat{\mathbf{v}}}$, any extreme point $\hat{\mathbf{x}}$ of $\rm{S}({\mathbf{y}})$ can be expressed as in equation \eqref{eq: l1 finite support RKBS}.
\end{proof}

Since the kernel session $K(\cdot,k_j)$ appearing in equation \eqref{eq: l1 finite support RKBS} has merely one nonzero entry at $k_j$ for each $j\in\mathbb{N}_M$, Theorem \ref{theorem: representer theorem for l1 with rank} shows that any extreme point of the solution set  $\rm{S}({\mathbf{y}})$ of the MNI problem \eqref{MNI in l1} has at most $\mathrm{rank}(\mathbf{V}_{\hat{\mathbf{v}}})$ nonzero components. Obviously, $\mathrm{rank}(\mathbf{V}_{\hat{\mathbf{v}}})\leq n$. According to Proposition \ref{prop: solution of dual problem gives solution of original problem} in Appendix B, the sequence $\hat{\mathbf{v}}\in\mathcal{V}$ satisfying \eqref{non empty set: l1 case} can be obtained by solving the dual problem \eqref{dual problem} with $\mathcal{B}_*:=c_0(\mathbb{N})$ and $\nu_j:=\mathbf{v}_j$, $j\in\mathbb{N}_n$. 
It has been proved in \cite{cheng2021minimum} that the resulting dual problem, as a finite dimensional optimization problem, may be solved by linear programming. Admittedly, the solutions of the dual problem may not be unique and the quantity $\mathrm{rank}(\mathbf{V}_{\hat{\mathbf{v}}})$ hinges on the choice of $\hat{\mathbf{v}}\in\mathcal{V}$, or the choice of the solution of the corresponding dual problem. 
We further remark that a data-independent representer theorem for MNI problem \eqref{MNI in l1}, established in \cite{unser2016representer}, expresses the extreme point $\hat{\mathbf{x}}$ of the solution set in terms of a linear combination of $n$ extreme points of the unit ball in $\ell_1(\mathbb{N})$, which is data-independent. Theorem  \ref{theorem: representer theorem for l1 with rank} differs from the data-independent representer theorem of \cite{unser2016representer} in the kernel representation \eqref{eq: l1 finite support RKBS}, where $k_j$, $j\in\mathbb{N}_M$, depend on the element $\hat{\mathbf{v}}$, which is a linear combination of the given data $\mathbf{v}_j$, $j\in\mathbb{N}_n$. Moreover, it follows from definition \eqref{section l1 N(v)} that the set $\mathbb{N}(\hat{\mathbf{v}})$ has a finite cardinality while there are infinitely many extreme points of the unit ball in $\ell_1(\mathbb{N})$.

Below, we present a specific example to demonstrate that the quantity $\mathrm{rank}(\mathbf{V}_{\hat{\mathbf{v}}})$ could be equal or strictly less than
$n$ depending on the distinct solution of the dual problem. We choose the functionals $\mathbf{v}_1,\mathbf{v}_2\in c_0(\mathbb{N})$ as
\begin{equation*}
    \begin{aligned}
        \mathbf{v}_{1} :=\left[\frac{1}{n}:n\in\mathbb{N}\right], \    \mathbf{v}_{2} :=\left[\frac{1}{(-2)^{n-1}}:n\in\mathbb{N}\right],
\end{aligned}
\end{equation*}
and consider the MNI problem \eqref{MNI in l1} with $\mathbf{y}:=[1,1]^\top$. In this case, the solution set of the dual problem \eqref{dual problem} can be characterized as
$$
\left\{\hat{\mathbf{c}}:=[\hat{c}_1, \hat{c}_2]^\top\in\mathbb{R}^2:\hat{c}_1+\hat{c}_2=1, -\frac{1}{2}\leq \hat{c}_1\leq \frac{3}{2}\right\},
$$
and the optimal value is $m_0=1$. 
We choose the vector $\hat{\mathbf{v}}$ satisfying \eqref{non empty set: l1 case} according to two distinct elements in the above solution set. We first select $\hat{\mathbf{c}}_1:=[-\frac{1}{2},\frac{3}{2}]^\top$ as a solution of the dual problem. 
Proposition \ref{prop: solution of dual problem gives solution of original problem} in Appendix B ensures that $\hat{\mathbf{v}}_1=-\frac{1}{2}\mathbf{v}_1+\frac{3}{2}\mathbf{v}_2$ satisfies \eqref{non empty set: l1 case}. Clearly, $\mathbb{N}(\hat{\mathbf{v}}_1)=\{1,2\}$ and $n(\hat{\mathbf{v}}_1)=2$. It follows that the matrix $\mathbf{V}_{\hat{\mathbf{v}}_1}$ has the form 
\begin{equation*}
    \mathbf{V}_{\hat{\mathbf{v}}_1}:=
    \left[\begin{array}{rr}
1 & \frac{1}{2} \\
1 & -\frac{1}{2}
\end{array}\right],
\end{equation*}
and $\mathrm{rank}(\mathbf{V}_{\hat{\mathbf{v}}_1})=2$. We next select an alternative solution of the dual problem, that is, $\hat{\mathbf{c}}_2:=[0,1]^\top$. Accordingly, $\hat{\mathbf{v}}_2=\mathbf{v}_2$ satisfies \eqref{non empty set: l1 case}. It follows from $\mathbb{N}(\hat{\mathbf{v}}_2)=\{1\}$, $n(\hat{\mathbf{v}}_2)=1$ that $\mathbf{V}_{\hat{\mathbf{v}}_2}=[1,1]^\top$ and $\mathrm{rank}(\mathbf{V}_{\hat{\mathbf{v}}_2})=1$. Notice that $\mathrm{rank}(\mathbf{V}_{\hat{\mathbf{v}}_1})$ is equal to the number of the given data points while $\mathrm{rank}(\mathbf{V}_{\hat{\mathbf{v}}_2})$
is strictly less than the number.
In this example, the MNI problem \eqref{MNI in l1} with $\mathbf{y}:=[1,1]^\top$ has a unique solution $\hat{\mathbf{x}}=[1,0,0,\ldots]\in\ell_1(\mathbb{N})$ with sparsity level $l=1$. Clearly, the quantity   $\mathrm{rank}(\mathbf{V}_{\hat{\mathbf{v}}_1})$ provides a precise characterization for the sparsity of the solution.

 By applying Theorem \ref{theorem: sparser representer theorem regularization} to the regularization problem \eqref{eq: regularization problem l1}, we get the following sparse representer theorem for the solutions.

\begin{theorem}\label{theorem: representer for regularization in l1}
    Suppose that $\mathbf{v}_{j}, j \in \mathbb{N}_{n}$, are linearly independent elements in $c_0(\mathbb{N})$, $\mathbf{y}_0\in \mathbb{R}^{n}$ and $\lambda>0$. Let $\mathcal{V}$ be defined by \eqref{V_span_l1}. If $\mathcal{Q}_{\mathbf{y}_0}$ is lower semi-continuous and convex, then every nonzero $\hat{\mathbf{x}}\in\mathrm{ext}\left(\rm{R}({\mathbf{y}}_0)\right)$ has the form \eqref{eq: l1 finite support RKBS} for some $\hat{\mathbf{v}}\in\mathcal{V}$, positive integer $M\leq\mathrm{rank}(\mathbf{V}_{\hat{\mathbf{v}}})$, $\alpha_j\neq 0$, $j\in\mathbb{N}_{M}$, with $\sum_{j\in\mathbb{N}_{M}}|\alpha_j|=\|\hat{\mathbf{v}}\|_{\infty}$ and $k_j\in\mathbb{N}(\hat{\mathbf{v}})$, $j\in\mathbb{N}_{M}$. 
\end{theorem}
 
\begin{proof}
    Again, we point out that $\ell_1(\mathbb{N})$ and $c_0(\mathbb{N})$ are a pair of RKBSs, $K$ defined by \eqref{reproducing kernel of lp} is the reproducing kernel for 
    $\ell_1(\mathbb{N})$   and in addition, $\ell_1(\mathbb{N})$ has $c_0(\mathbb{N})$ as its pre-dual space.
   The space $\ell_1(\mathbb{N})$, guaranteed by Lemma \ref{lemma: A1-A2-l1}, satisfies Assumption (A1) with $X_{\mathbf{v}}':=\mathbb{N}(\mathbf{v})$ for any $\mathbf{v}\in\mathcal{V}$ and Assumption (A2) with $C:=1$. In this case, the function $\varphi(t):=t$, $t\in\mathbb{R}_{+}$, is continuous, convex and strictly increasing. That is, the hypotheses of Theorem \ref{theorem: sparser representer theorem regularization} are all satisfied. Hence, by Theorem \ref{theorem: sparser representer theorem regularization} with $X_{\hat{\mathbf{v}}}':=\mathbb{N}(\hat{\mathbf{v}})$, $C:=1$ and $\mathbf{L}_{\hat{\mathbf{v}}}:=\mathbf{V}_{\hat{\mathbf{v}}}$, any nonzero extreme point $\hat{\mathbf{x}}$ of $\rm{R}(\mathbf{y}_0)$  has the form \eqref{eq: l1 finite support RKBS} for some $\hat{\mathbf{v}}\in\mathcal{V}$, positive integer $M\leq\mathrm{rank}(\mathbf{V}_{\hat{\mathbf{v}}})$, $\alpha_j\neq 0$, $j\in\mathbb{N}_{M}$, with $\sum_{j\in\mathbb{N}_{M}}|\alpha_j|=\|\hat{\mathbf{v}}\|_{\infty}$ and $k_j\in\mathbb{N}(\hat{\mathbf{v}})$, $j\in\mathbb{N}_{M}$.
\end{proof}

The regularization parameter $\lambda$ involved in 
the regularization problem \eqref{eq: regularization problem l1} can be used to further promote the sparsity level of the
solution. As a consequence of Theorem \ref{theorem: lambda is greater than}, we get the following choices of the regularization parameter for sparse solutions of problem \eqref{eq: regularization problem l1}. For each $\mathbf{y}\in\mathbb{R}^n$ and each $\lambda>0$, the subset $\mathcal{D}_{\lambda,\mathbf{y}}$, defined by \eqref{Dy}, takes the form 
\begin{equation}\label{Dy: l1 case}
\mathcal{D}_{\lambda,\mathbf{y}}:=\{\mathbf{V}\mathbf{x}:\mathbf{x}\in\rm{R}(\mathbf{y})\}. 
\end{equation}

\begin{theorem}\label{theorem: lambda is greater than, l1 case}
Suppose that $\mathbf{v}_{j}, j \in \mathbb{N}_{n}$, are linearly independent elements in $c_0(\mathbb{N})$, $\mathbf{y}_0 \in \mathbb{R}^{n}$, $\lambda>0$ and that $\mathcal{Q}_{\mathbf{y}_0}$ is lower semi-continuous and convex. Let  $\mathcal{D}_{\lambda,\mathbf{y}_0}$ be defined by \eqref{Dy: l1 case} with $\mathbf{y}:=\mathbf{y}_0$, $\hat{\mathbf{z}}\in\mathcal{D}_{\lambda,\mathbf{y}_0}\backslash\{\mathbf{0}\}$ and let $\mathcal{V}$ be defined by \eqref{V_span_l1}, $\hat{\mathbf{v}}\in\mathcal{V}$ satisfy \eqref{non empty set: l1 case} with $\mathbf{y}:=\hat{\mathbf{z}}$ and $\mathbb{N}(\hat{\mathbf{v}})$,  $\mathbf{V}_{\hat{\mathbf{v}}}$ be defined by \eqref{section l1 N(v)} and \eqref{def: interpolation matrix} with $\mathbf{v}:=\hat{\mathbf{v}}$, respectively. Then problem \eqref{eq: regularization problem l1} with $\mathbf{y}:=\mathbf{y}_0$ has a  solution  $\hat{\mathbf{x}}=\sum_{i\in\mathbb{N}_{l}}\hat{\alpha}_{k_i}K(\cdot,k_i)$ with $\hat\alpha_{k_i}\in\mathbb{R}\setminus\{0\}$, $k_i\in \mathbb{N}(\hat{\mathbf{v}})$, $i\in\mathbb{N}_{l}$ for some $l\in\mathbb{N}_{n(\hat{\mathbf{v}})}$ if and only if there exists 
$\mathbf{a}\in\partial \mathcal{Q}_{\mathbf{y}_0}(\mathbf{V}_{\hat{\mathbf{v}}}\hat{\bm{\alpha}})$ with $\hat{\bm{\alpha}}:=\sum_{i\in\mathbb{N}_{l}}\hat{\alpha}_{k_i}\mathbf{e}_{k_i}$ such that  
\begin{equation}\label{lambda is greater than 1 l1 case}
\lambda=-(\mathbf{V}_{\hat{\mathbf{v}}}^\top\mathbf{a})_{k_i}\mathrm{sign}(\hat\alpha_{k_i}),  \ \ i\in\mathbb{N}_l,\ \mbox{and}\ \lambda\geq |(\mathbf{V}_{\hat{\mathbf{v}}}^\top\mathbf{a})_{j}|, \ \ j\in\mathbb{N}_{n(\hat{\mathbf{v}})}\backslash\{k_i:i\in\mathbb{N}_l\}.
\end{equation}
\end{theorem}

\begin{proof}     
     Lemma \ref{lemma: A1-A2-l1} ensures that the RKBS $\ell_1(\mathbb{N})$ satisfies Assumption (A1) with $X_{\mathbf{v}}':=\mathbb{N}(\mathbf{v})$ for any $\mathbf{v}\in\mathcal{V}$  and Assumption (A2) with $C:=1$. As has been shown in the
    proof of Theorem \ref{theorem: representer theorem for l1 with rank}, assumption \eqref{non empty set: l1 case} yields that $\hat{\mathbf{v}}$ satisfies \eqref{Non-empty-set} with $\mathbf{y}:=\hat{\mathbf{z}}$. Note that $\varphi(t):=t$, $t\in\mathbb{R}_{+}$, is continuous, convex and strictly increasing. The hypotheses of Theorem \ref{theorem: lambda is greater than} are satisfied.
    Thus, by Theorem \ref{theorem: lambda is greater than} with noting that $\mathbf{L}_{\hat{\mathbf{v}}}:=\mathbf{V}_{\hat{\mathbf{v}}}$ and $X_{\hat{\mathbf{v}}}':=\mathbb{N}(\hat{\mathbf{v}})$, problem \eqref{eq: regularization problem l1} with $\mathbf{y}:=\mathbf{y}_0$ has a solution  $\hat{\mathbf{x}}=\sum_{i\in\mathbb{N}_{l}}\hat{\alpha}_{k_i}K(\cdot,k_i)$ with $\hat\alpha_{k_i}\in\mathbb{R}\setminus\{0\}$, $k_i\in \mathbb{N}(\hat{\mathbf{v}})$, $i\in\mathbb{N}_{l}$ for some $l\in\mathbb{N}_{n(\hat{\mathbf{v}})}$ if and only if there exists 
    $\mathbf{a}\in\partial \mathcal{Q}_{\mathbf{y}_0}(\mathbf{V}_{\hat{\mathbf{v}}}\hat{\bm{\alpha}})$ with $\hat{\bm{\alpha}}:=\sum_{i\in\mathbb{N}_{l}}\hat{\alpha}_{k_i}\mathbf{e}_{k_i}$ such that  
    \begin{align}
         \lambda&=-(\mathbf{V}_{\hat{\mathbf{v}}}^\top\mathbf{a})_{k_i}\mathrm{sign}(\hat\alpha_{k_i})/(\varphi'(C\|\hat{\bm{\alpha}}\|_1)),  \ \ i\in\mathbb{N}_l,\label{lambda is greater than 11 l1 case}\\
        \lambda&\geq |(\mathbf{V}_{\hat{\mathbf{v}}}^\top\mathbf{a})_{j}|/(C\varphi'(C\|\hat{\bm{\alpha}}\|_1)), \ \ j\in\mathbb{N}_{n(\hat{\mathbf{v}})}\backslash\{k_i:i\in\mathbb{N}_l\}.\label{lambda is greater than 22 l1 case}
    \end{align}
    Substituting $\varphi'(t)=1$, $t\in\mathbb{R}_+$ and $C=1$ into \eqref{lambda is greater than 1} and \eqref{lambda is greater than 2}, we get the desired characterization \eqref{lambda is greater than 1 l1 case}. 
\end{proof}

Theorems \ref{theorem: representer theorem for l1 with rank} and \ref{theorem: representer for regularization in l1} provide sparse representations for the solutions of the MNI problem and the regularization problem in $\ell_1(\mathbb{N})$, respectively. The sparsity of the solutions benefits from the capacity of  $\ell_1(\mathbb{N})$ in promoting sparsity, that is,  $\ell_1(\mathbb{N})$ satisfies Assumptions (A1) and (A2). To close this section, we show that in general the sequence spaces $\ell_p(\mathbb{N})$, for $1<p<+\infty$, will not promote sparsity. Recall that $\ell_p(\mathbb{N})$ with $1<p<+\infty$ is the Banach space of all real sequences $\mathbf{x}:=[x_j:j\in\mathbb{N}]$ such that 
$\|\mathbf{x}\|_p:=\left(\sum_{j\in\mathbb{N}}|x_j|^p\right)^{1/p}<+\infty$. It is known that $\ell_p(\mathbb{N})$ is reflexive and then the dual space $\ell_q(\mathbb{N})$ is its pre-dual space, where $1/p+1/q=1$. 
The dual bilinear form $\langle\cdot,\cdot\rangle_{\ell_p(\mathbb{N})}$ on $\ell_q(\mathbb{N})\times\ell_p(\mathbb{N})$ is defined by
$
\langle\mathbf{v},\mathbf{x}\rangle_{\ell_p(\mathbb{N})}:=\sum_{j\in\mathbb{N}}v_j x_j,
$
for all $\mathbf{v}:=[v_j:j\in\mathbb{N}]\in \ell_q(\mathbb{N})$ and all $\mathbf{x}:=[x_j:j\in\mathbb{N}]\in\ell_p(\mathbb{N})$. Clearly, $\ell_p(\mathbb{N})$, $\ell_q(\mathbb{N})$ are a pair of RKBSs and the function $K$ defined by \eqref{reproducing kernel of lp} is also the reproducing kernel of $\ell_p(\mathbb{N})$. We claim that $\ell_p(\mathbb{N})$ with $1<p<+\infty$ does not satisfy Assumptions (A1) and (A2). Indeed, the subdifferential of the norm $\|\cdot\|_q$ at any $\mathbf{v}:=[v_k:k\in\mathbb{N}]\in\ell_q(\mathbb{N})$ is a singleton, that is, \begin{equation}\label{subdifferential-lp}
\partial\|\cdot\|_q(\mathbf{v}):=\left\{\mathbf{u}:=\left[v_k\left|v_k\right|^{q-2}/\|\mathbf{v}\|_q^{q-1}:k\in\mathbb{N}\right]\right\}.
\end{equation} 
The vector $\mathbf{u}$ usually does not have the form $\alpha K(\cdot,j)$ for some $\alpha\in\{-1,1\}$ and $j\in\mathbb{N}$ unless $\mathbf{v}$ is in such a form. That is, Assumption (A1) does not hold. For any $m\in\mathbb{N}$, distinct points $l_j\in \mathbb{N}$, $j\in\mathbb{N}_m$, and $\bm{\alpha}=[\alpha_j:j\in\mathbb{N}_m]\in\mathbb{R}^m$, there holds 
$$
\left\|\sum\limits_{j\in\mathbb{N}_m}\alpha_j K(\cdot, l_j)\right\|_p=\|\bm{\alpha}\|_p,
$$
which yields that Assumption (A2) is not satisfied either for $\ell_p(\mathbb{N})$ with $1<p<+\infty$. 

We remark that Assumptions (A1) and (A2) are sufficient conditions for an RKBS to enjoy the ability of promoting the sparsity of the learning solutions. Even though the spaces $\ell_p(\mathbb{N})$ for $1<p<+\infty$ do not satisfy these assumptions, we still need to understand why these RKBSs cannot promote sparsity. 
To this end, we consider the MNI problem in $\ell_p(\mathbb{N})$ with $1<p<+\infty$. Suppose that $\mathbf{v}_j:=[v_{j,k}:k\in\mathbb{N}]$, $j\in\mathbb{N}_n$, are a finite number of linearly independent elements in $\ell_q(\mathbb{N})$. The operator $\mathcal{L}:\ell_p(\mathbb{N})\rightarrow\mathbb{R}^n$, defined by \eqref{operator L}, 
can also be taken as the semi-infinite matrix $\mathbf{V}$ with the form 
\eqref{mathbf T}. For a given vector $\mathbf{y}\in\mathbb{R}^n$, the subset $\mathcal{M}_{\mathbf{y}}$ of $\ell_p(\mathbb{N})$, defined by \eqref{hyperplane My}, has the form
\begin{equation}\label{hyperplane My lp}
    \mathcal{M}_{\mathbf{y}}:=\{\mathbf{x} \in \ell_p(\mathbb{N}): \mathbf{V}\mathbf{x}=\mathbf{y}\}.
\end{equation}
The MNI problem with $\mathbf{y}$ in $\ell_p(\mathbb{N})$ with $1<p<+\infty$ is formulated as 
\begin{equation}\label{MNI in lp}
    \inf \left\{\left\|\mathbf{x}\right\|_{p}: \mathbf{x} \in \mathcal{M}_{\mathbf{y}}\right\},
\end{equation} 
and its dual problem has the form 
\begin{equation}\label{MNI in lp dual}
\mathrm{sup}\left\{\sum_{j\in\mathbb{N}_n}c_jy_j:\left\|\sum\limits_{j=1}^n c_j\mathbf{v}_j\right\|_q=1\right\}.   
\end{equation}
Note that problem \eqref{MNI in lp} has a unique solution. By employing Propositions \ref{theorem: representer for MNI} and \ref{prop: solution of dual problem gives solution of original problem}, we represent the unique solution as follows. 

\begin{proposition}\label{representer-theorem-lp}
Let $p,q\in(1,+\infty)$ be such that $1/p+1/q=1$. Suppose that  $\mathbf{v}_{j}$, $j \in \mathbb{N}_{n}$, are linearly independent elements in $\ell_q(\mathbb{N})$ and $\mathbf{y}\in \mathbb{R}^{n}\backslash\{\mathbf{0}\}$. 
If 
$\hat{\mathbf{c}}:=[\hat{c}_j:j\in\mathbb{N}_n]\in\mathbb{R}^n$ is the solution of the dual problem \eqref{MNI in lp dual} and $\hat{\mathbf{v}}_p:=(\mathbf{y}^\top\hat{\mathbf{c}})\sum_{j\in\mathbb{N}_n} \hat{c}_j\mathbf{v}_j$, then the unique solution of problem \eqref{MNI in lp} with $\mathbf{y}$ has the form $\hat{\mathbf{x}}_p:=[\hat{x}_k:k\in\mathbb{N}]$ with 
$\hat{x}_k:=(\hat{\mathbf{v}}_p)_k\left|(\hat{\mathbf{v}}_p)_k\right|^{q-2}/\|\hat{\mathbf{v}}_p\|_q^{q-2}$, $k\in\mathbb{N}$.   
\end{proposition}
\begin{proof}
    Since $\hat{\mathbf{c}}:=[\hat{c}_j:j\in\mathbb{N}_n]\in\mathbb{R}^n$ is the solution of the dual problem \eqref{MNI in lp dual}, Proposition \ref{prop: solution of dual problem gives solution of original problem} ensures that $\hat{\mathbf{v}}_p:=(\mathbf{y}^\top\hat{\mathbf{c}})\sum_{j\in\mathbb{N}_n} \hat{c}_j\mathbf{v}_j$ satisfies \eqref{Non-empty-set} with $\mathcal{B}_*=\ell_q(\mathbb{N})$. That is, the hypothesis of Proposition \ref{theorem: representer for MNI} is satisfied. By Proposition \ref{theorem: representer for MNI}, the unique solution $\hat{\mathbf{x}}_p:=[\hat{x}_k:k\in\mathbb{N}]$ of problem \eqref{MNI in lp} can be represented as 
    \begin{equation}\label{representation-theorem-formula-lp}
       \hat{\mathbf{x}}_p=\sum\limits_{j\in\mathbb{N}_{n}} \gamma_j \mathbf{u}_j, 
    \end{equation} 
    for some $\gamma_j\in\mathbb{R}$, $j\in\mathbb{N}_{n}$, with  $\sum_{j\in\mathbb{N}_{n}}\gamma_j=\|\hat{\mathbf{v}}_p\|_q$ and $\mathbf{u}_j\in\mathrm{ext}\left(\partial\|\cdot\|_q(\hat{\mathbf{v}}_p)\right)$, $j\in\mathbb{N}_{n}$. Noting that the set $\partial\|\cdot\|_q(\hat{\mathbf{v}}_p)$ is a singleton, representation \eqref{representation-theorem-formula-lp} reduces to \begin{equation}\label{representer_lp}
    \hat{\mathbf{x}}_p=\|\hat{\mathbf{v}}_p\|_q\mathbf{u},
    \end{equation}  
    with $\mathbf{u}:=[u_k:k\in\mathbb{N}]\in\partial\|\cdot\|_q(\hat{\mathbf{v}}_p)$. It follows from equation \eqref{subdifferential-lp} with $\mathbf{v}:=\hat{\mathbf{v}}_p$ that $u_k=(\hat{\mathbf{v}}_p)_k\left|(\hat{\mathbf{v}}_p)_k\right|^{q-2}/\|\hat{\mathbf{v}}_p\|_q^{q-1}$, $k\in\mathbb{N}$. Hence, we get for each $k\in\mathbb{N}$ that $\hat{x}_k=(\hat{\mathbf{v}}_p)_k\left|(\hat{\mathbf{v}}_p)_k\right|^{q-2}/\|\hat{\mathbf{v}}_p\|_q^{q-2}$, which completes the proof.
\end{proof}

Although there is only one term involved in it, representation \eqref{representer_lp} cannot lead to sparsity of the solution $\hat{\mathbf{x}}_p$, since $\mathbf{u}$ may not be sparse under the kernel representation. In what follows, we give a necessary condition on the elements $\mathbf{v}_j$, $j\in\mathbb{N}_n$, such that the solution $\hat{\mathbf{x}}_p$ has finite nonzero components. For each $N\in\mathbb{N}$, we define a truncation operator $\mathcal{T}_N:\ell_p(\mathbb{N})\to\ell_p(\mathbb{N})$ by $\mathcal{T}_N(\mathbf{x}):=[x_{N+k}: k\in\mathbb{N}]$ for all $\mathbf{x}=[x_k:k\in\mathbb{N}]\in\ell_p(\mathbb{N})$. For each $\mathbf{x}:=[x_j:j\in\mathbb{N}]$, the support of $\mathbf{x}$, denoted by
$\mathrm{supp}(\mathbf{x})$, is defined to be the index set on which $\mathbf{x}$ is nonzero, that is,
$\mathrm{supp}(\mathbf{x}):=\left\{j\in\mathbb{N}:x_j\neq0\right\}.$ 

\begin{proposition}\label{prop: solution finite support then linearly dependent}
Let $p,q\in(1,+\infty)$ be such that $1/p+1/q=1$. Suppose that $\mathbf{v}_{j}, j \in \mathbb{N}_{n}$, are linearly independent elements in $\ell_q(\mathbb{N})$, $\mathbf{y}\in \mathbb{R}^{n}\backslash\{\mathbf{0}\}$ and $\hat{\mathbf{x}}_p$ is the unique solution of the MNI problem \eqref{MNI in lp} with $\mathbf{y}$ in $\ell_p(\mathbb{N})$. If there exists  $N\in\mathbb{N}$ such that $\mathrm{supp}(\hat{\mathbf{x}}_p)\subset\mathbb{N}_{N}$, then  $\mathcal{T}_{N}(\mathbf{v}_j)$, $j\in\mathbb{N}_n$, are linearly dependent. 
\end{proposition}
\begin{proof} 
Let $\hat{\mathbf{c}}:=[\hat{c}_j:j\in\mathbb{N}_n]\in\mathbb{R}^n$ be the solution of the dual problem \eqref{MNI in lp dual} and set $\hat{\mathbf{v}}_p:=(\mathbf{y}^\top\hat{\mathbf{c}})\sum_{j\in\mathbb{N}_n} \hat{c}_j\mathbf{v}_j$. 
Proposition \ref{representer-theorem-lp} guarantees that $\hat{\mathbf{x}}_p=[\hat{x}_k:k\in\mathbb{N}]$ with 
$\hat{x}_k:=(\hat{\mathbf{v}}_p)_k\left|(\hat{\mathbf{v}}_p)_k\right|^{q-2}/\|\hat{\mathbf{v}}_p\|_q^{q-2}$, $k\in\mathbb{N}$. 
This together with the assumption that $\mathrm{supp}(\hat{\mathbf{x}}_p)\subset\mathbb{N}_N$ leads to $\mathrm{supp}(\hat{\mathbf{v}}_p)\subset\mathbb{N}_N$. This implies that $\mathcal{T}_N(\hat{\mathbf{v}}_p)=0$. Substituting the definition of $\hat{\mathbf{v}}_p$ into this equation, we get that 
\begin{equation}\label{LinearComb}
    (\mathbf{y}^\top\hat{\mathbf{c}})\sum_{j\in\mathbb{N}_n} \hat{c}_j\mathcal{T}_N(\mathbf{v}_j)=0.
\end{equation}
Since $\mathbf{y}\neq\mathbf{0}$, the infimum of the MNI problem \eqref{MNI in lp} is nonzero. Then by Proposition \ref{prop: original problem=dual problem}, the quantity $\mathbf{y}^\top\hat{\mathbf{c}}$, as the supremum of the dual problem \eqref{MNI in lp dual}, is also nonzero. As a result, we have that $\hat{\mathbf{c}}\neq\mathbf{0}.$
Hence, it follows from \eqref{LinearComb} that $\sum_{j\in\mathbb{N}_n} \hat{c}_j\mathcal{T}_N(\mathbf{v}_j)=0$ for $\hat{\mathbf{c}}\neq\mathbf{0}.$
This ensures that $\mathcal{T}_{N}(\mathbf{v}_j)$, $j\in\mathbb{N}_n$, are linearly dependent. 
\end{proof}%

Proposition \ref{prop: solution finite support then linearly dependent} indicates that in general  the solution  $\hat{\mathbf{x}}_p$ of the MNI problem in the sequence spaces $\ell_p(\mathbb{N})$, for
$1 < p < +\infty$, will not have a finite number of nonzero components, unless strict conditions are imposed to the elements $\mathbf{v}_j$, $j\in\mathbb{N}_n$. This has been demonstrated in an example presented in \cite{cheng2021minimum}, where the solution of the MNI problem in $\ell_2(\mathbb{N})$ has infinite nonzero components, that is, the solution can only be expressed by infinitely many kernel sessions. Indeed, for the elements $\mathbf{v}_1$ and $\mathbf{v}_2$ chosen in the example, $\mathcal{T}_{N}(\mathbf{v}_1)$ and $\mathcal{T}_{N}(\mathbf{v}_2)$ will never be linearly dependent for any choice of $N\in\mathbb{N}$.

\section{Sparse Learning in the Measure Space} 
We study in this section
the MNI problem and the regularization problem in an RKBS constructed by the measure space. This RKBS is proved to have the space of continuous functions as both the adjoint RKBS and the pre-dual space. By verifying Assumptions (A1) and (A2) for the RKBS, we specialize Theorems \ref{theorem: sparser representer for MNI} and \ref{theorem: sparser representer theorem regularization} to this space and obtain the sparse representer theorems for the solutions of the MNI problem and the regularization problem in this space. 

We begin with introducing an RKBS, which has been considered in \cite{bartolucci2023understanding, lin2022reproducing, lin2022on, spek2023duality}. Let $X$ be a prescribed set and $X'$ be a locally compact Hausdorff space. Denote by $C_0(X')$ the space of all continuous functions $f:X'\to\mathbb{R}$ such that for any $\epsilon>0$, the set $\{x'\in X':|f(x')|\geq\epsilon\}$ is compact. We equip the maximum norm on $C_0(X')$, namely $\|f\|_{\infty}:=\sup_{x'\in X'}|f(x')|$ for all $f\in C_0(X')$. The Riesz-Markov representation theorem (\cite{conway2019course}) states that the dual space of $C_0(X')$ is isometrically isomorphic to the space $\mathfrak{M}(X')$ of real-valued regular Borel measures on $X'$ endowed with the total variation norm $\|\cdot\|_{\mathrm{TV}}$. Suppose that $K:X\times X'\to\mathbb{R}$ satisfies $K(x,\cdot)\in C_0(X')$ for all $x\in X$ and the density condition \begin{equation}\label{density_condition}
\overline{\mathrm{span}}\{K(x,\cdot): x\in X\}=C_0(X').
\end{equation}
Associated with the function $K$, we introduce a space of functions on $X$ by
\begin{equation}\label{RKBS B measure M(X)}
    \mathcal{B}_K:=\left\{f_{\mu}:=\int_{X'} K(\cdot, x')d\mu(x'):\mu\in \mathfrak{M}(X')\right\},
\end{equation}
equipped with 
\begin{equation}\label{norm RKBS B measure M(X)}
    \|f_{\mu}\|_{\mathcal{B}_K}:=\|\mu\|_{\mathrm{TV}}.
\end{equation}

In passing, we provide an example of a kernel $K$ that satisfies the density condition \eqref{density_condition}. We choose $X=X':=\mathbb{R}^d$ and consider the Gaussian kernel defined by  
\begin{equation}\label{Gausskernel}
K(x,y):=e^{-\frac{\|x-y\|^2}{2\sigma^2}}, \ \ x,y\in \mathbb{R}^d,
\end{equation}
with $\sigma>0$. To show that the Gaussian kernel $K$ satisfies the density condition \eqref{density_condition}, we establish that for $\nu\in\mathfrak{M}(\mathbb{R}^d)$ satisfying 
\begin{equation}\label{k mu = 0}
    \int_{\mathbb{R}^d} K(x,y)d\nu(y) =0, \quad\text{for all }x\in \mathbb{R}^d,
\end{equation}
we must have that $\nu=0$. For this purpose, we re-express  the Gaussian kernel $K$ as 
\begin{equation}\label{rewrite Gaussian kernel}
K(x,y)=\int_{\mathbb{R}^d}p(t)e^{i\langle t, x-y\rangle}dt,\quad x,y\in \mathbb{R}^d,
    \end{equation}
where 
\begin{equation*}\label{def pv}
    p(t):=\left(\frac{\sigma}{\sqrt{2\pi}}\right)^d e^{-\frac{\sigma^2\|t\|^2}{2}},\quad t\in \mathbb{R}^d.
\end{equation*}
Substituting representation \eqref{rewrite Gaussian kernel} into equation \eqref{k mu = 0} yields
\begin{equation*}
    \int_{\mathbb{R}^d} \int_{\mathbb{R}^d}p(t)e^{i\langle t, x-y\rangle}dt d\nu(y) =0,\ \ x\in \mathbb{R}^d,
\end{equation*}
which further leads to 
\begin{equation*}
   \int_{\mathbb{R}^d}\int_{\mathbb{R}^d} \int_{\mathbb{R}^d}p(t)e^{i\langle t, x-y\rangle}dtd\nu(y)d\nu(x) =0.
\end{equation*}
By the Fubini theorem, we get that 
$$
\int_{\mathbb{R}^d}p(t)\left|\hat \nu(t)\right|^2 dt =0, \ \ \mbox{where}\ \ \hat \nu(t):= \int_{\mathbb{R}^d}e^{i\langle t, x\rangle}d\nu(x), \ \ t\in \mathbb{R}^d.
$$ 
Since $p(t)>0$ for all $t\in \mathbb{R}^d$, the above equation yields that 
$\hat\nu(t)=0$ for all $t\in \mathbb{R}^d$ and hence $\nu=0$.

We now return to the investigation of the space $\mathcal{B}_K$ defined by \eqref{RKBS B measure M(X)} and \eqref{norm RKBS B measure M(X)}. We first show that it is an RKBS on $X$ which has $C_0(X')$ as a pre-dual space.

\begin{proposition}\label{prop: B M(X) is an RKBS}
     Let $X$ be a prescribed set and $X'$ be a locally compact Hausdorff space. If the function $K:X\times X'\to\mathbb{R}$ satisfies $K(x,\cdot)\in C_0(X')$ for any $x\in X$ and the density condition \eqref{density_condition}, then the space $\mathcal{B}_K$ defined by \eqref{RKBS B measure M(X)} endowed with the norm \eqref{norm RKBS B measure M(X)} is an RKBS on $X$ and $C_0(X')$ is a pre-dual space of $\mathcal{B}_K$. 
\end{proposition}

\begin{proof}
We first establish that $\mathcal{B}_K$ is an RKBS of functions on $X$.  It follows from the density condition \eqref{density_condition} that the mapping 
     $\Phi:\mathfrak{M}(X')\to\mathcal{B}_K$, defined by 
     \begin{equation}\label{isometric isomorphism from M(X') to BK}\Phi(\mu):=\int_{X'}K(\cdot,x')d\mu(x'),\ \ \mbox{for all}\ \ \mu\in\mathfrak{M}(X'), 
     \end{equation}
     is an isometric isomorphism from $\mathfrak{M}(X')$ to $\mathcal{B}_K$. Thus, $\mathcal{B}_K$ is a Banach space of functions on $X$. By Definition \ref{RKBS}, 
it suffices to verify that $\delta_{x}$, $x\in X$, are all continuous on $\mathcal{B}_K$. For all $x\in X$ and all $f\in\mathcal{B}_K$, by \eqref{RKBS B measure M(X)}, we have that
$$
f(x)=\int_{X'} K(x, x')d\mu(x'), \ \  \mbox{for some}\ \ \mu\in \mathfrak{M}(X').
$$
It follows from the definition of the total variation norm and \eqref{norm RKBS B measure M(X)} that
\begin{align*}
|\delta_{x}(f)|&=\left|\int_{X'} K(x, x')d\mu(x')\right|\\
&\leq\|K(x,\cdot)\|_{\infty}\|\mu\|_{\mathrm{TV}}\\
&=\|K(x,\cdot)\|_{\infty}\|f\|_{\mathcal{B}_K}.
\end{align*}
That is, $\delta_{x}$ is continuous on $\mathcal{B}_K$. %

It remains to show that
\begin{equation}\label{predual-of-B_K}    (C_0(X'))^*=\mathcal{B}_K.
\end{equation}
It is known from the Riesz-Markov representation theorem (\cite{conway2019course}) that
\begin{equation}\label{predual-of-B_K-proof}    (C_0(X'))^*=\mathfrak{M}(X').
\end{equation}
Moreover, the mapping $\Phi:\mathfrak{M}(X')\rightarrow\mathcal{B}_K$ defined by \eqref{isometric isomorphism from M(X') to BK} 
is an isometric isomorphism. Thus, the measure space $\mathfrak{M}(X')$ is isometrically isomorphic to $\mathcal{B}_K$. This together with \eqref{predual-of-B_K-proof} leads to  \eqref{predual-of-B_K}. 
\end{proof}

We next reveal that the $\delta$-dual space $\mathcal{B}'_K$ of $\mathcal{B}_K$ is isometrically isomorphic to $C_0(X')$, a Banach space of functions, and the function $K$ coincides with the reproducing kernel of $\mathcal{B}_K$. 

\begin{proposition}\label{prop: measure space delta dual ii to C0}
Let $X$ be a prescribed set and $X'$ be a locally compact Hausdorff space. Suppose that the function $K:X\times X'\to\mathbb{R}$ satisfies $K(x,\cdot)\in C_0(X')$ for any $x\in X$ and the density condition \eqref{density_condition}. If $\mathcal{B}_K$ is an RKBS on $X$ defined by \eqref{RKBS B measure M(X)} endowed with the norm \eqref{norm RKBS B measure M(X)}, then the $\delta$-dual space $\mathcal{B}'_K$ of $\mathcal{B}_K$ is isometrically isomorphic to $C_0(X')$ and the function $K$ is the reproducing kernel of $\mathcal{B}_K$.
\end{proposition}

\begin{proof}
    We first prove that $\mathcal{B}'_K$ is isometrically isomorphic to $C_0(X')$. According to the definition of $\mathcal{B}_K'$ and the density condition \eqref{density_condition}, it suffices to verify that $\Delta:=\mathrm{span}\{\delta_x:x\in X\}$ is isometrically isomorphic to the linear span $K_X:=\mathrm{span}\{K(x,\cdot):x\in X\}$.
    We introduce a mapping $\Psi:\Delta\to K_X$ by 
   \begin{equation}\label{isometric isomorphism_delta_KX}   \Psi\left(\sum_{j\in\mathbb{N}_m}\alpha_j\delta_{x_j}\right) := \sum_{j\in\mathbb{N}_m} \alpha_j K(x_j,\cdot),\ \mbox{ for all}\ m\in\mathbb{N}, \alpha_j\in\mathbb{R}, x_j\in X, j\in\mathbb{N}_m,
   \end{equation}
   and proceed to prove that $\Psi$ is an isometric isomorphism between $\Delta$ and $K_{X}$. Clearly, $\Psi$ is linear and surjective. It remains to prove that $\Psi$ is isometric, that is, $\left\|\ell\right\|_{\mathcal{B}_K^*}=\left\|\Psi(\ell)\right\|_{\infty}$, for all $\ell\in\Delta$. 
    For any $\ell:=\sum_{j\in\mathbb{N}_m} \alpha_j\delta_{x_j}\in\Delta$ with $m\in\mathbb{N}$, $\alpha_j\in\mathbb{R}$, $x_j\in X$, $j\in\mathbb{N}_m$ and any $f_{\mu}:=\int_{X'} K(\cdot,x')d\mu(x')\in\mathcal{B}_K$ with $\mu\in\mathfrak{M}(X')$, it holds that 
    $$
    \ell(f_{\mu})=\int_{X'}\left( \sum_{j\in\mathbb{N}_m} \alpha_jK(x_j,x')\right)d\mu(x'),
    $$ 
    which together with definition \eqref{isometric isomorphism_delta_KX} further leads to \begin{equation}\label{l-f-mu}
    \ell(f_{\mu})=\int_{X'}(\Psi(\ell))(x')d\mu(x').
    \end{equation} 
    It follows from the definition of the total variation norm that $|\ell(f_{\mu})|\leq\|\Psi(\ell)\|_{\infty}\|\mu\|_{\mathrm{TV}}.
    $
    By definition \eqref{norm RKBS B measure M(X)}, we rewrite the above inequality as $|\ell(f_{\mu})|\leq\|\Psi(\ell)\|_{\infty}\|f_\mu\|_{\mathcal{B}_K}$ for all $f_\mu\in\mathcal{B}_K$ which implies that $\|\ell\|_{\mathcal{B}_K^*}\leq\|\Psi(\ell)\|_\infty$. To prove $\|\ell\|_{\mathcal{B}_K^*}=\|\Psi(\ell)\|_\infty$, it suffices to identify a specific function $f_{\mu}$ in $\mathcal{B}_K$ such that $|\ell(f_{\mu})|=\|\Psi(\ell)\|_\infty\|f_{\mu}\|_{\mathcal{B}_K}$. Since $\Psi(\ell)\in C_0(X')$, there exists $z\in X'$ at which the function $\Psi(\ell)$ attains its norm, that is, $|\Psi(\ell)(z)|=\|\Psi(\ell)\|_\infty$. By noting that  $\delta_z\in\mathfrak{M}(X')$ satisfies $\|\delta_z\|_{\mathrm{TV}}=1$, we obtain that  $f_{\delta_z}\in\mathcal{B}_K$ and  $\|f_{\delta_z}\|_{\mathcal{B}_K}=1$. It follows from equation \eqref{l-f-mu} that 
    $$
    \ell(f_{\delta_z})=\int_{X'}(\Psi(\ell))(x') d\delta_z(x') =\Psi(\ell)(z).
    $$
    Noting that $|\Psi(\ell)(z)|=\|\Psi(\ell)\|_\infty$ and $\|f_{\delta_z}\|_{\mathcal{B}_K}=1$, we get from above equation that  $|\ell(f_{\delta_z})|=\|\Psi(\ell)\|_{\infty}\|f_{\delta_z}\|_{\mathcal{B}_K}$. Consequently, we conclude that $\|\ell\|_{\mathcal{B}_K^*}=\|\Psi(\ell)\|_\infty$ for all $\ell\in\Delta$ and hence, $\Psi$ is an isometric isomorphism between $\Delta$ and $K_{X}$.  

    We next identify the reproducing kernel of $\mathcal{B}_K$. Since the $\delta$-dual space $\mathcal{B}'_K$ is isometrically isomorphic to a Banach space of functions on $X'$, Proposition \ref{reproducing kernels for RKBS} ensures that there exists a unique reproducing kernel. To prove that $K$ is the reproducing kernel, we need to verify the reproducing property. According to the isometric isomorphism $\Psi$ defined by \eqref{isometric isomorphism_delta_KX}, 
    the bilinear form on  $C_0(X')\times\mathcal{B}_K$ can be defined by \begin{equation}\label{bilinear form}
    \left\langle f,f_\mu\right\rangle_{C_0(X')\times \mathcal{B}_K}:=\left\langle \Psi^{-1}(f),f_\mu\right\rangle_{\mathcal{B}_K},
    \end{equation}
    for all $f\in C_0(X')$ and $f_\mu\in\mathcal{B}_K$. By equation \eqref{bilinear form} with noting that $\Psi^{-1}( K(x,\cdot))=\delta_x$ for all $x\in X$, we obtain for any $x\in X$ and any $f_\mu\in\mathcal{B}_K$ that
\begin{equation*}
    f_{\mu}(x)=\left\langle\delta_x, f_\mu\right\rangle_{\mathcal{B}_K}=\left\langle K(x,\cdot), f_\mu\right\rangle_{C_0(X')\times\mathcal{B}_K}.
\end{equation*}   
 That is, the reproducing property holds. Thus, we get the conclusion that $K$ is the reproducing kernel of $\mathcal{B}_K$. 
\end{proof}

In the next result, we show that $C_0(X')$ is the adjoint RKBS of $\mathcal{B}_K$. 

\begin{proposition}\label{prop: adjoint RKBS of measure RKBS}
Let $X$ be a prescribed set and $X'$ be a locally compact Hausdorff space. Suppose that the function $K:X\times X'\to\mathbb{R}$ satisfies $K(x,\cdot)\in C_0(X')$ for any $x\in X$ and the density condition \eqref{density_condition}. If $\mathcal{B}_K$ is an RKBS on $X$ defined by \eqref{RKBS B measure M(X)} endowed with the norm \eqref{norm RKBS B measure M(X)}, then $C_0(X')$ is the adjoint RKBS of $\mathcal{B}_K$.
\end{proposition}
\begin{proof}
    We first establish that $C_0(X')$ is an RKBS. It is clear that $C_0(X')$ is a space of functions. For any $g\in C_0(X')$ and $x'\in X'$, there holds $|\delta_{x'}(g)|=|g(x')|\leq\|g\|_\infty$ which implies that the point evaluation functionals are continuous and hence $C_0(X')$ is an RKBS. We then show that $K(\cdot,x')\in\mathcal{B}_K$ for all $x\in X'$. It follows for each $x\in X$, $x'\in X'$ that \begin{equation}\label{delta-phi-K}
    K(x,x')=\int_{X'} K(x, y')d\delta_{x'}(y'),
    \end{equation} 
    which together with  definition \eqref{RKBS B measure M(X)} of $\mathcal{B}_K$ yields that $K(\cdot,x')\in\mathcal{B}_K$. It remains to verify the reproducing property \eqref{reproducing property in B'} in $C_0(X')$. As pointed out in the proof of Proposition \ref{prop: B M(X) is an RKBS}, the mapping $\Phi:\mathfrak{M}(X')\to\mathcal{B}_K$ defined by \eqref{isometric isomorphism from M(X') to BK} is an isometric isomorphism from $\mathfrak{M}(X')$ to $\mathcal{B}_K$. Moreover, it follows from equation \eqref{delta-phi-K} that $\Phi(\delta_{x'})=K(\cdot,x')$ for all $x'\in X'$.  Then we have for any $g\in C_0(X')$ and $x'\in X'$ that 
    $$
    g(x')=\langle g, \delta_{x'}\rangle_{\mathfrak{M}(X')}=\langle g,K(\cdot,x')\rangle_{\mathcal{B}_K},
    $$ 
    which proves \eqref{reproducing property in B'} with $\mathcal{B}':=C_0(X')$. Consequently, the desired result follows from the definition of the adjoint RKBS.
\end{proof}

 We now turn to describing the MNI problem and the regularization problem in $\mathcal{B}_K$. In this section, we consider learning a target function in $\mathcal{B}_K$ from the point-evaluation functional data. Specifically, suppose that $x_j\in X$, $j\in\mathbb{N}_n$, are $n$ distinct points in $X$ and $K(x_j,\cdot)$, $j\in\mathbb{N}_n$, are linearly independent elements in $C_0(X')$. Associated with these point-evaluation functionals, we set 
\begin{equation}\label{V_span_kernel}
\mathcal{V}:=\mathrm{span}\left\{K(x_j,\cdot):j\in\mathbb{N}_n\right\}.
\end{equation}
The operator $\mathcal{L}:{\mathcal{B}_K} \rightarrow \mathbb{R}^{n}$, defined by  \eqref{operator L}, is specialized as  
\begin{equation}\label{L YES DC}
\mathcal{L}(f_\mu):=\left[f_\mu(x_j): j \in \mathbb{N}_{n}\right],\ \mbox{for all}\ f_\mu\in\mathcal{B}_K.
\end{equation}
For a given vector $\mathbf{y}\in\mathbb{R}^n$, the subset $\mathcal{M}_{\mathbf{y}}$ of $\mathcal{B}_K$, defined by \eqref{hyperplane My}, has the form
\begin{equation}\label{hyperplane YES DC}
    \mathcal{M}_{\mathbf{y}}:=\{f_\mu \in {\mathcal{B}_K}: \mathcal{L}(f_\mu)=\mathbf{y}\}.
\end{equation}
With the notation above, we formulate the MNI problem in ${\mathcal{B}_K}$ as 
\begin{equation}\label{MNI in RKBS B measure M(X)}
    \inf \left\{\left\|f_\mu\right\|_{{\mathcal{B}_K}}: f_\mu \in \mathcal{M}_{\mathbf{y}}\right\},
\end{equation}
and the regularization problem in ${\mathcal{B}_K}$ as
\begin{equation}\label{eq: regularization problem RKBS B measure M(X)}
    \inf \left\{\mathcal{Q}_{\mathbf{y}}(\mathcal{L}(f_\mu))+\lambda\| f_\mu\|_{{\mathcal{B}_K}}:  f_\mu\in {\mathcal{B}_K}\right\},
\end{equation}
where $\mathcal{Q}_{\mathbf{y}}: \mathbb{R}^{n} \rightarrow \mathbb{R}_{+}$ is a loss function and $\lambda$ is a positive regularization parameter.

Before establishing the sparse representer theorems for the solutions of problems \eqref{MNI in RKBS B measure M(X)} and \eqref{eq: regularization problem RKBS B measure M(X)}, we verify Assumptions (A1) and (A2) for the RKBS $\mathcal{B}_K$. To this end, we  characterize the extreme points of the subdifferential set of the maximum norm at any $g\in C_0(X')$. We start from the result concerning the extreme points of the unit ball in $\mathcal{B}_K$. 
\begin{lemma}\label{lemma: extreme points of RKBS B}
Let $X$ be a prescribed set and $X'$ be a locally compact Hausdorff space. Suppose that the function $K:X\times X'\to\mathbb{R}$ satisfies $K(x,\cdot)\in C_0(X')$ for any $x\in X$ and the density condition \eqref{density_condition}. If $\mathcal{B}_K$ is an RKBS on $X$ defined by \eqref{RKBS B measure M(X)} endowed with the norm \eqref{norm RKBS B measure M(X)} and $B_K^0$ is the closed unit ball of ${\mathcal{B}_K}$ with center at the origin, then 
\begin{equation}\label{extreme point of RKBS unit ball}
    \mathrm{ext}\left(B_K^0\right)=\left\{-K(\cdot,x'), K(\cdot,x'):x'\in X'\right\}.
\end{equation}
\end{lemma}

\begin{proof}
Recall that the mapping $\Phi:\mathfrak{M}(X')\to\mathcal{B}_K$ defined by \eqref{isometric isomorphism from M(X') to BK} is an isometric isomorphism from $\mathfrak{M}(X')$ to $\mathcal{B}_K$. Let $\mathfrak{M}^0(X')$ denote the closed unit ball of $\mathfrak{M}(X')$ with center at the origin. Clearly, $\Psi$ is also an isometric isomorphism from $\mathfrak{M}^0(X')$ to $B^0_K$. Note that isometric isomorphisms from one normed space onto another preserve extreme points  (\cite{megginson2012introduction}).  Hence, we have that \begin{equation}\label{ext of BK0=rho ext measure}
     \mathrm{ext}(B_K^0)=\Phi(\mathrm{ext}(\mathfrak{M}^0(X'))).
 \end{equation} 
 It is known (\cite{bredies2020sparsity}) that $\mathrm{ext}\left(\mathfrak{M}^0(X')\right)=\{-\delta_{x'},\delta_{x'}:x'\in X'\}$, which together with equation \eqref{ext of BK0=rho ext measure} leads to  
 \begin{equation}\label{ext of BK}
  \mathrm{ext}(B_K^0)=\{-\Phi(\delta_{x'}),\Phi(\delta_{x'}):x'\in X'\}.
 \end{equation}
 Substituting $\Phi(\delta_{x'})=K(\cdot,x')$ for all $x'\in X'$ into \eqref{ext of BK}, we get the desired
result \eqref{extreme point of RKBS unit ball} of this lemma.
\end{proof}

Next, we characterize the extreme points of the subdifferential set $\partial\|\cdot\|_\infty(g)$ for a nonzero $g\in C_0(X')$. For each $g\in C_0(X')$, let $\mathcal{N}(g)$ denote the subset of $X'$ where  the function $g$ attains its supremum norm $\|g\|_\infty$, namely
\begin{equation}\label{def: infinity set for function}
    \mathcal{N}(g):=\left\{x'\in X':|g(x')|=\|g\|_\infty\right\}.
\end{equation}
For each $g\in C_0(X')$, we introduce a subset of ${\mathcal{B}_K}$ by 
\begin{equation}\label{def: Omage f}
    \Omega(g):=\left\{\mathrm{sign}(g(x'))K(\cdot,x'):x'\in\mathcal{N}(g)\right\}.
\end{equation}
We show in the next lemma that for each nonzero $g\in C_0(X')$, the set of the extreme points of the subdifferential set $\partial\|\cdot\|_{\infty}(g)$ coincides with $\Omega(g)$.

\begin{lemma}\label{lemma: extreme points measure space}
Let $X'$ be a locally compact Hausdorff space. If $g\in C_0(X')\backslash\{0\}$ and $\Omega(g)$ is defined by \eqref{def: Omage f}, then
\begin{equation}\label{extreme point partial infinity norm functions}
    \mathrm{ext}\left(\partial\|\cdot\|_{\infty}(g)\right)=\Omega(g).
\end{equation}
\end{lemma}

\begin{proof}
We first prove that  $\Omega(g)\subset\mathrm{ext}\left(\partial\|\cdot\|_{\infty}(g)\right)$. We assume that  $f\in\Omega(g)$ and proceed to show $f\in\mathrm{ext}\left(\partial\|\cdot\|_{\infty}(g)\right)$.
It follows from definition \eqref{def: Omage f} of $\Omega(g)$ that there exists $x'\in\mathcal{N}(g)$ such that 
\begin{equation}\label{form of h}
f=\mathrm{sign}(g(x'))K(\cdot,x').
\end{equation} 
We start with proving $f\in \partial\|\cdot\|_{\infty}(g)$. Recall $\Phi$ defined by \eqref{isometric isomorphism from M(X') to BK} provides an isometric isomorphism from $\mathfrak{M}(X')$ to $\mathcal{B}_K$ and $\Phi(\delta_{x'})=K(\cdot,x')$. Noticing that $\delta_{x'}\in\mathfrak{M}(X')$ satisfies $\|\delta_{x'}\|_{\mathrm{TV}}=1$, we obtain that $\|K(\cdot,x')\|_{\mathcal{B}_K}=1$. This together equation \eqref{form of h} leads directly to  $\|f\|_{\mathcal{B}_K}=1$. According to the reproducing property \eqref{reproducing property in B'}, we have that $\langle g,f\rangle_{\mathcal{B}_K}=\mathrm{sign}(g(x'))g(x')$. Noting that $x'\in\mathcal{N}(g)$, the above equation implies $\langle g,f\rangle_{\mathcal{B}_K}=\|g\|_{\infty}$. Hence, we conclude by equation \eqref{subdifferential = norming functional} that $f\in \partial\|\cdot\|_{\infty}(g)$.
It suffices to show that for any $f_1,f_2\in\partial\|\cdot\|_{\infty}(g)$ satisfying $f=(f_1+f_2)/2$, there holds $f=f_1=f_2$. Since $g\neq0$ and $x'\in\mathcal{N}(g)$, we get that $g(x')\neq0$. Lemma \ref{lemma: extreme points of RKBS B} ensures that function $f$ with the form \eqref{form of h} satisfies $f\in\mathrm{ext}(B_K^0)$. Moreover, it follows from $\partial\|\cdot\|_{\infty}(g)\subset B_K^0$ that   $f_1,f_2\in B_K^0$. This combined with $f\in\mathrm{ext}(B_K^0)$ and the definition of extreme points leads to $f=f_1=f_2$. Again, using the definition of extreme points, we obtain that $f\in\mathrm{ext}\left(\partial\|\cdot\|_{\infty}(g)\right)$.  

It remains to show that  $\mathrm{ext}\left(\partial\|\cdot\|_{\infty}(g)\right)\subset\Omega(g)$. Assume that $f\in\mathrm{ext}\left(\partial\|\cdot\|_{\infty}(g)\right)$. %
It follows from Proposition \ref{prop: ext is smaller} that $\mathrm{ext}\left(\partial\|\cdot\|_{\infty}(g)\right)\subset\mathrm{ext}\left(B^0_K\right).$ Hence, $f\in\mathrm{ext}\left(B^0_K\right)$. Lemma \ref{lemma: extreme points of RKBS B} ensures that there exist $\alpha\in\{-1,1\}$ and $x'\in X'$ such that 
\begin{equation}\label{f-alpha-k}
   f=\alpha K(\cdot,x'). 
\end{equation} 
Since $f\in\mathrm{ext}\left(\partial\|\cdot\|_{\infty}(g)\right)$, we get by equation \eqref{subdifferential = norming functional} that
$\langle g,f\rangle_{\mathcal{B}_K}=\|g\|_\infty.$
Substituting representation \eqref{f-alpha-k} into the above equation, we get that $\|g\|_\infty=\alpha\langle g,K(\cdot,x')\rangle_{\mathcal{B}_K}$, which together with the reproducing property \eqref{reproducing property in B'} leads to   $\|g\|_\infty=\alpha g(x')$. Clearly, $x'\in\mathcal{N}(g)$ and $\alpha=\mathrm{sign}(g(x'))$. Due to the definition \eqref{def: Omage f} of $\Omega(g)$, we conclude that $f\in\Omega(g)$.
\end{proof}

Note that for any $g\in C_0(X')$, the set $\mathcal{N}(g)$ is bounded in $X'$ but its cardinality may not be finite. To ensure that the RKBS $\mathcal{B}_K$ satisfies Assumption (A1), we need to impose the following  assumption on the reproducing kernel $K$. 
\vspace{2mm}

(A3) For any nonzero $g\in\mathcal{V}$, with $\mathcal{V}$ being defined by \eqref{V_span_kernel}, the cardinality of $\mathcal{N}(g)$ is finite. %

\vspace{2mm}
Below, we present an example of kernel $K$ which satisfies Assumption (A3). Let $X=X'=\mathbb{R}$ and $K$ be the Gaussian kernel defined by 
\eqref{Gausskernel} with $d=1$. Suppose that $x_j$, $j\in\mathbb{N}_n$, are $n$ distinct points in $\mathbb{R}$. We will show that for any nonzero $g\in C_0(\mathbb{R})$ having the form $g:=\sum_{j\in\mathbb{N}_n}\alpha_j K(x_j,\cdot)$ with $\alpha_j\in\mathbb{R}$, $j\in\mathbb{N}_n$, the subset $\mathcal{N}(g)$, defined by \eqref{def: infinity set for function}, has a finite cardinality. By the definition \eqref{Gausskernel} of $K$, $g$ is differentiable. We denote by $\mathcal{Z}(g)$ the set of the zeros of the derivative $g'$, that is, 
$\mathcal{Z}(g):=\{x\in\mathbb{R}:g'(x)=0\}.$
Clearly, $\mathcal{N}(g)\subset\mathcal{Z}(g)$. We next show that  $\mathcal{N}(g)$ has a finite cardinality.
Assume to the contrary that $\mathcal{N}(g)$ is infinite. We let $x_j$, $j\in\mathbb{N}$, be a sequence of distinct points in $\mathcal{N}(g)$. It follows from $\mathcal{N}(g)\subset\mathcal{Z}(g)$ that 
$x_j\in\mathcal{Z}(g)$, $j\in\mathbb{N}$. 
Since $\mathcal{N}(g)$ is bounded, the sequence $x_j$, $j\in\mathbb{N}$, is also bounded. Hence, Bolzano-Weierstrass theorem ensures that there is a subsequence $x_{j_k}$, $k\in\mathbb{N}$, which converges to some  $x\in\mathbb{R}$. 
According to the continuity of $g'$, we obtain that $\lim_{k\to\infty}g'(x_{j_k})=g'(x)$, which together with $x_{j_k}\in\mathcal{Z}(g)$, $k\in\mathbb{N}$, implies that $x\in\mathcal{Z}(g)$. Consequently, we conclude that $x$ is a cumulative point of $\mathcal{Z}(g)$. However, the set $\mathcal{Z}(g)$, as the set of the zeros of the analytic function $g'$, has no cumulative points, a contradiction. Therefore, the set $\mathcal{N}(g)$ must have a finite cardinality. That is, the Gaussian kernel with the form \eqref{Gausskernel} satisfies Assumption (A3).

Below, we validate that the RKBS $\mathcal{B}_K$ satisfies Assumptions (A1) and (A2) when the kernel $K$ satisfies  Assumption (A3). 

\begin{lemma}\label{lemma: A1-A2-BK}
Let $X$ be a prescribed set and $X'$ be a locally compact Hausdorff space. Suppose that the function $K:X\times X'\to\mathbb{R}$ satisfies $K(x,\cdot)\in C_0(X')$ for any $x\in X$ and the density condition \eqref{density_condition}. Let $x_j\in X$, $j\in\mathbb{N}_n$, be $n$ distinct points in $X$ and $K(x_j,\cdot)$, $j\in\mathbb{N}_n$, be linearly independent elements in $C_0(X')$, and let $\mathcal{V}$ be defined by \eqref{V_span_kernel}. If $K$ satisfies Assumption (A3), then the RKBS $\mathcal{B}_K$ defined by \eqref{RKBS B measure M(X)} endowed with the norm \eqref{norm RKBS B measure M(X)} satisfies Assumption (A1) with $X_g':=\mathcal{N}(g)$ for any $g\in C_0(X')$ and Assumption (A2) with $C:=1$. 
\end{lemma}
\begin{proof}
We first show that Assumption (A1) holds. For a nonzero $g\in\mathcal{V}$, Lemma \ref{lemma: extreme points measure space} ensures  that $\mathrm{ext}\left(\partial\|\cdot\|_{\infty}(g)\right)=\Omega(g)$, which together with definition \eqref{def: Omage f} of $\Omega(g)$ leads to
\begin{equation*}
\mathrm{ext}(\partial\|\cdot\|_{\infty}(g))=\left\{\mathrm{sign}(g(x'))K(\cdot,x'):x'\in\mathcal{N}(g)\right\}.
\end{equation*}
Note that for any $x'\in\mathcal{N}(g)$, we have that $g(x')\neq 0$ as $g$ is nonzero. Hence, we get from the above equation that 
\begin{equation*}
\mathrm{ext}(\partial\|\cdot\|_{\infty}(g))\subset\left\{-K(\cdot,x_j'),K(\cdot,x_j'):x_j'\in\mathcal{N}(g)\right\}.
\end{equation*}
Assumption (A3) guarantees the finite cardinality of the set  $\mathcal{N}(g)$. Thus, Assumption (A1) holds with $X_g':=\mathcal{N}(g)$  for any $g\in\mathcal{V}$.
    
We next prove that Assumption (A2) holds.
For any $m\in\mathbb{N}$, distinct points $x_j'\in X'$, $j\in\mathbb{N}_m$, and $\bm{\alpha}=[\alpha_j:j\in\mathbb{N}_m]\in\mathbb{R}^m$, it follows from definition \eqref{norm RKBS B measure M(X)} of the norm of $\mathcal{B}_K$ that 
{\small\begin{equation*}
\left\|\sum_{j\in\mathbb{N}_m}\alpha_j K(\cdot, x_j')\right\|_{\mathcal{B}_K}=\left\|\sum_{j\in\mathbb{N}_m}\alpha_j\delta_{x_j'}\right\|_{\mathrm{TV}}.
\end{equation*}}
It is clear that $\|\sum_{j\in\mathbb{N}_m}\alpha_j\delta_{x_j'}\|_{\mathrm{TV}}=\|\bm{\alpha}\|_1$, which together with the above equation leads to 
$\|\sum_{j\in\mathbb{N}_m}\alpha_j K(\cdot, x_j')\|_{\mathcal{B}_K}=\|\bm{\alpha}\|_1.$
Therefore, Assumption (A2) holds with $C:=1$. 
\end{proof}

We are ready to establish the sparse representation theorems for the solutions of the MNI problem \eqref{MNI in RKBS B measure M(X)} and the regularization problem \eqref{eq: regularization problem RKBS B measure M(X)}. Suppose that kernel $K$ satisfies Assumption (A3). For each $g\in\mathcal{V}$, we denote by $n(g)$ the cardinality of $\mathcal{N}(g)$ and suppose that $\mathcal{N}(g)=\{x_j'\in X':j\in\mathbb{N}_{n(g)}\}$. For each $g\in\mathcal{V}$, we introduce a kernel matrix by
\begin{equation}\label{def: matrix V g}
\mathbf{V}_{g}:=[K(x_i,x_j'):i\in\mathbb{N}_n,j\in\mathbb{N}_{n(g)}]\in \mathbb{R}^{n\times n(g)}.
\end{equation}
In the following theorems, we always let $X$ be a prescribed set and $X'$ be a locally compact Hausdorff space. Suppose that the function $K:X\times X'\to\mathbb{R}$ satisfies $K(x,\cdot)\in C_0(X')$ for any $x\in X$ and the density condition \eqref{density_condition} and that $\mathcal{B}_K$ is the RKBS on $X$ defined by \eqref{RKBS B measure M(X)} endowed with the norm \eqref{norm RKBS B measure M(X)}. In addition, let $x_j\in X$, $j\in\mathbb{N}_n$, be $n$ distinct points in $X$ and $K(x_j,\cdot)$, $j\in\mathbb{N}_n$, be linearly independent elements in $C_0(X')$. 

For the MNI problem \eqref{MNI in RKBS B measure M(X)}, we get the following sparse representer theorem by employing Theorem \ref{theorem: sparser representer for MNI}.

\begin{theorem}\label{theorem: MNI representer for RKBS B measure M(X)}
    Let $\mathbf{y}\in\mathbb{R}^n\backslash\{\bm{0}\}$, $\mathcal{V}$, and $\mathcal{M}_{\mathbf{y}}$ be defined by \eqref{V_span_kernel} and \eqref{hyperplane YES DC}, respectively. Suppose that $\hat g\in\mathcal{V}$ satisfy 
    \begin{equation}\label{non empty set: measure case}
    (\|\hat g\|_{\infty}{\mathrm{co}}(\Omega(\hat g)))\cap{\mathcal{M}}_{\mathbf{y}}\neq\emptyset. 
    \end{equation}
    If Assumption (A3) holds,  $\mathcal{N}(\hat{g})$ and  $\mathbf{V}_{\hat g}$ are  defined by \eqref{def: infinity set for function} and \eqref{def: matrix V g} with $g:=\hat g$, respectively, then for any $\hat f\in\mathrm{ext}\left(\rm{S}({\mathbf{y}})\right)$, there exist $\hat{\alpha}_j\neq 0$, $j\in\mathbb{N}_{M}$, with  $\sum_{j\in\mathbb{N}_{M}}|\hat{\alpha}_j|=\|\hat g\|_{\infty}$ and $x_j'\in\mathcal{N}(\hat g)$, $j\in\mathbb{N}_{M}$, such that 
    \begin{equation}\label{eq: representer for RKBS}
        \hat f=\sum\limits_{j\in\mathbb{N}_{M}} \hat{\alpha}_j K(\cdot,x_j'),
    \end{equation}
    for some positive integer $ {M}\leq\mathrm{rank}(\mathbf{V}_{\hat g})$,
\end{theorem}
\begin{proof}
We prove this result by employing Theorem \ref{theorem: sparser representer for MNI}. We first point out that the hypotheses about the RKBS in Theorem \ref{theorem: sparser representer for MNI} are satisfied. The RKBS $\mathcal{B}_K$ guaranteed by Proposition \ref{prop: measure space delta dual ii to C0} has $C_0(X')$ as its pre-dual space and $K$ as its reproducing kernel. In addition, according to Proposition \ref{prop: adjoint RKBS of measure RKBS}, $\mathcal{B}_K$ and $C_0(X')$ are a pair of RKBSs. 
We next verify that $\hat g$ satisfies \eqref{Non-empty-set}. Lemma \ref{lemma: extreme points measure space} ensures that $\mathrm{ext}\left(\partial\|\cdot\|_{\infty}(\hat{g})\right)=\Omega(\hat{g})$, which together with Krein-Milman theorem leads to $\partial \|\cdot\|_\infty(\hat g)=\overline{\mathrm{co}}(\Omega(\hat g))$. It follows from Assumption (A3) that $\Omega(\hat g)$ is of finite cardinality. Thus, the above equation reduces to $\partial \|\cdot\|_\infty(\hat g)={\mathrm{co}}(\Omega(\hat g))$. Substituting this equation into relation \eqref{non empty set: measure case} leads to  $(\|\hat{g}\|_{\infty}\partial\|\cdot\|_\infty(\hat g))\cap{\mathcal{M}}_{\mathbf{y}}\neq\emptyset$. That is, $\hat{g}$ satisfies \eqref{Non-empty-set}. Finally, according to Lemma \ref{lemma: A1-A2-BK}, we have that Assumption (A1) holds with $X_g':=\mathcal{N}(g)$ for any $g\in\mathcal{V}$ 
and Assumption (A2) holds with $C:=1$. Thus, the hypotheses in Theorem \ref{theorem: sparser representer for MNI} are all satisfied. By the reproducing property \eqref{reproducing property in B'}, there holds for each $i\in\mathbb{N}_n$, $j\in\mathbb{N}_{n(\hat{g})}$, that $\langle K(x_i,\cdot), K(\cdot,x_j')\rangle_{\mathcal{B}_K}=K(x_i,x_j'),$
which shows that the matrix $\mathbf{L}_{\hat{g}}$ defined by \eqref{finite matrix V hat nu} coincides with the matrix $\mathbf{V}_{\hat g}$. 
Hence, Theorem \ref{theorem: sparser representer for MNI} combined with $X'_{\hat{g}}:=\mathcal{N}(\hat{g})$ and $\mathbf{L}_{\hat{g}}:=\mathbf{V}_{\hat g}$ guarantees that any extreme point $\hat f$ of $\mathrm{S}(\mathbf{y})$ can be expressed in the form of \eqref{eq: representer for RKBS}. 
\end{proof}

Theorem \ref{theorem: MNI representer for RKBS B measure M(X)} is again data-dependent, similar to Theorem \ref{theorem: representer theorem for l1 with rank}. It improves the data-independent representer theorem, presented in \cite{bartolucci2023understanding}, which expresses the extreme point of the solution set of the MNI problem \eqref{MNI in RKBS B measure M(X)} in terms of a linear combination of $n$ elements from the set $\mathrm{ext}(B_K^0)$, where $B_K^0$ is the closed unit ball in $\mathcal{B}_K$, independent of the given data. While Theorem \ref{theorem: MNI representer for RKBS B measure M(X)} expresses an extreme point of the solution set of the MNI problem \eqref{MNI in RKBS B measure M(X)} as a linear combination of at most $M$ elements of $\mathrm{ext}(\partial\|\cdot\|_\infty(\hat g))$, which is a much smaller data-dependent subset of $\mathrm{ext}(B_K^0)$. 

By specializing Theorem \ref{theorem: sparser representer theorem regularization} to the RKBS $\mathcal{B}_K$, we establish below the sparse representer theorem for the solutions of the regularization problem \eqref{eq: regularization problem RKBS B measure M(X)}.

\begin{theorem}\label{theorem: representer for regularization in RKBS B measure M(X)}
Suppose that $\mathbf{y}_0\in\mathbb{R}^n$, $\lambda>0$ and that  $\mathcal{Q}_{\mathbf{y}_0}$ is lower semi-continuous and convex. Let $\mathcal{V}$ be defined by \eqref{V_span_kernel}. If Assumption (A3) holds, then every nonzero $\hat{f}\in\mathrm{ext}\left(\rm{R}({\mathbf{y}_0})\right)$ has the representation \eqref{eq: representer for RKBS} for some $\hat{g}\in\mathcal{V}$, positive integer $M\leq\mathrm{rank}(\mathbf{V}_{\hat{g}})$, $\alpha_j\neq 0$, $j\in\mathbb{N}_{M}$, with $\sum_{j\in\mathbb{N}_{M}}|\alpha_j|=\|\hat{g}\|_{\infty}$ and $x_j'\in\mathcal{N}(\hat{g})$, $j\in\mathbb{N}_M$. 
\end{theorem}
 
\begin{proof}
    As has been shown in Propositions \ref{prop: measure space delta dual ii to C0} and \ref{prop: adjoint RKBS of measure RKBS}, $\mathcal{B}_K$ and $C_0(X')$ are a pair of RKBSs with $K$ being the reproducing kernel of $\mathcal{B}_K$ and moreover, $\mathcal{B}_K$ takes $C_0(X')$ as its pre-dual. Since Assumption (A3) holds, Lemma \ref{lemma: A1-A2-BK} ensures that Assumption (A1) is satisfied with $X_g':=\mathcal{N}(g)$ for any $g\in\mathcal{V}$ and Assumption (A2) is satisfied with $C:=1$. We notice that the regularizer having the form $\varphi(t):=t$, $t\in\mathbb{R}_{+}$, is continuous, convex and strictly increasing. That is,  the hypotheses of Theorem \ref{theorem: sparser representer theorem regularization} are satisfied. We then obtain the desired result by using Theorem \ref{theorem: sparser representer theorem regularization} and replacing $X'_{\hat{g}}$ and $\mathbf{L}_{\hat{g}}$  by $\mathcal{N}(\hat{g})$ and $\mathbf{V}_{\hat g}$, respectively. 
\end{proof}

Remarks on the relation of Theorems \ref{theorem: MNI representer for RKBS B measure M(X)} and \ref{theorem: representer for regularization in RKBS B measure M(X)} with the result of \cite{song2013reproducing} are in order.
The RKBS with the $\ell_1$ norm considered in \cite{song2013reproducing} is a subspace of the RKBS defined by \eqref{RKBS B measure M(X)} and \eqref{norm RKBS B measure M(X)}. Representer theorems for the MNI and the regularization problems were also obtained in  \cite{song2013reproducing} under an additional admissibility condition. Our results differ from those of \cite{song2013reproducing} in several aspects. First of all, the sparse representer theorems presented in Theorems \ref{theorem: MNI representer for RKBS B measure M(X)} and \ref{theorem: representer for regularization in RKBS B measure M(X)} do not require the admissibility condition on the kernel function as  \cite{song2013reproducing} requires. 
Second, \cite{song2013reproducing} represented solutions as in \eqref{eq: representer for RKBS} with $x_j'$, $j\in\mathbb{N}_n$, being given  training samples, while we represent the extreme points of the solution sets in the form \eqref{eq: representer for RKBS}, with the points $x_j'$, $j\in\mathbb{N}_M$, not necessarily training samples. We also provide a precise characterization for the points $x_j'$, $j\in\mathbb{N}_M$, which belong to the data-dependent set $\mathcal{N}(\hat g)$. Finally, the number $n$ of the kernel sections in the solution representation of \cite{song2013reproducing} coincides with the number of the training samples, while the number $M$ of the kernel sections in representation \eqref{eq: representer for RKBS} may be smaller than the number of the training samples. 

The regularization parameter $\lambda$ in \eqref{eq: regularization problem RKBS B measure M(X)} allows to further promote the sparsity level of the regularized solutions. We show in the next theorem how the choice of the parameter $\lambda$ can accomplish it by employing Theorem \ref{theorem: lambda is greater than}. 

\begin{theorem}
 Suppose that $\mathbf{y}_0\in\mathbb{R}^n$, $\lambda>0$ and that  $\mathcal{Q}_{\mathbf{y}_0}$ is lower semi-continuous and convex. Let $\mathcal{D}_{\lambda,\mathbf{y}_0}$ be defined by \eqref{Dy} with $\mathbf{y}:=\mathbf{y}_0$, $\hat{\mathbf{z}}\in\mathcal{D}_{\lambda,\mathbf{y}_0}\backslash\{\mathbf{0}\}$ and let $\mathcal{V}$ be defined by \eqref{V_span_kernel}, $\hat{g}\in\mathcal{V}$ satisfy \eqref{non empty set: measure case} with $\mathbf{y}:=\hat{\mathbf{z}}$. If Assumption (A3) holds, $\mathcal{N}(\hat g)$ and $\mathbf{V}_{\hat{g}}$ be defined by \eqref{def: infinity set for function} and \eqref{def: matrix V g} with $g:=\hat{g}$, respectively, then problem \eqref{eq: regularization problem RKBS B measure M(X)} with $\mathbf{y}:=\mathbf{y}_0$ has a solution  $\hat{f}=\sum_{i\in\mathbb{N}_{l}}\hat{\alpha}_{k_i} K(\cdot,x_{k_i})$ with $\hat{\alpha}_{k_i}\in\mathbb{R}\setminus\{0\}$,  $x_{k_i}\in\mathcal{N}(\hat g)$, $i\in\mathbb{N}_{l}$ for some $l\in\mathbb{N}_{n(\hat{g})}$ if and only if there exists 
    $\mathbf{a}\in\partial \mathcal{Q}_{\mathbf{y}}(\mathbf{V}_{\hat{g}}\hat{\bm{\alpha}})$ for $\hat{\bm{\alpha}}:=\sum_{i\in\mathbb{N}_{l}}\hat{\alpha}_{k_i}\mathbf{e}_{k_i}$ such that      
    \begin{equation}\label{lambda is greater than measure case}
         \lambda=-(\mathbf{V}_{\hat{g}}^\top\mathbf{a})_{k_i}\mathrm{sign}(\hat\alpha_{k_i}),  \ \ i\in\mathbb{N}_l,\text{ and }
        \lambda\geq |(\mathbf{V}_{\hat{g}}^\top\mathbf{a})_{j}|, \ \ j\in\mathbb{N}_{n(\hat g)}\backslash\{k_i:i\in\mathbb{N}_l\}.
    \end{equation} 
\end{theorem}

\begin{proof}
We prove this theorem by specializing Theorem \ref{theorem: lambda is greater than} to problem \eqref{eq: regularization problem RKBS B measure M(X)}. Once again, as has been shown in Propositions \ref{prop: measure space delta dual ii to C0} and \ref{prop: adjoint RKBS of measure RKBS}, $\mathcal{B}_K$ and $C_0(X')$ are a pair of RKBSs with $K$ being the reproducing kernel of $\mathcal{B}_K$ and moreover, $\mathcal{B}_K$ takes $C_0(X')$ as its pre-dual.
Moreover, Assumption (A3) ensures that Assumption (A1) holds with $X_g':=\mathcal{N}(g)$ for any $g\in\mathcal{V}$ and Assumption (A2) holds with $C:=1$. As in the proof of Theorem \ref{theorem: MNI representer for RKBS B measure M(X)}, one can see that assumption \eqref{non empty set: measure case} ensures that $\hat g$ satisfies \eqref{Non-empty-set} with $\mathbf{y}:=\hat{\mathbf{z}}$.
We also notice that $\varphi(t):=t$, $t\in\mathbb{R}_{+}$ in the regularization problem \eqref{eq: regularization problem RKBS B measure M(X)}, is continuous, convex and strictly increasing. The hypotheses of Theorem \ref{theorem: lambda is greater than} are all fulfilled. Therefore, by Theorem \ref{theorem: lambda is greater than} with $\mathbf{L}_{\hat{g}}:=\mathbf{V}_{\hat{g}}$ and $X_{\hat g}':=\mathcal{N}(\hat g)$, problem \eqref{eq: regularization problem RKBS B measure M(X)} with $\mathbf{y}:=\mathbf{y}_0$ has a solution  $\hat{f}=\sum_{i\in\mathbb{N}_{l}}\hat{\alpha}_{k_i} K(\cdot,x_{k_i})$ with $\hat{\alpha}_{k_i}\in\mathbb{R}\setminus\{0\}$,  $x_{k_i}\in\mathcal{N}(\hat g)$, $i\in\mathbb{N}_{l}$ for some $l\in\mathbb{N}_{n(\hat{g})}$ if and only if there exists 
$\mathbf{a}\in\partial \mathcal{Q}_{\mathbf{y}}(\mathbf{V}_{\hat{g}}\hat{\bm{\alpha}})$ with $\hat{\bm{\alpha}}:=\sum_{i\in\mathbb{N}_{l}}\hat{\alpha}_{k_i}\mathbf{e}_{k_i}\in\Omega_l$ such that       
\begin{align}
\lambda&=-(\mathbf{V}_{\hat{g}}^\top\mathbf{a})_{k_i}\mathrm{sign}(\hat\alpha_{k_i})/(\varphi'(C\|\hat{\bm{\alpha}}\|_1)),  \ \ i\in\mathbb{N}_l,\label{lambda is greater than 11 measure case}\\
\lambda&\geq |(\mathbf{V}_{\hat{g}}^\top\mathbf{a})_{j}|/(C\varphi'(C\|\hat{\bm{\alpha}}\|_1)), \ \ j\in\mathbb{N}_{n(\hat{g})}\backslash\{k_i:i\in\mathbb{N}_l\}.\label{lambda is greater than 22 measure case}
\end{align}
Substituting $\varphi'(t)=1$, $t\in\mathbb{R}_+$ and $C=1$ into \eqref{lambda is greater than 11 measure case} and \eqref{lambda is greater than 22 measure case} yields the desired result \eqref{lambda is greater than measure case}. 
\end{proof}

\section{Conclusion}
 
We have studied attributes of RKBSs that can promote sparsity of learning solutions in the spaces. We have proposed sufficient conditions on RKBSs which give rise to explicit and data-dependent sparse representer theorems for solutions of the MNI problem and the regularization problem in the spaces. Following the established general sparse representer theorems, we have shown that two specific RKBSs, the sequence space $\ell_1(\mathbb{N})$ and the measure space, have sparse representations for solutions of the MNI problem and the regularization problem in these spaces.

\bigskip

\noindent{\bf Acknowledgment:}
R. Wang is supported in part by the Natural Science Foundation of China under grant 12171202 and by the National Key Research and Development Program of China (grants no. 2020YFA0714101 and 2020YFA0713600); Y. Xu is supported in part by the US National Science Foundation under grants DMS-1912958 and DMS-2208386, and by the US National Institutes of Health under grant R21CA263876. Send all correspondence to Y. Xu.

\appendix
\section*{Appendix A. Explicit Representer Theorems for MNI in Banach Spaces}
\label{app:theorem}

An explicit representer theorem for the solutions of the MNI problem in a general Banach space having a pre-dual space is stated in Proposition \ref{theorem: representer for MNI} of Section \ref{sec: sparse representer theorem of MNI}. In this appendix, we give a complete proof for the proposition. 

We first present several useful properties of the solution set $\rm{S}(\mathbf{y})$ of problem \eqref{MNI} with $\mathbf{y}\in\mathbb{R}^n$. 
 For this purpose, we recall two well-known results. The generalized Weierstrass Theorem (\cite{kurdila2006convex}) shows that if $\mathcal{X}$ is a compact topological space and a functional $\mathcal{T}: \mathcal{X}\rightarrow\mathbb{R}$ is lower semi-continuous, then there exists $f_{0} \in \mathcal{X}$ such that
$$
\mathcal{T}f_{0}=\inf\{\mathcal{T}(f): f \in \mathcal{X}\}.
$$ 
A consequence of the Banach-Alaoglu Theorem (\cite{megginson2012introduction}) ensures that any bounded and weakly${}^*$ closed subset of $\mathcal{B}^{*}$ is weakly${}^*$ compact. 

\begin{lemma}\label{lemma: solution set of MNI}
    If the Banach space $\mathcal{B}$ has a pre-dual space $\mathcal{B}_*$ and $\nu_j\in \mathcal{B}_*$, $j\in\mathbb{N}_n$, are linearly independent, then for any ${\mathbf{y}}\in\mathbb{R}^n$, the solution set $\rm{S}(\mathbf{y})$ of problem \eqref{MNI} is nonempty, convex and weakly${}^*$ compact. 
\end{lemma}
\begin{proof}
    We first show that $\rm{S}(\mathbf{y})$ is nonempty. Note that the linear independence of $\nu_j$, $j\in\mathbb{N}_n$, guarantees that $\mathcal{M}_{\mathbf{y}}$ is nonempty. Pick $g\in \mathcal{M}_{\mathbf{y}}$ and set $r:=\|g\|_\mathcal{B}+1$. It follows that ${B_r}\cap \mathcal{M}_{\mathbf{y}}\neq\emptyset$, where ${B_r}$ denotes the closed ball centered at the origin with radius $r$ under the norm $\|\cdot\|_\mathcal{B}$. Accordingly, we rewrite the MNI problem \eqref{MNI} as
    \begin{equation}\label{MNI_new}
    \inf\left\{\|f\|_\mathcal{B}:f\in {B_r}\cap \mathcal{M}_{\mathbf{y}}\right\}.
    \end{equation}
    We prove that problem \eqref{MNI_new} has at least one solution by using the generalized Weierstrass Theorem. Note that the norm $\|\cdot\|_{\mathcal{B}}$ is weakly$^{*}$ lower semi-continuous on $\mathcal{B}$. It suffices to verify that ${B_r}\cap \mathcal{M}_{\mathbf{y}}$ is weakly${}^*$ compact. We first claim that ${B_r}\cap \mathcal{M}_{\mathbf{y}}$ is weakly${}^*$ closed. Indeed, suppose that the sequence $f_m$, $m\in\mathbb{N}$, in ${B_r}\cap \mathcal{M}_{\mathbf{y}}$ converges  weakly${}^*$ to $f$. We then get that
    $\lim _{m \rightarrow+\infty} \left\langle f_{m}, \nu_j\right\rangle_{\mathcal{B}_*} =\langle f,\nu_j\rangle_{\mathcal{B}_*},\ \mbox{for all}\ j\in\mathbb{N},$
    which together with equation \eqref{natural-map-predual} leads to 
    $\lim _{m \rightarrow+\infty} \left\langle\nu_j, f_{m} \right\rangle_{\mathcal{B}} =\langle\nu_j, f\rangle_{\mathcal{B}},\ \mbox{for all}\ j\in\mathbb{N}.$ 
    Noting that $\left\langle\nu_j, f_{m} \right\rangle_{\mathcal{B}}=y_j$, for all $m\in\mathbb{N}$ and all $j\in\mathbb{N}_n$, the above equation yields that $\langle\nu_j, f\rangle_{\mathcal{B}}=y_j$, for all $j\in\mathbb{N}_n$. That is to say, $f\in\mathcal{M}_{\mathbf{y}}$.
	According to the weakly$^{*}$ lower semi-continuity of the norm
	$\|\cdot\|_{\mathcal{B}}$ on $\mathcal{B}$, we have that
\begin{equation}\label{x in Br}
\|f\|_\mathcal{B}\leq\liminf_m\|f_m\|_\mathcal{B}\leq r.
\end{equation}
We get the conclusion that $f\in {B_r}\cap \mathcal{M}_{\mathbf{y}}$ and hence ${B_r}\cap \mathcal{M}_{\mathbf{y}}$ is weakly${}^*$ closed. Obviously, ${B_r}\cap {\mathcal{M}}_{\mathbf{y}}$ is bounded in $\mathcal{B}$. Consequently, the Banach-Alaoglu Theorem guarantees that  ${B_r}\cap \mathcal{M}_{\mathbf{y}}$ is weakly${ }^*$ compact. By virtue of the generalized Weierstrass theorem, there exists an $f_0\in {B_r}\cap \mathcal{M}_{\mathbf{y}}$ such that $\|f_0\|_\mathcal{B}=\inf\left\{\|f\|_\mathcal{B}:f\in {B_r}\cap \mathcal{M}_{\mathbf{y}}\right\}.$
    That is, $f_0\in\rm{S}(\mathbf{y})$.
    
    We next verify that $\rm{S}(\mathbf{y})$ is convex. It is clear that for any $f_1,f_2\in\rm{S}(\mathbf{y})$ and any $\theta\in[0,1]$, there holds $\theta f_1+(1-\theta) f_2\in\mathcal{M}_{\mathbf{y}}$. Moreover, we also have that 	\begin{equation*}
	    \left\|\theta f_1+(1-\theta)f_2\right\|_{\mathcal{B}}\leq \theta\|f_1\|_{\mathcal{B}}+(1-\theta)\|f_2\|_\mathcal{B},
	\end{equation*}
	which together with $f_1$ and $f_2$ both attaining the infimum in \eqref{MNI} yields that $\theta f_1+(1-\theta)f_2$ is also a solution of \eqref{MNI}. That is, $\theta f_1+(1-\theta)f_2\in\rm{S}(\mathbf{y})$. Thus,  $\rm{S}(\mathbf{y})$ is a convex set. 
 
 We finally prove that $\rm{S}(\mathbf{y})$ is weakly${ }^*$ compact. Again by the Banach-Alaoglu theorem, it suffices to prove that $\rm{S}(\mathbf{y})$ is bounded and weakly${ }^*$ closed. Since all the elements in $\rm{S}(\mathbf{y})$ have the infimum in \eqref{MNI} as their norms, the set $\rm{S}(\mathbf{y})$ is obviously bounded. Suppose that the sequence $f_m$, $m\in\mathbb{N}$, in $\rm{S}(\mathbf{y})$ converges weakly${}^*$ to $f$. As pointed out earlier, we have that $f\in\mathcal{M}_{\mathbf{y}}$. It follows from the weakly$^{*}$ lower semi-continuity of
	$\|\cdot\|_{\mathcal{B}}$ that  $\|f\|_{\mathcal{B}}\leq\liminf_m\|f_m\|_\mathcal{B}=\|f_0\|_{\mathcal{B}}$. Note that $\|f\|_{\mathcal{B}}=\|f_0\|_{\mathcal{B}}$ otherwise it will contradict to the fact that $f_0$ is a solution. Consequently, we conclude that $f\in\rm{S}(\mathbf{y})$ and thus $\rm{S}(\mathbf{y})$ is weakly${ }^*$ closed. %
 \end{proof}
 
We next express the solution set $\rm{S}(\mathbf{y})$ by using the elements in the set $\mathcal{V}$ defined by \eqref{V_span}. For any $f\in\mathcal{B}$, we introduce a subset of $\mathcal{V}$ by
\begin{equation}\label{def: J(hat f)}
    \mathcal{J}( f):=\left\{\nu\in\mathcal{V}: f\in\|\nu\|_{\mathcal{B}_*}\partial\|\cdot\|_{\mathcal{B}_*}(\nu)\right\}.
\end{equation}
Theorem 12 in \cite{wang2021representer} guarantees that for any solution $\hat f$ of problem \eqref{MNI}, the set $\mathcal{J}(\hat f)$ is nonempty. The next lemma shows that the sets defined by \eqref{def: J(hat f)} associated with different solutions of problem \eqref{MNI} are exactly the same. 

\begin{lemma}\label{prop: duality mappings are same}
If $\mathcal{B}$ is a Banach space having a pre-dual space $\mathcal{B}_{*}$ and $\nu_{j} \in \mathcal{B}_{*}$, $j \in \mathbb{N}_{n}$, are linearly independent, then for any two solutions $\hat f,\hat g$ of the MNI problem \eqref{MNI} with $\mathbf{y} \in \mathbb{R}^{n}\backslash\{\mathbf{0}\}$, there holds $ \mathcal{J}(\hat f)=\mathcal{J}(\hat g).   $
\end{lemma}

\begin{proof}
    It suffices to show that $\mathcal{J}(\hat f)\subseteq\mathcal{J}(\hat g)$. Suppose that $\mu\in\mathcal{J}(\hat f)$. That is,  $\mu\in\mathcal{V}$ and $\hat f\in\|\mu\|_{\mathcal{B}_*}\partial\|\cdot\|_{\mathcal{B}_*}(\mu)$. It is clear that $\mu\neq 0$ since $\hat f\neq 0$. By equation \eqref{subdifferential = norming functional} with $\mathcal{B}$ being replaced by $\mathcal{B}_*$ and equation \eqref{natural-map-predual}, we have that
    \begin{equation}\label{hat f B norm equals to u B_*}
        \|\hat f\|_{\mathcal{B}}=\|\mu\|_{\mathcal{B}_*}\quad\mathrm{and}\quad\langle\mu, \hat f \rangle_{\mathcal{B}}=\|\hat f\|_{\mathcal{B}}\|\mu\|_{\mathcal{B}_*}.
    \end{equation}
    Since both $\hat f$ and $\hat g$ are solutions of the MNI problem \eqref{MNI} with $\mathbf{y}$, it holds that $\|\hat f\|_{\mathcal{B}}=\|\hat g\|_{\mathcal{B}}$ and 
    $\langle\nu_j, \hat f\rangle_{\mathcal{B}}=\langle\nu_j, \hat g\rangle_{\mathcal{B}}, \ \mbox{for all}\ j\in\mathbb{N}_n,$
    which together with $\mu\in\mathcal{V}$ lead to 
    \begin{equation}\label{relation hat f hat g}
     \|\hat f\|_{\mathcal{B}}=\|\hat g\|_{\mathcal{B}}\quad \mbox{and}\quad \langle\mu, \hat f\rangle_{\mathcal{B}}=\langle\mu, \hat g\rangle_{\mathcal{B}}.
     \end{equation}
     Combining equations \eqref{hat f B norm equals to u B_*} with \eqref{relation hat f hat g}, we obtain that 
    $\|\hat g\|_{\mathcal{B}}=\|\mu\|_{\mathcal{B}_*}$ and $ \langle\mu, \hat g\rangle_{\mathcal{B}}=\|\hat g\|_{\mathcal{B}}\|\mu\|_{\mathcal{B}_*}.$
    This implies $\mu\in\mathcal{J}(\hat g)$. Consequently, we conclude that $\mathcal{J}(\hat f)\subseteq\mathcal{J}(\hat g)$.
\end{proof}

As we have seen in Lemma \ref{prop: duality mappings are same}, the set $\mathcal{J}(\hat f)$ is independent of the choice of the solution $\hat f$. This result allows us to express the solution set $\rm{S}(\mathbf{y})$ by using any fixed $\hat\nu\in\mathcal{J}(\hat f)$. 

\begin{proposition}\label{prop: representation_solution_set}
Suppose that $\mathcal{B}$ is a Banach space having a pre-dual space $\mathcal{B}_{*}$, $\nu_{j} \in \mathcal{B}_{*}$, $j \in \mathbb{N}_{n}$, are linearly independent and $\mathbf{y}\in \mathbb{R}^{n}\backslash\{\mathbf{0}\}$. If $\mathcal{V}$ and $\mathcal{M}_{\mathbf{y}}$ are defined by \eqref{V_span} and \eqref{hyperplane My}, respectively, and $\hat\nu\in\mathcal{V}$ satisfies \eqref{Non-empty-set}, then 
\begin{equation}\label{represent SMNI cap 2}
\rm{S}(\mathbf{y})=(\|\hat\nu\|_{\mathcal{B}_*}\partial\|\cdot\|_{\mathcal{B}_*}(\hat\nu))\cap\mathcal{M}_{\mathbf{y}}.
\end{equation}  
\end{proposition}
\begin{proof}
  We first assume that $\hat f\in(\|\hat\nu\|_{\mathcal{B}_*}\partial\|\cdot\|_{\mathcal{B}_*}(\hat\nu))\cap\mathcal{M}_{\mathbf{y}}$. That is, $\hat{f} \in \mathcal{M}_{\mathbf{y}}$ and there exists $\hat{\nu}\in\mathcal{V}$ satisfying the inclusion relation
  \eqref{f hat in dualality mapping}. Theorem 12 in \cite{wang2021representer} ensures that $\hat{f}\in\rm{S}(\mathbf{y})$. Hence, we have that $(\|\hat\nu\|_{\mathcal{B}_*}\partial\|\cdot\|_{\mathcal{B}_*}(\hat\nu))\cap\mathcal{M}_{\mathbf{y}}\subseteq\rm{S}(\mathbf{y}).$ Conversely, assume that $\hat{f}\in\rm{S}(\mathbf{y})$. We choose $\hat g\in(\|\hat\nu\|_{\mathcal{B}_*}\partial\|\cdot\|_{\mathcal{B}_*}(\hat\nu))\cap\mathcal{M}_{\mathbf{y}}$. Similar arguments show that $\hat g$ is also a solution of problem \eqref{MNI} with $\mathbf{y}$ and $\hat\nu\in\mathcal{J}(\hat g).$ With the help of Lemma \ref{prop: duality mappings are same}, we obtain that $\hat\nu\in\mathcal{J}(\hat f).$ That is, $\hat f\in\|\hat\nu\|_{\mathcal{B}_*}\partial\|\cdot\|_{\mathcal{B}_*}(\hat\nu)$, which together with $\hat f\in\mathcal{M}_{\mathbf{y}}$ leads further to $\hat f\in(\|\hat\nu\|_{\mathcal{B}_*}\partial\|\cdot\|_{\mathcal{B}_*}(\hat\nu))\cap\mathcal{M}_{\mathbf{y}}$. We thus get that $\rm{S}(\mathbf{y})\subseteq(\|\hat\nu\|_{\mathcal{B}_*}\partial\|\cdot\|_{\mathcal{B}_*}(\hat\nu))\cap\mathcal{M}_{\mathbf{y}},$ which completes the proof of the desired result.
\end{proof}

We will characterize the extreme points of $\rm{S}(\mathbf{y})$ by employing a well-known result about the extreme points. \cite{dubins1962extreme}  concerns a characterization of the extreme points of a subset, which is the intersection of
a bounded, closed and convex subset with a finite number of  hyperplanes. 
Specifically, let $\mathcal{X}$ be a topological vector space over the field of real numbers. Suppose that $A$ is a bounded, closed and convex subset of $\mathcal{X}$. \cite{dubins1962extreme} proved that every extreme point of the intersection of $A$ with $n$ hyperplanes is a convex combination of at most $n+1$ extreme points of $A$. %
Below, we give a technical lemma.

\begin{lemma}\label{linear_dependent}
    Suppose that $\mathcal{B}$ is a Banach space having a pre-dual space $\mathcal{B}_{*}$ and $\nu_{j} \in \mathcal{B}_{*}$, $j \in \mathbb{N}_{n}$, are linearly independent. Let $\mathcal{L}$ and $\mathcal{V}$ be defined by \eqref{operator L} and \eqref{hyperplane My}, respectively. If $\nu\in \mathcal{V}\backslash\{0\}$ and $w_j\in\partial\|\cdot\|_{\mathcal{B}_*}(\nu)$, $j\in\mathbb{N}_{n+1}$, then there exist $a_j\in\mathbb{R}$, $j\in\mathbb{N}_{n+1}$, not all zero, such that  
    $\sum_{j\in\mathbb{N}_{n+1}}a_j\mathcal{L}(w_{j})=\mathbf{0}$ and $\sum_{j\in\mathbb{N}_{n+1}}a_j=0.$
\end{lemma}
\begin{proof}
Assume that  $\nu:=\sum_{i\in\mathbb{N}_n}c_i\nu_i$ and set $\mathbf{c}:=[c_i:i\in\mathbb{N}_n]\in\mathbb{R}^n$.  By definition \eqref{operator L} of the operator $\mathcal{L}$, we have for each $j\in\mathbb{N}_{n+1}$ that
$$
\left\langle\mathcal{L}(w_j),\mathbf{c}\right\rangle_{\mathbb{R}^n}=\left\langle \sum_{i\in\mathbb{N}_n} c_iv_i,w_j\right\rangle_{\mathcal{B}}=\left\langle \nu,w_j\right\rangle_{\mathcal{B}},
$$ 
which together with  $w_j\in\partial\|\cdot\|_{\mathcal{B}_*}(\nu)$ further leads to 
$\left\langle\mathcal{L}(w_j),\mathbf{c}\right\rangle_{\mathbb{R}^n}=\|\nu\|_{\mathcal{B}_*}.$ It is clear that $\mathcal{L}(w_j)$, $j\in\mathbb{N}_{n+1}$, as $n+1$ elements in $\mathbb{R}^n$, are linearly dependent. Hence, there exist $a_j\in\mathbb{R}$, $j\in\mathbb{N}_{n+1}$, not all zero, such that $\sum_{j\in\mathbb{N}_{n+1}}a_j\mathcal{L}(w_{j})=\mathbf{0}$. By equation $\left\langle\mathcal{L}(w_j),\mathbf{c}\right\rangle_{\mathbb{R}^n}=\|\nu\|_{\mathcal{B}_*}$, we get that 
$$
0=\left\langle\sum_{j\in\mathbb{N}_{n+1}}a_j\mathcal{L}(w_{j}),\mathbf{c}\right\rangle_{\mathbb{R}^n}=\|\nu\|_{\mathcal{B}_*}\sum_{j\in\mathbb{N}_{n+1}}a_j,
$$ 
which together with $\nu\neq 0$ results that $\sum_{j\in\mathbb{N}_{n+1}}a_j=0$.     
\end{proof}

We are ready to provide a complete proof for Proposition \ref{theorem: representer for MNI}. 
\vspace{3mm}

\begin{proof} [of Proposition \ref{theorem: representer for MNI}]
    With the help of Proposition \ref{prop: representation_solution_set}, we represent the solution set $\rm{S}(\mathbf{y})$ as in equation \eqref{represent SMNI cap 2}. It follows from equation \eqref{subdifferential = norming functional} and $\hat\nu\neq0$ that the subset $\|\hat\nu\|_{\mathcal{B}_*}\partial\|\cdot\|_{\mathcal{B}_*}(\hat\nu)$ is bounded and thus weakly$^{*}$ bounded. By the definition of subdifferential, we get that  $\|\hat\nu\|_{\mathcal{B}_*}\partial\|\cdot\|_{\mathcal{B}_*}(\hat\nu)$ is also convex and weakly${}^*$ closed. In addition, it is clear that $\mathcal{M}_{\mathbf{y}}$ is the intersection of $n$ hyperplanes. Consequently, the solution set $\rm{S}(\mathbf{y})$ may be seen as the intersection between a subset of $\mathcal{B}$, which is weakly$^{*}$ bounded, weakly$^{*}$ closed and convex, and $n$ hyperplanes. Then \cite{dubins1962extreme} ensures that any extreme point $\hat f$ of $\rm{S}({\mathbf{y}})$ has the form 
    $
    \hat f=\sum_{j\in\mathbb{N}_{n+1}} \gamma^{'}_ju^{'}_j,
    $
    with $\gamma^{'}_j\geq0$ satisfying $\sum_{j\in\mathbb{N}_{n+1}}\gamma^{'}_j=1$ and $u^{'}_j\in\mathrm{ext}\left(\|\hat\nu\|_{\mathcal{B}_*}\partial\|\cdot\|_{\mathcal{B}_*}(\hat\nu)\right)$. By setting $\gamma^{''}_j:=\|\hat\nu\|_{\mathcal{B}_*}\gamma^{'}_j$, $u_j^{''}:=u^{'}_j/\|\hat\nu\|_{\mathcal{B}_*}$, $j\in\mathbb{N}_{n+1}$, and noting that 
    $\mathrm{ext}\left(\|\hat\nu\|_{\mathcal{B}_*}\partial\|\cdot\|_{\mathcal{B}_*}(\hat\nu)\right)=\|\hat\nu\|_{\mathcal{B}_*}\mathrm{ext}\left(\partial\|\cdot\|_{\mathcal{B}_*}(\hat\nu)\right),$
    we rewrite $\hat f$ as in
    \begin{equation}\label{expansion of hat f}
        \hat f =\sum\limits_{j\in\mathbb{N}_{n+1}} \gamma^{''}_ju^{''}_j,
    \end{equation}
    with $\gamma^{''}_j\geq0$ satisfying $\sum_{j\in\mathbb{N}_{n+1}}\gamma^{''}_j=\|\hat\nu\|_{\mathcal{B}_*}$ and $u^{''}_j\in\mathrm{ext}\left(\partial\|\cdot\|_{\mathcal{B}_*}(\hat\nu)\right)$. 

    We will represent $\hat f$ with the form  \eqref{expansion of hat f} as in \eqref{eq: expansion of p}. Note that if there exists $j_0\in\mathbb{N}_{n+1}$ such that $\gamma_{j_0}''=0$, then representation \eqref{expansion of hat f} reduces to the desired result \eqref{eq: expansion of p}. Hence, it remains to consider the case that $\gamma_j''>0$ for all $j\in\mathbb{N}_{n+1}$. We will show that $u_j''$, $j\in\mathbb{N}_{n+1}$ are linearly dependent. By Lemma \ref{linear_dependent} and noting that $\hat\nu\in \mathcal{V}\backslash\{0\}$ and $u^{''}_{j}\in\partial\|\cdot\|_{\mathcal{B}_*}(\hat\nu)$, $j\in\mathbb{N}_{n+1}$, there exist $a_j\in\mathbb{R}$, $j\in\mathbb{N}_{n+1}$, not all zero, such that  
    \begin{equation}\label{linear combination of alpha j L(uj)}
        \sum\limits_{j\in\mathbb{N}_{n+1}}a_j\mathcal{L}(u_{j}^{''})=\mathbf{0}\ \, \mbox{and}\ \, \sum\limits_{j\in\mathbb{N}_{n+1}}a_j=0.
    \end{equation} 
    By setting $\alpha:=\frac{\min_{j\in\mathbb{N}_{n+1}} \gamma_{j}''}{\max_{j\in\mathbb{N}_{n+1}} |a_j|}$, we introduce two elements $f_1,f_2$ in $\mathcal{B}$ by       \begin{equation*}\label{u_epsilon}
        f_1:=\hat f+\frac{\alpha}{2} \sum\limits_{j\in\mathbb{N}_{n+1}}a_{j}u^{''}_{j},\ \  f_2:=\hat f-\frac{\alpha}{2} \sum\limits_{j\in\mathbb{N}_{n+1}}a_{j}u^{''}_{j}.
    \end{equation*}
   It follows from the first equation in \eqref{linear combination of alpha j L(uj)} and $\mathcal{L}(\hat f)=\mathbf{y}$ that  $\mathcal{L}(f_1)=\mathbf{y}$. That is, $f_1\in\mathcal{M}_{\mathbf{y}}$. Substituting equation \eqref{expansion of hat f} into the representation of $f_1$, we obtain that 
    \begin{equation}\label{expansion of u_epsilon}
        f_1=\sum\limits_{j\in\mathbb{N}_{n+1}}\left(\gamma^{''}_{j}+\frac{\alpha a_j}{2}\right)u^{''}_{j}.
    \end{equation}
    According to the definition of $\alpha$, the coefficient of each $u_j''$ appearing in \eqref{expansion of u_epsilon} is positive. Moreover, combining $\sum_{j\in\mathbb{N}_{n+1}}\gamma^{''}_j=\|\hat\nu\|_{\mathcal{B}_*}$ with the second equation in \eqref{linear combination of alpha j L(uj)}, we get that 
    \begin{equation*}
    \sum\limits_{j\in\mathbb{N}_{n+1}}\left(\gamma^{''}_{j}+\frac{\alpha a_j}{2}\right)=\sum_{j\in\mathbb{N}_{n+1}}\gamma^{''}_j+\frac{\alpha}{2}\sum\limits_{j\in\mathbb{N}_{n+1}}a_j=\|\hat\nu\|_{\mathcal{B}_*}.
    \end{equation*}
    We thus conclude that  $f_1/\|\hat\nu\|_{\mathcal{B}_*}$ is a convex combination of $u^{''}_j$, $j\in\mathbb{N}_{n+1}$.
    Recall that $u^{''}_j\in\mathrm{ext}\left(\partial\|\cdot\|_{\mathcal{B}_*}(\hat\nu)\right)$ and hence $f_1/\|\hat\nu\|_{\mathcal{B}_*}\in\partial\|\cdot\|_{\mathcal{B}_*}(\hat\nu)$, that is, 
    $f_1\in \|\hat\nu\|_{\mathcal{B}_*}\partial\|\cdot\|_{\mathcal{B}_*}(\hat\nu)$. Above all, we obtain that $f_1\in(\|\hat\nu\|_{\mathcal{B}_*}\partial\|\cdot\|_{\mathcal{B}_*}(\hat\nu))\cap\mathcal{M}_{\mathbf{y}}$, which guaranteed by Proposition \ref{prop: representation_solution_set} is equivalent to $f_1\in\rm{S}(\mathbf{y})$. By similar arguments we also get that $f_2\in\rm{S}(\mathbf{y})$. It is clear that 
    $\hat f=(f_1+f_2)/2,$ 
    which together with the assumption that $\hat f\in\mathrm{ext}\left(\rm{S}(\mathbf{y})\right)$ further
    leads to $\hat f=f_1=f_2$. Substituting the above relation into the representations of $f_1$ and $f_2$ with noting that $\alpha>0$, we have that
    \begin{equation}\label{proof: sum of ajuj''}
    \sum_{j\in\mathbb{N}_{n+1}}a_{j}u^{''}_{j}=0.     
    \end{equation}
    Since $a_j\in\mathbb{R}$, $j\in\mathbb{N}_{n+1}$ are not all zero, \eqref{proof: sum of ajuj''} implies that $u^{''}_{j}$, $j\in\mathbb{N}_{n+1},$ are linearly dependent. 

    Without loss of generality, we assume that $a_{n+1}\neq0$. It follows from equation \eqref{proof: sum of ajuj''} that 
    \begin{equation}\label{proof: un+1''}
    u_{n+1}''=-\sum_{j\in\mathbb{N}_{n}}\frac{a_j}{a_{n+1}}u^{''}_{j}.
    \end{equation}     
    By substituting relation \eqref{proof: un+1''} into \eqref{expansion of hat f}, we obtain that 
    \begin{equation}\label{proof: hat f is expansion of uj''}
        \hat f=\sum\limits_{j\in\mathbb{N}_n}\left(\gamma_j''-\frac{\gamma_{n+1}'' }{a_{n+1}}a_j\right)u_j''
    \end{equation}
    For each $j\in\mathbb{N}_n$, by letting $\gamma_j:=\gamma_j''-\frac{\gamma_{n+1}'' }{a_{n+1}}a_j$ and $u_j:=u_j''$, we can rewrite \eqref{proof: hat f is expansion of uj''} as \eqref{eq: expansion of p}. In addition, due to $\sum_{j\in\mathbb{N}_{n+1}}\gamma^{''}_j=\|\hat\nu\|_{\mathcal{B}_*}$ and the second equation of \eqref{linear combination of alpha j L(uj)}, we obtain that $\sum_{j\in\mathbb{N}_n}\gamma_j=\|\hat\nu\|_{\mathcal{B}_*}$. This completes the proof. 
\end{proof}

The next proposition shows that for any $\nu\in\mathcal{B}_*\backslash\{0\}$, the set $\mathrm{ext}\left(\partial\|\cdot\|_{\mathcal{B}_*}(\nu)\right)$ is smaller than the set of extreme points of the closed unit ball $B_0$. That is to say, Proposition \ref{theorem: representer for MNI} provides an even more precise
characterization for $u_j$, $j\in\mathbb{N}_n$, appearing in representation \eqref{eq: expansion of p}.

\begin{proposition}\label{prop: ext is smaller}
If Banach space $\mathcal{B}$ has a pre-dual space $\mathcal{B}_{*}$, then for any $\nu\in\mathcal{B}_*\backslash\{0\}$, there holds 
$\mathrm{ext}\left(\partial\|\cdot\|_{\mathcal{B}_*}(\nu)\right)\subset\mathrm{ext}\left(B_0\right).$
\end{proposition}
\begin{proof}
We assume that $\nu\in\mathcal{B}_*\backslash\{0\}$ and  $f\in\mathrm{ext}\left(\partial\|\cdot\|_{\mathcal{B}_*}(\nu)\right)$. It is sufficient to present $f\in\mathrm{ext}(B_0)$. By equation \eqref{subdifferential = norming functional}, we get that $\|f\|_{\mathcal{B}}=1$ and thus $f\in B_0$. For any $f_1,f_2\in B_0$ satisfying $f=(f_1+f_2)/2$, we shall prove that $f_1=f_2=f$. We first show that $\|f_1\|_{\mathcal{B}}=\|f_2\|_{\mathcal{B}}=1$. Otherwise, without loss of generality, we suppose that $\|f_1\|_{\mathcal{B}}<1$. Then there holds
\begin{equation*}
1=\|f\|_{\mathcal{B}}=\left\|\frac{f_1+f_2}{2}\right\|_{\mathcal{B}}\leq \frac{\|f_1\|_{\mathcal{B}}+\|f_2\|_{\mathcal{B}}}{2}<1,
\end{equation*}
which leads to a contradiction. Hence, $\|f_1\|_{\mathcal{B}}=\|f_2\|_{\mathcal{B}}=1$. We next prove that $\langle f_i,\nu\rangle_{\mathcal{B}_*}=\|\nu\|_{\mathcal{B}_*}$, for $i=1,2$. If the claim is not true, without loss of generality, we suppose that $\langle f_1,\nu\rangle_{\mathcal{B}_*}<\|\nu\|_{\mathcal{B}_*}$.  It follows that 
\begin{equation*}
    \|\nu\|_{\mathcal{B}_*}=\langle f,\nu\rangle_{\mathcal{B}_*}=\frac{\langle f_1,\nu\rangle_{\mathcal{B}_*}+\langle f_2,\nu\rangle_{\mathcal{B}_*}}{2}<\|\nu\|_{\mathcal{B}_*},
\end{equation*}
which is a contradiction as well. Thus, $\langle f_i,\nu\rangle_{\mathcal{B}_*}=\|\nu\|_{\mathcal{B}_*}$, for $i=1,2$. We then conclude by equation \eqref{subdifferential = norming functional} that $f_1,f_2\in\partial\|\cdot\|_{\mathcal{B}_*}(\nu)$. This combined with $f\in\mathrm{ext}\left(\partial\|\cdot\|_{\mathcal{B}_*}(\nu)\right)$ and the definition of extreme points leads to $f_1=f_2=f$. Again using the definition of extreme point, we obtain that $f\in\mathrm{ext}(B_0)$, which completes the proof.
\end{proof}

\section*{Appendix B. A Dual Problem of MNI in Banach Spaces}
In this appendix, we formulate a dual problem, a finite dimensional optimization problem, of the MNI problem \eqref{MNI} in a general Banach space having a pre-dual space. We then show that the element $\hat\nu\in\mathcal{V}$ appearing in Proposition \ref{theorem: representer for MNI} can be obtained by solving the resulting dual problem. Throughout this appendix, we suppose that $\mathcal{B}$ is a Banach space having the pre-dual space $\mathcal{B}_{*}$ and $\nu_{j} \in \mathcal{B}_{*}$, $j \in \mathbb{N}_{n}$, are linearly independent.

We establish the dual problem of problem \eqref{MNI} via a functional analysis approach. For this purpose, we recall the notion of the quotient space. 
If $\mathcal{M}$ is a closed subspace
of a Banach space $\mathcal{B}$, then the quotient space $\mathcal{B}/\mathcal{M}$ is defined to be the collection of cosets $f+\mathcal{M}$, for all
$f\in\mathcal{B}$. The quotient space is a Banach space when endowed with the norm
$$
\left\|f+\mathcal{M}\right\|_{\mathcal{B}/\mathcal{M}}:=\inf\left\{\|f+g\|_{\mathcal{B}}:g\in\mathcal{M}\right\}.
$$
It is known (\cite{conway2019course}) that $\mathcal{M}^*$ is isometrically isomorphic to $\mathcal{B}^*/\mathcal{M}^\bot$.

We now transform the MNI problem \eqref{MNI} into an equivalent dual
problem.

\begin{proposition}\label{prop: original problem=dual problem}
For $\mathbf{y}:=[y_j:j\in\mathbb{N}_n] \in \mathbb{R}^{n}\backslash\{\mathbf{0}\}$, let $\mathcal{M}_{\mathbf{y}}$ be defined by \eqref{hyperplane My}. 
Then 
\begin{equation}\label{eq: dual norm equi}
    \inf\left\{\|f\|_{\mathcal{B}}:f\in\mathcal{M}_{\mathbf{y}}\right\}=\sup\left\{ \sum_{j\in\mathbb{N}_n}c_jy_j:\left\|\sum_{j\in\mathbb{N}_n} c_j\nu_j\right\|_{\mathcal{B}_*}=1\right\}.
\end{equation}
\end{proposition}

\begin{proof}
By setting $\mathcal{M}_0$ to be $\mathcal{M}_{\mathbf{y}}$ with $\mathbf{y}=\mathbf{0}$, we represent $\mathcal{M}_{\mathbf{y}}$ as a translation of $\mathcal{M}_0$, that is,  $\mathcal{M}_\mathbf{y}:=\mathcal{M}_0+f_0$ for some $f_0\in\mathcal{M}_{\mathbf{y}}$. Then the MNI problem \eqref{MNI} may be rewritten as 
$$
\inf\left\{\|f\|_{\mathcal{B}}:f\in\mathcal{M}_{\mathbf{y}}\right\}=\inf\left\{\|f_0+g\|_{\mathcal{B}}:g\in\mathcal{M}_{0}\right\},
$$
which further leads to 
\begin{equation}\label{proof: dual equi 1}
   \inf\left\{\|f\|_{\mathcal{B}}:f\in\mathcal{M}_{\mathbf{y}}\right\}=\left\|f_0+\mathcal{M}_0\right\|_{{\mathcal{B}}/\mathcal{M}_0}.
\end{equation}
By the isometric isomorphism between $(\mathcal{B}_*)^*/\mathcal{V}^\bot$ and $\mathcal{V}^*$, with noting that  $\mathcal{B}=(\mathcal{B}_*)^*$ and $\mathcal{M}_0=\mathcal{V}^\bot$, we have that 
\begin{equation}\label{proof: dual equi 2}
    \left\|f_0+\mathcal{M}_0\right\|_{{\mathcal{B}}/\mathcal{M}_0}=\sup\left\{\left\langle \sum_{j\in\mathbb{N}_n}c_j\nu_j, f_0\right\rangle_{\mathcal{B}}: \left\|\sum_{j\in\mathbb{N}_n} c_j \nu_j\right\|_{\mathcal{B}_*}=1\right\}.
\end{equation}
Substituting $\langle\nu_j, f_0\rangle_{\mathcal{B}}=y_j$, $j\in\mathbb{N}_n$, into the right-hand-side of equation \eqref{proof: dual equi 2}, we get that 
$$
\left\|f_0+\mathcal{M}_0\right\|_{{\mathcal{B}}/\mathcal{M}_0}=\sup\left\{ \sum_{j\in\mathbb{N}_n}c_jy_j:\left\|\sum_{j\in\mathbb{N}_n} c_j\nu_j\right\|_{\mathcal{B}_*}=1\right\}.
$$
Again substituting the above equation  into \eqref{proof: dual equi 1}, we get the desired equation \eqref{eq: dual norm equi}. 
\end{proof}

As a finite dimensional optimization problem, the dual problem
\begin{equation}\label{dual problem}
    \sup\left\{ \sum_{j\in\mathbb{N}_n}c_jy_j:\left\|\sum_{j\in\mathbb{N}_n} c_j\nu_j\right\|_{\mathcal{B}_*}=1\right\},
\end{equation}
shares the same optimal value with the MNI problem \eqref{MNI}. We remark that such a dual problem  was considered in \cite{cheng2021minimum} for the MNI problem with $\mathcal{B}:=\ell_1(\mathbb{N})$ and in \cite{ChengWangXu2023} for a class of regularization problems.  

The next result concerns how to obtain an element $\hat\nu\in\mathcal{V}$ satisfying \eqref{Non-empty-set} once we have a solution of the dual problem \eqref{dual problem} at hand. 

\begin{proposition}\label{prop: solution of dual problem gives solution of original problem}
For $\mathbf{y}:=[y_j:j\in\mathbb{N}_n] \in \mathbb{R}^{n}\backslash\{\mathbf{0}\}$, let $\mathcal{M}_{\mathbf{y}}$ be defined by \eqref{hyperplane My}. If $m_0$ is the infimum of the MNI problem \eqref{MNI} with $\mathbf{y}$ and  $\hat{\mathbf{c}}:=[\hat{c}_1,\hat{c}_2,\ldots,\hat{c}_n]\in\mathbb{R}^n$ is a solution of the dual problem \eqref{dual problem} with $\mathbf{y}$, then $\hat{\nu}:=m_0\sum_{j\in\mathbb{N}_n}\hat{c}_j\nu_j$ satisfies \eqref{Non-empty-set}.
\end{proposition}
\begin{proof}
Noting that $\rm{S}(\mathbf{y})$ is nonempty, we choose $\hat f\in\rm{S}(\mathbf{y})$ and proceed to prove that $\hat f\in\|\hat{\nu}\|_{\mathcal{B}_*}\partial\|\cdot\|_{\mathcal{B}_*}(\hat{\nu})\cap\mathcal{M}_{\mathbf{y}}$. Clearly, $\hat f\in\mathcal{M}_{\mathbf{y}}$. It suffices to show that $\hat f\in\|\hat{\nu}\|_{\mathcal{B}_*}\partial\|\cdot\|_{\mathcal{B}_*}(\hat{\nu})$. Since $\hat{\mathbf{c}}$ is a solution of problem \eqref{dual problem}, we get that $    \|\sum_{j\in\mathbb{N}_n} \hat{c}_j\nu_j\|_{\mathcal{B}_*}=1
    $. This together with the definition of $\hat{\nu}$ leads to $\|\hat{\nu}\|_{\mathcal{B}_*}=m_0$. By noting that $\|\hat f\|_{\mathcal{B}}=m_0$, we obtain that $\left\|\frac{\hat{f}}{\|\hat{\nu}\|_{\mathcal{B}_*}}\right\|_{\mathcal{B}}=1$. It follows that     
$$
\left\langle\frac{\hat{f}}{\|\hat{\nu}\|_{\mathcal{B}_*}},\hat{\nu}\right\rangle_{\mathcal{B}_*}=\sum_{j\in\mathbb{N}_n} \hat{c}_j\langle\hat{f},\nu_j\rangle_{\mathcal{B}_*}.
$$
Substituting $\langle\hat{f},\nu_j\rangle_{\mathcal{B}_*}=y_j$, $j\in\mathbb{N}_n$, into the above equation, we have that 
$$
\left\langle\frac{\hat{f}}{\|\hat{\nu}\|_{\mathcal{B}_*}},\hat{\nu}\right\rangle_{\mathcal{B}_*}=\sum_{j\in\mathbb{N}_n}\hat{c}_j y_j.
$$ 
Noting that $\hat{\mathbf{c}}$ is a solution of problem \eqref{dual problem}, the above equation,   guaranteed by Proposition \ref{prop: original problem=dual problem}, yields that   $\langle\hat{f}/\|\hat{\nu}\|_{\mathcal{B}_*},\hat{\nu}\rangle_{\mathcal{B}_*}=m_0$. This together with $\|\hat{\nu}\|_{\mathcal{B}_*}=m_0$ leads directly to $\langle\hat{f}/\|\hat{\nu}\|_{\mathcal{B}_*},\hat{\nu}\rangle_{\mathcal{B}_*}=\|\hat{\nu}\|_{\mathcal{B}_*}$. Consequently, we conclude that $\hat{f}/\|\hat{\nu}\|_{\mathcal{B}_*}\in\partial\|\cdot\|_{\mathcal{B}_*}(\hat{\nu})$, that is, $\hat f\in\|\hat{\nu}\|_{\mathcal{B}_*}\partial\|\cdot\|_{\mathcal{B}_*}(\hat{\nu}).$ This completes the proof of this proposition.
\end{proof}

Proposition \ref{prop: solution of dual problem gives solution of original problem} ensures
that the element $\hat\nu\in\mathcal{V}$ appearing in Proposition \ref{theorem: representer for MNI} can be obtained by solving the dual problem \eqref{dual problem}.

\vskip 0.2in
\bibliography{ref}

\end{document}